\documentclass[twoside,12pt]{article}
\DeclareMathSizes{20}{30}{16}{12}
\usepackage{float}
\usepackage{amsmath}
\usepackage{mathtools}
\usepackage{mathrsfs}
\usepackage{stmaryrd}
\usepackage{bbm}
\usepackage{yfonts}
\usepackage{amsfonts}
\usepackage{tcolorbox}
\usepackage{tikz-cd}
\usepackage[utf8]{inputenc}

\usepackage{caption}
\usepackage{indentfirst}
\usepackage[Symbol]{upgreek}
\usepackage{enumerate}   
\usepackage{upref}

\usepackage{graphicx}
\usepackage[margin=0.8 in]{geometry}
\usepackage{fancyhdr}
\usepackage{hyperref}
\hypersetup{
    colorlinks,
    citecolor=black,
    filecolor=black,
    linkcolor=black,
    urlcolor=black
}
\newcommand{\N}{\mathbb{N}}
\newcommand{\Z}{\mathbb{Z}}

\newcommand{\ra}{\rightarrow}

\newcommand{\calI}{\mathcal{I}}

\newcommand{\calE}{\mathcal{E}}

\newcommand{\calT}{\mathcal{T}}
\newcommand{\calD}{\mathcal{D}}

\newcommand{\calR}{\mathcal{R}}

\newcommand{\calM}{\mathcal{M}}

\newcommand{\calP}{\mathcal{P}}

\newcommand{\op}{\operatorname}
\newcommand{\w}{\widehat}

\newcommand{\Ox}{\mathcal{O}}
\newcommand{\F}{\mathcal{F}}

\newcommand{\ov}{\overline}

\newcommand{\frakX}{\mathfrak{X}}

\newcommand{\frakU}{\mathfrak{U}}

\newcommand{\frakS}{\mathcal{S}}
\newcommand{\frakY}{\mathfrak{Y}}
\newcommand{\frakT}{\mathfrak{T}}
\newcommand{\frakZ}{\mathfrak{Z}}

\newcommand{\frakG}{\mathfrak{G}}

\newcommand{\Sp}{\operatorname{Spec}}

\long\def\/*#1*/{}

\usepackage{unicode-math}
\usepackage{fontspec}

\usepackage{sectsty}
\sectionfont{\centering}

\usepackage[toc,page]{appendix}
\usepackage{amsthm}

\newtheorem*{proof*}{Proof}

\newtheorem{theorem}[subsection]{Theorem}
\newtheorem{theorem*}{Theorem}
\newtheorem{proposition}[subsection]{Proposition}
\newtheorem{corollaire}[subsection]{Corollary}
\newtheorem{lemma}[subsection]{Lemma}

\theoremstyle{definition}

\newtheorem{definition}[subsection]{Definition}
\newtheorem{remark}[subsection]{Remarks}

\newtheorem{exmp}[subsection]{Examples}

\theoremstyle{definition}
\newtheorem{parag}[subsection]{}

\numberwithin{equation}{subsection} 

\makeatletter
\renewcommand{\theequation}{%
  \ifnum\value{subsubsection}>0
    \thesubsubsection.\arabic{equation}%
  \else
    \thesubsection.\arabic{equation}%
  \fi
}
\makeatother

\title{Local Logarithmic Cartier Transform}
\author{Sami Fersi}
\date{}

\AtEndDocument{\bigskip{\footnotesize%
  \textsc{Sami Fersi, Laboratoire Alexander Grothendieck, Institut des Hautes Études Scientifiques, 35 Route de Chartres, 91440 Bures-sur-Yvette, France} \par
  \textit{E-mail address}: \texttt{fersi@ihes.fr} \par
}}
\bibliography{./references/ref}
\begin{document}
\maketitle

\begin{abstract}
This article is the first of three articles whose goal is to generalize the Cartier transform of Ogus and Vologodsky to the logarithmic setting. In this article, we generalize a local version, due to Shiho, of the Cartier transform to log smooth schemes. More precisely, let $k$ be a perfect field of positive characteristic and equip $\op{Spec}k$ with the trivial logarithmic structure. For a log smooth morphism of logarithmic schemes $X \ra S,$ where $S$ is log flat and locally of finite type over $\op{Spec}k,$ we obtain, under the assumption that the exact relative Frobenius lifts over the Witt vectors of $k$ to a morphism between log smooth schemes over $S,$ a fully faithful functor from the category of quasi-coherent modules on the base change $X'=X\times_{S,F_S}S$ of $X$ by the Frobenius $F_S$ of $S,$ equipped with a quasi-nilpotent Higgs field, to the category of quasi-coherent modules on $X$ equipped with a quasi-nilpotent integrable connection. For this, we prove a log flat descent theorem for morphisms, based on previous work by Kato.
\end{abstract}

\tableofcontents

\section{Introduction}

\begin{parag}
In this article and two related articles (\cite{SF2} and \cite{SF3}), our goal is to generalize the \emph{Cartier transform} of Ogus and Vologodsky to the logarithmic case. In the classical smooth case, this transform can be seen as a generalization of two results for schemes in positive characteristic: \emph{Cartier descent} and \emph{Deligne-Illusie decomposition of the De Rham complex}.
For a smooth morphism of schemes $X\ra S$ of positive characteristic $p>0,$ Cartier descent is an equivalence between the category of quasi-coherent $\Ox_X$-modules equipped with an integrable connection with vanishing $p$-curvature, and the category of quasi-coherent $\Ox_{X'}$-modules, where $X'=X\times_{S,F_S}S$ is the base change of $X$ by the absolute Frobenius $F_S$ of $S$ (\cite{Katz} 5.1). Under this equivalence, an $\Ox_{X'}$-module $\calE'$ corresponds to the module $F_{X/S}^*\calE',$ where $F_{X/S}:X\ra X'$ is the relative Frobenius, equipped with the canonical connection
$$
\nabla^{\mathrm{can}}:
\begin{array}[t]{clc}
F_{X/S}^*\calE' & \ra & \left (F_{X/S}^*\calE'\right ) \otimes_{\Ox_X}\Omega^1_{X/S}\\
aF_{X/S}^* x & \mapsto & \left (F_{X/S}^*x\right ) \otimes da,
\end{array}
$$
for local sections $a$ and $x$ of $\Ox_X$ and $\calE'$ respectively.
On the other hand, Deligne and Illusie proved the following: if $k$ is a perfect field of characteristic $p,$ $X\ra S=\Sp k$ is a smooth morphism of schemes and $W_2(k)$ is the ring of Witt vectors $W(k)$ of $k$ modulo $p^2,$ and if the relative Frobenius $F_{X/S}$ lifts to a morphism between smooth schemes over $W_2(k),$ then there exists a quasi-isomorphism
$$
\bigoplus_{i\in \N} \Omega^i_{X'/S}[-i] \xrightarrow{\sim} F_{X/S*}\Omega^{\bullet}_{X/S},
$$
inducing the Cartier isomorphism on cohomologies. If only $X$ lifts to a smooth scheme over $W_2(k),$ then there exists an isomorphism in $D(\Ox_{X'})$
$$
\bigoplus_{0\le i<p} \Omega_{X'/S}^i[-i] \xrightarrow{\sim} F_{X/S*}\tau_{<p} \Omega^{\bullet}_{X/S},
$$
where $\tau_{<p}$ is the truncation functor up to $p-1.$ The \emph{Cartier transform} of Ogus and Vologodsky comes as a generalization of these two results. More precisly, the Cartier transform is an equivalence between the categories of $\Ox_{X'}$-modules equipped with Higgs fields and $\Ox_X$-modules equipped with integrable connections, both satisfying certain nilpotence conditions. In addition, this transform is compatible with the natural cohomologies.

Shiho \cite{Shiho} provided a local construction of the Cartier transform under the assumption that the relative Frobenius $F_{X/S}$ lifts to a morphism of smooth formal schemes over the Witt vectors $W(k)$ of a perfect field $k$ of characteristic $p.$ More precisely, he interpreted modules with Higgs fields and integrable connections in terms of stratifications relative to certain groupoids. Using faithfully flat descent, he then proved the equivalence for modules with stratifications and deduced the one for Higgs fields and connections. The groupoids are defined as PD-envelopes of dilatations of the ideal of the diagonal embedding of the formal scheme. Oyama used similarly constructed groupoids to provide a crystalline-like interpretation of the Cartier transform \cite{Oyama}. As opposed to Shiho's, Oyama's construction is global and doesn't depend on a lifting of $F_{X/S}.$ In this article, we generalize Shiho's local construction to the logarithmic case. In a second article \cite{SF2}, we generalize Oyama's topos-theoretic construction.

While the Deligne-Illusie decomposition of the de Rham complex generalizes to the log smooth case without problems, as was proved by Kato in (\cite{Kat89} 4.12), one encounters a major problem when attempting to generalize Shiho's construction, as it relies heavily on faithfully flat descent applied to the relative Frobenius which is faithfully flat in the smooth setting. This is not the case in the log smooth setting. To bypass this hurdle, we prove that the relative Frobenius is log flat in the log smooth setting and then we prove, in this article and based on some work by Kato, a log flat descent theorem for morphisms. However, unlike the smooth setting where we get an equivalence of categories, our result allows us to get a fully faithful functor. The lack of essential surjectivity results from the lack of flatness of the relative Frobenius in the log smooth setting. We address this issue in an other paper \cite{SF3} by using the indexed algebras of Lorenzon \cite{Lor2000} and constructing an indexed logarithmic Cartier transform.

In this article, we start by generalizing Shiho functor to log smooth logarithmic schemes. When attempting to define groupoids similar to those introduced by Shiho, one major problem we encounter is the \emph{exactification} of the diagonal immersion: decomposing the diagonal immersion into an exact immersion followed by a log étale morphism. This is possible locally or, globally if the logarithmic schemes are equipped with \emph{frames}, a notion introduced by Kato and Saito in \cite{Saito04}. We hence work systematically with framed logarithmic schemes and we can thus construct the two groupoids corresponding to Higgs fields and integrable connections and the functor between them. We show that it is fully faithful using a log flat descent theorem for morphisms, that we prove based on results by Kato.
\end{parag}

\begin{parag}
We now describe our work in this article in more details. We fix a prime number $p.$ For a morphism $X\ra S$ of fs logarithmic schemes of characteristic $p,$ let $X''$ be the base change, in the category of fine logarithmic schemes, of $X$ by the absolute Frobenius morphism $F_S:S\ra S$ of $S.$ Let $F_{X/S}$ be the relative Frobenius morphism of $X$ with respect to $S$ i.e. the unique morphism $X\ra X''$ making the following diagram commutative
\begin{equation}
\begin{tikzcd}
X\ar{dr}{F_{X/S}}\ar[bend right=-30]{drr}{F_X}\ar[swap,bend right=30]{ddr} & & \\
 & X''\ar{r}\ar{d} & X\ar{d} \\
 & S\ar{r}{F_S} & S
\end{tikzcd}
\end{equation}
Recall that $F_{X/S}$ is \emph{weakly inseparable} (\cite{Ogus2018} III 2.4), so $F_{X/S}$ factors uniquely as a composition of a log étale morphism $G:X'\ra X''$ and an inseparable morphism $F:X\ra X'$ (\cite{Ogus2018} IV 3.3.8): 
\begin{equation}
\begin{tikzcd}
X\ar{r}{F}\ar{dr}\ar[bend right=30]{rdd} & X'\ar{d}{G}\ar{dr}{\pi} & \\
 & X''\ar{d}\ar{r} & X\ar{d} \\
 & S \ar{r}{F_S} & S
\end{tikzcd}
\end{equation}
The morphism $F:X\ra X'$ is called the \emph{exact relative Frobenius of $X$ with respect to $S$} (\cite{Ogus2018} IV 3.3.9).
In section 3, we prove that the exact relative Frobenius of a log smooth morphism of fs logarithmic schemes is log flat \eqref{thmlogflat}.
\end{parag}

\begin{parag}
We fix a perfect field $k$ of characteristic $p$ and we denote its ring of Witt vectors by $W(k).$ If a gothic letter $\frakX$ denotes a logarithmic $p$-adic formal scheme over $W(k),$ then the corresponding roman letter $X=\frakX \times_{\op{Spf}W(k)} \op{Spec}k$ will denote its special fiber.
We consider a log smooth morphism $f:\frakX \ra \frakS$ of fs $p$-adic logarithmic formal schemes, flat over $\op{Spf}W(k).$ Denote by $f_1:X \ra S$ its special fiber.
Following Shiho's construction, we suppose that the exact relative Frobenius  $F_1:X\ra X'$ lifts to an $\frakS$-morphism of logarithmic formal schemes $F:\frakX\ra \frakX',$ such that $\frakX'$ is flat over $\op{Spf}W(k).$ The flatness condition implies that there exists a unique morphism $\frac{dF}{p}$ fitting into the commutative diagram
$$
\begin{tikzcd}
F^*\omega^1_{\frakX'/\frakS} \ar{r}{dF} \ar[swap]{dr}{ \frac{dF}{p}} & p\omega^1_{\frakX / \frakS} \\
& \omega^1_{\frakX / \frakS} \ar{u}{\times p}.
\end{tikzcd}
$$
Let $n\ge 0$ be an integer. Let $\calE'$ be an $\Ox_{\frakX'}$-module equipped with a $p^{n+1}$-connection
$$
\nabla':\calE' \ra \calE' \otimes_{\Ox_{\frakX'}} \omega^1_{\frakX' / \frakS},
$$
i.e. an additive morphism satisfying, for local sections $a$ and $x$ of $\Ox_{\frakX'}$ and $\calE'$ respectively, the modified Leibniz rule
$$
\nabla'(ax)=a\nabla'(x)+p^{n+1}x\otimes da.
$$
We obtain an $\Ox_{\frakX}$-module $\calE=F^*\calE'$ equipped with a $p^n$-connection
$$
\nabla:\begin{array}[t]{clc}
\calE & \ra & \calE \otimes_{\Ox_{\frakX}} \omega^1_{\frakX / \frakS} \\
a\otimes x' & \mapsto &  \zeta(a\otimes x')+F^*x'\otimes da, 
\end{array}
$$
where $a$ and $x'$ are local sections of $\Ox_{\frakX}$ and $\calE'$ respectively and $\zeta$ is the composition
$$
\zeta:\calE \xrightarrow{F^*\nabla'} \calE \otimes_{\Ox_{\frakX}}F^*\omega^1_{\frakX' / \frakS} \xrightarrow{\op{Id}_{\calE} \otimes \frac{dF}{p}} \calE \otimes_{\Ox_{\frakX}}\omega^1_{\frakX / \frakS}.
$$
If $\nabla'$ is integrable then so is $\nabla.$
This construction thus yields a functor
$$
p^{n+1}\text{-}\op{MIC}(\frakX'/\frakS) \ra p^n\text{-}\op{MIC}(\frakX/\frakS),
$$
where $p^{n}\text{-}\op{MIC}(\frakX/\frakS)$ is the category of $\Ox_{\frakX}$-modules equipped with integrable $p^{n}$-connections.
In particular, if $n=0$ and if we work modulo $p^l$ for some positive integer $l,$ we have a functor
\begin{equation}\label{intro1}
\Phi_l:p\text{-}\op{MIC}(\frakX_l'/\frakS_l) \ra \op{MIC}(\frakX_l/\frakS_l),
\end{equation}
where $\frakX_l$ (resp $\frakS_l$ )denotes the logarithmic scheme obtained from $\frakX$ (resp $\frakS$) by reduction modulo $p^l.$
To study this functor, we interpret modules equipped with quasi-nilpotent integrable $p^n$-connections as modules with stratifications with respect to certain groupoids. To introduce these groupoids, we need to exactify the diagonal immersion. This is possible using \emph{frames}, a notion introduced by Saito and Kato \cite{Saito04}.
\end{parag}

\begin{parag}
Let $\boldsymbol{\op{L}}$ be the category of fine saturated (fs) logarithmic schemes. For an fs monoid $M,$ we denote by $[M]$ the presheaf
$$
[M]:\begin{array}[t]{clc}
\boldsymbol{\op{L}} & \ra & \boldsymbol{\op{Sets}} \\
T & \mapsto & \op{Hom} \left (M,\Gamma(T,\ov{\calM}_T) \right ).
\end{array}
$$
A frame on an fs logarithmic scheme or a logarithmic formal scheme $T$ is a morphism of presheaves $T \ra [M],$ or equivalently a morphism of monoids $M \ra \Gamma\left (T,\ov{\calM}_T \right ),$ 
that lifts, étale locally on $T,$ to a chart $T \ra \Gamma \left (T,\calM_T \right ).$ An important result we use is a direct generalization for formal schemes of a result by Kato and Saito (\cite{Saito04} 4.2.8): suppose that $\frakX$ and $\frakS$ are fs logarithmic $p$-adic formal schemes equipped with frames $Q\ra \Gamma(\frakX,\ov{\calM}_{\frakX})$ and $P\ra \Gamma(\frakS,\ov{\calM}_{\frakS})$ respectively and that we are given a morphism of monoids $\theta:P \ra Q$ such that $(f,\theta):(\frakX,Q) \ra (\frakS,P)$ is a morphism of framed logarithmic formal schemes, i.e. the diagram
$$
\begin{tikzcd}
P \ar{r}{\theta} \ar{d} & Q\ar{d} \\
\Gamma \left (\frakS,\ov{\calM}_{\frakS} \right ) \ar[swap]{r}{f^{\flat}} & \Gamma \left (\frakX,\ov{\calM}_{\frakX} \right )
\end{tikzcd}
$$
is commutative. The presheaf on $\boldsymbol{\op{L}}$
$$
\frakX \times_{\frakS,[Q]}^{\op{log}} \frakX: T \mapsto \frakX(T) \times_{\frakS(T)\times [Q](T)}^{\op{log}}\frakX(T),
$$
is then representable by a logarithmic $p$-adic formal scheme which is affine and log étale over $\frakX \times_{\frakS}^{\op{log}}\frakX.$
Furthermore, the diagonal immersion $\frakX \ra \frakX \times_{\frakS}^{\op{log}}\frakX$ factors as
$$
\begin{tikzcd}
 & \frakX \times_{\frakS,[Q]}^{\op{log}} \frakX \ar{d} \\
\frakX \ar{r} \ar{ur} & \frakX \times_{\frakS}^{\op{log}}\frakX,
\end{tikzcd}
$$
where $\frakX \ra \frakX \times_{\frakS,[Q]}^{\op{log}} \frakX$ is strict and the conormal sheaf of the immersion
$\frakX_l \ra \frakX_l \times_{\frakS_l,[Q]}^{\op{log}} \frakX_l$
is canonically isomorphic to the module of logarithmic differentials $\omega^1_{\frakX_l/\frakS_l},$ for all positive integers $l.$
\end{parag}

\begin{parag}\label{intro16}
We equip $\op{Spf}W(k)$ with the trivial logarithmic structure.
We consider an fs logarithmic $p$-adic formal scheme $\frakS$ log flat and locally of finite type over $\op{Spf}W(k).$ We also consider a log smooth (\cite{Ogus2018} III 2.5.1) morphism $f:(\frakX,Q) \ra (\frakS,P)$ of framed fs logarithmic $p$-adic formal schemes. Let $\calI$ be the ideal of the exact diagonal immersion $\frakX \ra \frakY=\frakX \times_{\frakS,[Q]}^{\op{log}}\frakX.$
For a positive integer $n,$ let $R_{\frakX,n}$ be the dilatation of $\calI+(p)$ with respect to $p^n$ in $\frakY.$ In other words, $R_{\frakX,n}$ is the largest open formal subscheme of the admissible blow-up of $\calI+(p^n)$ in $\frakY$ such that $\left ( \calI+(p^n) \right ) \Ox_{R_{\frakX,n}}=(p^n).$ We then consider, for every positive integer $l,$ the PD-envelope $\left (P_{\frakX/\frakS,n} \right )_l$ of the immersion $\frakX_l \ra \left (R_{\frakX,n} \right )_l$ and equip it with the logarithmic structure pull back of that of $\frakY_l.$ We consider the inductive limit
$$
P_{\frakX/\frakS,n} =\lim\limits_{\substack{\longrightarrow \\l\ge 1}}\left (P_{\frakX/\frakS,n}\right )_l.
$$
The logarithmic schemes $\left (P_{\frakX/\frakS,n} \right )_l$ provide stratified interpretations for $p^n$-connections. More precisely, we prove that the logarithmic schemes $R_{\frakX,n}$ and $P_{\frakX/\frakS,n}$ have natural structures of formal groupoids and that the category $n\text{-}\op{MHS}(\frakX_l/\frakS_l)$ of $\Ox_{\frakX_l}$-modules equipped with stratifications relative to the formal groupoid $\left (P_{\frakX/\frakS,n} \right )_l,$ is equivalent to the category $p^n\text{-}\op{MIC}^{\text{qn}}\left (\frakX_l/\frakS_l\right )$ of $\Ox_{\frakX_l}$-modules equipped with a quasi-nilpotent integrable $p^n$-connection \eqref{equivstrat}.
Furthermore, if we suppose that the exact relative Frobenius $F_1:X\ra X'$ lifts to a morphism of framed logarithmic formal $\frakS$-schemes $F:\frakX \ra \frakX',$ where $\frakX'$ is log smooth over $\frakS,$ then $F$ induces a morphism
\begin{equation}\label{Muzannu}
\varphi:P_{\frakX/\frakS,0} \ra P_{\frakX'/\frakS,1}.
\end{equation}
In addition, in \ref{lem96} by proving that the ideal of the diagonal $\frakX \ra \frakX\times_{\frakX'}^{\op{log}}\frakX$ has a unique PD-structure, we prove the existence of a morphism
$$
\Psi:\frakX\times_{\frakX'}^{\op{log}}\frakX \ra P_{\frakX/\frakS,0}.
$$
These morphisms fit into the commutative diagram
$$
\begin{tikzcd}
 & & \frakX'\ar{d}{\iota'} \\
 & P_{\frakX/\frakS,0} \ar{r}{\varphi} \ar{d} & P_{\frakX'/\frakS,1}\ar{d} \\
\frakX\times_{\frakX'}^{\op{log}}\frakX \ar{ur}{\psi} \ar[bend right=-30]{uurr} \ar{r} & \frakX\times^{\op{log}}_{\frakS,[Q]}\frakX \ar{r}{F^2} & \frakX'\times_{\frakS,[Q']}^{\op{log}}\frakX'.
\end{tikzcd}
$$
The morphism $\varphi$ allows us to define a functor
\begin{equation}\label{Psiintro}
\Psi_l:\begin{array}[t]{clc}
1\text{-}\op{MHS}\left (\frakX_l'/\frakS_l\right ) & \ra & 0\text{-}\op{MHS}\left (\frakX_l/\frakS_l\right ) \\
(\calE',\epsilon') & \mapsto & (F_l^*\calE',\varphi_l^*\epsilon').
\end{array}
\end{equation}
\begin{theorem}[\ref{THMX1}, \ref{THMX2}]
Let $f:(\frakX,Q) \ra (\frakS,P)$ be a log smooth morphism of framed fs logarithmic $p$-adic formal schemes such that $\frakS$ is log flat and locally of finite type over $\op{Spf}W(k).$ Denote by $X\ra S$ its special fiber. We suppose that the exact relative Frobenius $X\ra X'$ lifts to a morphism of framed fs $p$-adic logarithmic formal schemes $F:(\frakX,Q) \ra (\frakX',Q')$ over $(\frakS,P),$ where $Q'$ is a monoid defined from $Q$ and $P$ \eqref{PFrob} and such that $\frakX'$ is log smooth over $\frakS.$ Let $n$ be a positive integer. The functors $\Phi_l$ \eqref{intro1} and $\Psi_l$ \eqref{Psiintro} induce fully faithful functors between quasi-coherent and quasi-nilpotent objects, fitting into a commutative diagram
$$
\begin{tikzcd}
1\text{-}\op{MHS}^{qcoh}(\frakX_l'/\frakS_l) \ar{rr}{\Psi_l} \ar[swap,sloped]{d}{\sim} & & 0\text{-}\op{MHS}^{qcoh}(\frakX_l/\frakS_l) \ar[sloped]{d}{\sim} \\
p\text{-}\op{MIC}^{qcoh,qn}(\frakX_l'/\frakS_l) \ar{rr}{\Phi_l} & & \op{MIC}^{qcoh,qn}(\frakX_l/\frakS_l),
\end{tikzcd}
$$
where the vertical arrows are equivalences of categories \eqref{equivstrat} and $qcoh$ and $qn$ denote quasi-coherent and quasi-nilpotent objects respectively.
\end{theorem}
To show the full faithfulness, we prove a kind of logarithmic flat descent theorem for morphisms of quasi-coherent modules, generalizing a result of Kato for structural rings \cite{Kat19}:
\begin{theorem}[\ref{eraflogflatdescent}]
Let $T$ be an fs logarithmic scheme, $(T_i \xrightarrow{f_i} T)_{i\in I}$ a log flat covering (\ref{logflattop}, note that $f_i$ is in particular of Kummer type) and $\calE$ a quasi-coherent $\Ox_T$-module. For any $i,j\in I,$ let $f_{ij}:T_i\times_T^{\op{log}}T_j \ra T$ be the canonical morphism. Then, the sequence of $\Ox_T$-modules
$$0 \ra \calE \ra \prod_{i\in I}f_{i*}f_i^*\calE \ra \prod_{i,j\in I}f_{ij*}f_{ij}^*\calE,$$
where the last arrow is the difference between the morphisms induced by the projections $T_i\times_T^{\op{log}}T_j \ra T_i$ and $T_i\times_T^{\op{log}}T_j \ra T_j,$ is exact.
\end{theorem}
\end{parag}

\textbf{Acknowledgement.} This article is the first part of my thesis prepared at Université Paris-Saclay and IHES. I express my greatest gratitude to my PhD advisor Ahmed Abbes for introducing me to this topic, his guidance and his patience. I also thank Atsushi Shiho and Daxin Xu for their helpful comments and suggestions.

\section{Notations and conventions}

\begin{parag}
In this article, we fix a universe $\mathbb{U}$ and a prime number $p.$
\end{parag}

\begin{parag}\label{Not9}
For multi-indices $I=(I_1,\hdots,I_d),J=(J_1,\hdots,J_d)\in \N^d,$ we denote by $I!$ the product
$I!=\prod_{i=1}I_i!$
and by $\begin{pmatrix}I \\ J \end{pmatrix}$ the binomial coefficient
$
\begin{pmatrix}I \\ J \end{pmatrix}=\prod_{i=1}^d\begin{pmatrix}I_i \\ J_i \end{pmatrix}.
$
\end{parag}

\begin{parag}\label{Not8}
If $X$ is a scheme or a formal scheme, we denote by $|X|$ its underlying topological space.
\end{parag}

\begin{parag}
If $\F$ is a sheaf of sets on a site $X,$ the notation $x\in \F$ means that there exists an object $U$ of $X$ such that $x\in \F(U).$
\end{parag}

\begin{parag}
For morphisms of fs logarithmic schemes $X\ra S$ and $Y\ra S,$ we denote by $X\times_S^{\op{log}}Y$ their fiber product in the category of fs logarithmic schemes. We keep the notation $X\times_SY$ for the fiber product of $X\ra S$ and $Y\ra S$ in the category of logarithmic schemes.
\end{parag}

\begin{parag}\label{Not2}
If $A$ is a ring, $a\in A$ and $M$ is an $A$-module, we denote by $M_{/a\text{-tor}}$ the quotient of $M$ by the submodule of $a$-torsion of $M$ i.e. the submodule consisting of elements $x\in M$ such that $a^kx=0$ for some positive integer $k.$
\end{parag}

\begin{parag}\label{Not1}
The logarithmic structures considered in this article are all defined on the small étale site.  
For a morphism of logarithmic schemes $X\ra S,$ we denote by $\omega^1_{X/S}$ the sheaf of logarithmic differentials of $X$ over $S$ of order $1$ (\cite{Ogus2018} IV 1.2.4). For any positive integer $i,$ we set $\omega^i_{X/S}=\Lambda^i\omega^1_{X/S}.$
\end{parag}

\begin{parag}\label{Not3}
For a logarithmic scheme $T$ of positive characteristic $p,$ we denote by $F_T$ the absolute Frobenius morphism of $T,$ i.e. the morphism given by the identity on the underlying topological spaces, $\Ox_T\ra \Ox_T,\ x\mapsto x^p$ and $\calM_T\ra \calM_T,\ m\mapsto pm.$
\end{parag}

\begin{parag}\label{Not4}
Let $R$ be a ring and $P$ a monoid, we denote by $R[P]$ the free $R$-algebra on $P$ and by $A_R[P]$ the scheme $\Sp R[P]$ equipped with the logarithmic structure induced by the canonical chart $P\ra R[P].$ Let $p$ be a prime number and $n\ge 1$ an integer. We denote by $A_n[P]$ (resp. $A[P]$) the logarithmic scheme $A_{\Z/p^n\Z}[P]$ (resp. $A_{\Z}[P]$).
\end{parag}

\begin{parag}\label{Not5}
If $M$ is a monoid, we denote by $M^{\times}$ the subset of $M$ consisting of invertible elements, by $\ov{M}$ the quotient $M/M^{\times}$ and by $M^{gp}$ the group associated to $M.$ We also denote by $M^{int}$ the image of $M$ by the canonical morphism $M\ra M^{gp}.$ Similarly, if $\calM$ is a sheaf of monoids, we denote by $\calM^{\times}$ the subsheaf of invertible sections, by $\ov{\calM}$ the quotient sheaf $\calM/\calM^{\times}$ i.e. the sheaf associated to the presheaf $U\mapsto \ov{\calM(U)}=\calM(U)/\calM(U)^{\times},$ by $\calM^{gp}$ the sheaf associated to the presheaf $U\mapsto \calM(U)^{gp}$ and by $\calM^{int}$ the sheaf associated to the presheaf $U\mapsto \calM(U)^{int}.$
\end{parag}

\begin{parag}
A morphism of monoids (resp. sheaves of monoids) is said to be surjective if it is so as a map of sets (resp. morphism of sheaves of sets).
\end{parag}

\begin{parag}\label{Not7}
We say that a morphism $u:M\ra N$ of monoids is \emph{strict} if the induced morphism $\ov{u}:\ov{M} \ra \ov{N}$ is an isomorphism.
\end{parag}

\begin{parag}\label{parag42}
We denote by $\boldsymbol{\op{L}}$ the category of fs logarithmic schemes and by $\boldsymbol{\widehat{\op{L}}}$ the category of presheaves of sets on $\boldsymbol{\op{L}}.$ Recall that the Yoneda embedding $\boldsymbol{\op{L}}\ra \boldsymbol{\widehat{\op{L}}}$ commutes with representable projective limits and in particular with fiber products.
\end{parag}

\begin{parag}
For a scheme $X,$ we denote by $\text{ét}_{/X}$ the small étale site on $X$ and by $X_{\text{ét}}$ the small étale topos on $X.$
\end{parag}

\begin{parag}\label{Not6}
For any scheme $X,$ we denote by $\op{QCoh}(X_{\text{zar}})$ (resp. $\op{QCoh}(X_{\text{ét}})$) the category of quasi-coherent $\Ox_X$-modules on the small Zariski site of $X$ (resp. the category of quasi-coherent $\Ox_X$-modules on the small étale site of $X$).
By (\cite{SP} \href{https://stacks.math.columbia.edu/tag/03DX}{Proposition 03DX}), $\op{QCoh}(X_{\text{zar}})$ and $\op{QCoh}(X_{\text{ét}})$ are canonically equivalent.
\end{parag}

\begin{parag}\label{dxuflat}
Let $k$ be a perfect field of characteristic $p,$ $W(k)$ its ring of Witt vectors and $\frakS=\op{Spf}W(k).$ Following (\cite{DXU19} 2.5), we say that a $p$-adic formal $\frakS$-scheme $\frakX$ is \emph{flat over $\frakS$} or that $\frakX$ is a \emph{flat formal $\frakS$-scheme} if multiplication by $p$ on $\Ox_{\frakX}$ is injective. This is equivalent to the fact that, for every affine open formal subscheme $U$ of $\frakX,$ the algebra $\Gamma(U,\Ox_{\frakX})$ is flat over $W(k).$ In this sense, a $p$-adic formal $\frakS$-scheme $\frakX$ is flat over $\frakS$ if and only if $\frakX_n$ is flat over $\frakS_n$ for all positive integers $n,$ where $\frakX_n$ and $\frakS_n$ are obtained from $\frakX$ and $\frakS$ respectively by reduction modulo $p^n.$ Indeed, let $A$ be a $p$-adic $W(k)$-algebra and suppose that, for all integers $n\ge 1,$ $A_n=A/p^nA$ is flat over $W_n(k)=W(k)/p^nW(k).$ Let $x\in A$ such that $px=0.$ By the flatness of $A_n$ over $W_n(k),$ the image of $x$ in $A_n$ belongs to $p^{n-1}A/p^nA.$ Since $A$ is separated, $x=0$ and so $A$ is flat over $W(k).$ 
\end{parag}

\section{The exact relative Frobenius}

\begin{parag}\label{PFrob}
Let $f:X\ra S$ be a morphism of fine logarithmic schemes of characteristic $p.$ Let $X''$ be the base change, in the category of fine logarithmic schemes, of $X$ by the absolute Frobenius morphism $F_S:S\ra S$ of $S$ (\ref{Not3}). Let $F_{X/S}$ be the relative Frobenius morphism of $X$ with respect to $S$ i.e. the unique morphism $X\ra X''$ making the following diagram commutative
\begin{equation}\label{diag481}
\begin{tikzcd}
X\ar{dr}{F_{X/S}}\ar[bend right=-30]{drr}{F_X}\ar[swap,bend right=30]{ddr}{f} & & \\
 & X''\ar{r}\ar{d} & X\ar{d}{f} \\
 & S\ar{r}{F_S} & S
\end{tikzcd}
\end{equation}
Recall that $F_{X/S}$ is \emph{weakly inseparable} (\cite{Ogus2018} III 2.4), so $F_{X/S}$ factors uniquely as a composition of a log étale morphism $G:X'\ra X''$ and an inseparable morphism $F:X\ra X'$ (\cite{Ogus2018} IV 3.3.8): 
\begin{equation}\label{diag51}
\begin{tikzcd}
X\ar{r}{F}\ar{dr}\ar[bend right=30]{rdd} & X'\ar{d}{G}\ar{dr}{\pi} & \\
 & X''\ar{d}\ar{r} & X\ar{d}{f} \\
 & S \ar{r}{F_S} & S
\end{tikzcd}
\end{equation}
The morphism $F:X\ra X'$ is called the \emph{exact relative Frobenius of $X$ with respect to $S$} (\cite{Ogus2018} IV 3.3.9). If $f$ is of Cartier type (\cite{Kat89} 4.8) then $X'=X''$ and $F=F_{X/S}.$

In the rest of this article, we introduce the following notations: if $P$ is a monoid, we denote by
$$F_P:P \ra P,\ x\mapsto px$$
the Frobenius morphism of $P.$ If $\theta:P\ra Q$ is a morphism of fine monoids, we set $Q''=(Q\oplus_{P,F_P}P)^{int}$ and denote by $Q'$ the inverse image of $Q$ by
\begin{equation}\label{eqfrobmon1}
Q''^{gp} \ra Q^{gp},\ (x,y) \mapsto px+\theta^{gp}(y).
\end{equation}
Denote by $v:Q' \ra Q$ the morphism induced by \eqref{eqfrobmon1}. 
If $\theta:P\ra Q$ is a chart for $f,$ then, by the proof of (\cite{Kat89} 4.10), the canonical morphism $X''\ra A_1[Q'']$ is a chart, $X'=X''\times_{A_1[Q'']}A_1[Q']$ and $v:Q' \ra Q$ is a chart of the exact relative Frobenius $F:X\ra X'.$
Note that if $Q$ is saturated, then $Q'$ is clearly saturated. It follows that if $X$ and $S$ are fs, then so is $X'.$
\end{parag}

\begin{definition}\label{defkummer}[\cite{Kat19} 2.1 and 2.2]~ 
\begin{enumerate}
\item A morphism $u:M\ra N$ of fs monoids is said to be \emph{Kummer} if it is injective and for any $y\in N,$ there exists a positive integer $n$ such that $ny\in u(M).$
\item A morphism $f:X\ra Y$ of fs logarithmic schemes is said to be \emph{of Kummer type} if, for every geometric point $\ov{x}$ of $X$ and $\ov{y}=f(\ov{x}),$ the morphism of monoids
$f^{\flat}_{\ov{x}}:\ov{\calM}_{Y,\ov{y}} \ra \ov{\calM}_{X,\ov{x}}$
is Kummer.
\end{enumerate}
\end{definition}

\begin{proposition}\label{propfrob12}
Let $\theta:P\ra Q$ a morphism of fine monoids such that $\theta^{gp}$ is injective and the torsion subgroup of $\op{coker}\theta^{gp}$ is finite of order coprime with $p.$ Let $Q'$ and $v:Q' \ra Q$ be as defined in \ref{PFrob}. Then $v$ is Kummer.
\end{proposition}

\begin{proof}
Let $(x,y),(x',y')\in Q'$ such that $v(x,y)=v(x',y').$
Then
$px+\theta^{gp}(y)=px'+\theta^{gp}(y')\in Q$
and
$p(x-x')=\theta^{gp}(y'-y).$
It follows that $p(x-x')=0$ in $\op{coker}\theta^{gp}.$ And since the torsion subgroup of $\op{coker}\theta^{gp}$ has a finite order coprime with $p,$ we deduce that $x=x'+\theta^{gp}(t)$ for a certain element $t\in P^{gp}.$ Then
$$\theta^{gp}(y+pt)=\theta^{gp}(y)+p(x-x')=\theta^{gp}(y').$$
Since $\theta^{gp}$ is injective, we get $y+pt=y'$ and so, in $Q',$
$$(x,y)=(x'+\theta^{gp}(t),y)=(x',y+pt)=(x',y').$$
Finally, for every $x\in Q,$ $(x,0)\in Q'$ and $v(x,0)=px.$
\end{proof}

\begin{proposition}\label{FKummer}
Let $f:X\ra S$ be a morphism of fs logarithmic schemes of characteristic $p$ and $F:X\ra X'$ the exact relative Frobenius. Then $F$ is of Kummer type \eqref{defkummer}.
\end{proposition}

\begin{proof}
Let $\ov{x}\ra X$ be a geometric point and $\ov{x}'=F(\ov{x}).$ The morphism $\ov{F}_{\ov{x}}^{\flat}:\ov{\calM}_{X',\ov{x}'} \ra \ov{\calM}_{X,\ov{x}}$ is by definition exact. By (\cite{Ogus2018} I 4.3.10), to prove that it is Kummer, it is sufficient to prove that it is small i.e. for all $z\in \ov{\calM}_{X,\ov{x}}^{gp}$ there exists a positive integer $n$ and $y\in \ov{\calM}_{X',\ov{x}'}^{gp}$ such that $nz=\left (\ov{F}_{\ov{x}}^{\flat}\right )^{gp} \left (y \right )$ (\cite{Ogus2018} I 4.3.1). This is immediate since, for every local section $m$ of $\calM_{X},$ we have $pm=F^{\flat}\left (\pi^{\flat}(m) \right ),$ where $\pi:X'\ra X$ is given in \eqref{diag51}.
\end{proof}

\begin{theorem}\label{thmlogflat}
Let $f:X\ra S$ be a log smooth morphism of fine logarithmic schemes of characteristic $p.$ Then the exact relative Frobenius $F:X\ra X'$ \eqref{diag51} is log flat (\cite{Ogus2018} IV 4.1.1).
\end{theorem}

\begin{proof}
By (\cite{Kat89} 3.5), there exists, étale locally on $X$ and $S,$ a chart $\theta:P \ra Q$ of $f$ such that $\theta^{gp}$ is injective, the torsion subgroup of $\op{coker}\theta^{gp}$ has a finite order coprime with $p$ and the morphism $g:X \ra S\times_{A_1[P]}A_1[Q]$ induced by $f$ and $A_1[\theta]$ is smooth as a morphism of schemes. Set $T=S\times_{A_1[P]}A_1[Q]$ and $F_T:T\ra T$ the absolute Frobenius morphism of $T.$ The morphism $g:X\ra T$ is smooth so the relative Frobenius
$$F_{X/T}:X\ra X\times_{T,F_T}T$$
is flat. Let $X''$ be the fiber product of $f:X\ra S$ and $F_S:S\ra S$ in the category of fine logarithmic schemes,
$F_P:P \ra P,\ x\mapsto px,$
$Q''=(Q\oplus_{P,F_P}P)^{int}$ and $Q'$ the inverse image of $Q$ by
$v:Q''^{gp} \ra Q^{gp},\ (x,y) \mapsto px+\theta^{gp}(y).$
By the proof of (\cite{Kat89} 4.10), the canonical morphism $X''\ra A_1[Q'']$ is a chart, $X'=X''\times_{A_1[Q'']}A_1[Q']$ and $v:Q' \ra Q$ is a chart of the exact relative Frobenius $F:X\ra X'.$ By the lemma \ref{lemX333} below, there exists a canonical isomorphism $X'\times_{A_1[Q']}A_1[Q] \xrightarrow{\sim} X\times_{T,F_T}T$ and the morphism $G:X\ra X'\times_{A_1[Q']}A_1[Q],$ induced by $F:X\ra X'$ and $X\ra A_1[Q],$ identifies with $F_{X/T}.$ So $G$ is flat. It remains to prove that the chart $v:Q'\ra Q$ is injective, which we did in \ref{propfrob12}.
\end{proof}

\begin{corollaire}\label{corlogflat}
Let $f:X\ra S$ be a log smooth morphism of fine logarithmic locally Noetherian schemes of characteristic $p$ and $F:X\ra X'$ the exact relative Frobenius \eqref{diag51}. If $F$ is integral (\cite{Kat89} 4.3), then the underlying morphism of schemes of $F$ is faithfully flat.
\end{corollaire}

\begin{proof}
This is an immediate consequence of \ref{thmlogflat} and (\cite{Ogus2018} IV 4.3.5).
\end{proof}

\begin{lemma}\label{lemX333}
Keep the same hypothesis and notation of \ref{thmlogflat} and its proof. Then, there exists a canonical isomorphism
$$X'\times_{A_1[Q']}A_1[Q]\xrightarrow{\sim}X\times_{T,F_T}T.$$
\end{lemma}

\begin{proof}
Clearly,
$X'\times_{A_1[Q']}A_1[Q]=X''\times_{A_1[Q'']}A_1[Q].$
Consider the cartesian diagram
$$
\begin{tikzcd}
T\ar{r}{q_1}\ar[swap]{d}{q_2} & S\ar{d}{\beta} \\
A_1[Q] \ar{r}{A_1[\theta]} & A_1[P]
\end{tikzcd}
$$
and the following commutative diagram
$$
\begin{tikzcd}
X\times_{T,F_T}T \ar{r}\ar[swap]{d}  & X''\times_{A_1[Q'']}A_1[Q] \ar{d} \ar{r} & X'' \ar{d} \ar{r} & X\ar{d} \\
T\ar{r} \ar[swap,bend right=20]{rrr}{F_T} & T''\times_{A_1[Q'']}A_1[Q] \ar{r} & T''\ar{r} & T 
\end{tikzcd}
$$
The big and two right squares are cartesian, so the left square is also cartesian and it is thus sufficient to prove that the morphism
$\varphi:T\ra T''\times_{A_1[Q'']}A_1[Q],$
defined by the relative Frobenius $F_{T/S}:T\ra T''$ and $q_2:T\ra A_1[Q],$ is an isomorphism. For that, we exhibit an inverse morphism.
Consider the two cartesian diagrams
$$
\begin{tikzcd}
T'' \ar{r}{p_1}\ar[swap]{d}{p_2} & T\ar{d}{q_1} & & T''\times_{A_1[Q'']}A_1[Q] \ar{r}{\rho_1} \ar[swap]{d}{\rho_2} & T''\ar{d}\\
S\ar{r}{F_S} & S & & A_1[Q] \ar{r} & A_1[Q'']
\end{tikzcd}
$$
The morphism $\varphi$ fits into the commutative diagram
$$
\begin{tikzcd}
T \ar[bend right=-20]{rrd}{F_{T/S}} \ar[swap,bend right=30]{ddr}{q_2} \ar{dr}{\varphi} &  & \\
 & T''\times_{A_1[Q'']}A_1[Q]\ar{r}{\rho_1}\ar{d}{\rho_2} & T''\ar{d} \\
 & A_1[Q]\ar{r} & A_1[Q'']
\end{tikzcd}
$$
Consider the morphism
$\psi:T''\times_{A_1[Q'']}A_1[Q] \ra T$
defined in the diagram
$$
\begin{tikzcd}
T''\times_{A_1[Q'']}A_1[Q] \ar{r}{\rho_1} \ar[swap,bend right=30]{ddr}{\rho_2} \ar{dr}{\psi} & T''\ar{dr}{p_2} & \\
 & T\ar{r}{q_1}\ar{d}{q_2} & S\ar{d}{\beta} \\
 & A_1[Q]\ar{r}{A_1[\theta]} & A_1[P]
\end{tikzcd}
$$
The fact that $\psi\circ \varphi=\op{Id}_T$ follows immediately from the definitions of $\varphi$ and $\psi.$ Let us check that $\varphi\circ \psi=\op{Id}_{T''\times_{A_1[Q'']}A_1[Q]}.$ By definition, we have $\rho_2\circ \varphi\circ \psi=q_2\circ \psi=\rho_2.$
Now, we just need to check that $\rho_1\circ \varphi \circ \psi=\rho_1.$
First, we have $p_2\circ \rho_1 \circ \varphi \circ \psi=p_2\circ F_{T/S}\circ \psi=q_1\circ \psi=p_2\circ \rho_1.$
It is thus sufficient to check that $p_1\circ \rho_1\circ \varphi\circ \psi=p_1\circ \rho_1,$
which is equivalent to the two following equalities
\begin{alignat}{2}
q_1\circ p_1\circ \rho_1\circ \varphi\circ \psi=q_1\circ p_1\circ \rho_1 \label{Ja1} \\
q_2\circ p_1\circ \rho_1\circ \varphi\circ \psi=q_2\circ p_1\circ \rho_1. \label{Ja2}
\end{alignat}
The equality (\ref{Ja1}) follows from
\begin{alignat*}{2}
q_1\circ p_1\circ \rho_1\circ \varphi\circ \psi &= F_S\circ p_2\circ \rho_1\circ \varphi\circ \psi
= F_S\circ p_2\circ \rho_1
= q_1 \circ p_1 \circ \rho_1.
\end{alignat*}
Now, consider the morphism $F_Q:Q\ra Q,\ x\mapsto px.$
It fits into the following commutative diagram
$$
\begin{tikzcd}
T''\times_{A_1[Q'']}A_1[Q] \ar{r}{\rho_1} \ar[swap]{d}{\rho_2} & T''\ar{d}{p_2}\ar{r}{p_1} & T\ar{d}{q_2} \\
A_1[Q] \ar{r} \ar[swap,bend right=20]{rr}{A_1[F_Q]} & A_1[Q''] \ar{r} & A_1[Q] 
\end{tikzcd}
$$
the equality (\ref{Ja2}) follows then from
\begin{alignat*}{2}
q_2 \circ p_1 \circ \rho_1 \circ \varphi \circ \psi &= q_2 \circ p_1 \circ F_{T/S} \circ \psi 
= q_2 \circ F_T \circ \psi \\
&= A_1[F_Q] \circ q_2 \circ \psi 
= A_1[F_Q] \circ \rho_2 \\
&= q_2 \circ p_1 \circ \rho_1.
\end{alignat*}
This finishes the proof.
\end{proof}

\begin{proposition}\label{era3proplogstr}
Let $f:X\ra Y$ and $g:Z\ra Y$ be morphisms of logarithmic schemes such that $g$ is strict. Suppose that there exists a morphism of schemes $h:X \ra Z$ such that $f=g\circ h$ in the category of schemes. Then there exists a unique morphism of sheaves of monoids $h^{\flat}:h^{-1}\calM_Z \ra \calM_X$ making $h$ a morphism of logarithmic schemes and such that $f=g\circ h$ in the category of logarithmic schemes.
\end{proposition}

\begin{proof}
Since $g$ is strict, the canonical morphism $g^*\calM_Y \ra \calM_Z,$ where $g^*\calM_Y$ is the logarithmic structure pullback of $\calM_Y,$ is an isomorphism. Then $h^*\calM_Z=f^*\calM_Y.$
The data of a morphism of prelogarithmic structures
$h^{-1}\calM_Z \ra \calM_X$
is equivalent to the data of a morphism of logarithmic structures
$h^*\calM_Z \ra \calM_X.$
Since $h^*\calM_Z=f^*\calM_Y,$ the only morphism satisfying the desired condition is
$f^{\flat}:f^*\calM_Y \ra \calM_X.$
\end{proof}

\section{Logarithmic differential operators}

Let $S$ be a scheme of positive characteristic $p$ equipped with a fine logarithmic structure and a PD structure $\gamma$ and let $X\ra S$ be a smooth morphism of fine logarithmic schemes. We suppose that $\gamma$ extends to $X.$

\begin{parag}\label{P11}
Let $P_{X/S}$ be the logarithmic PD-envelope of the diagonal immersion $X\ra X\times_SX$ (\cite{Kat89} 5.4). We denote by $\ov{\calI}_{X/S}$ its PD-ideal. By (\cite{Kat89} 5.8.1), we have a canonical isomorphism
$$
\omega^1_{X/S} \xrightarrow{\sim} \ov{\calI}/\ov{\calI}^{[2]}.
$$
The canonical closed immersion $X\ra P_{X/S}$ is a nilideal and hence a universal homeomorphism. We can thus use it to identify the small étale topoi $X_{\text{ét}}$ and $P_{X/S,\text{ét}}.$ We denote by $\calP_{X/S}$ the structural ring of $P_{X/S}$ and, for every integer $n\ge 0,$
$$
\calP_{X/S}^{\{n\}}=\calP_{X/S}/\ov{\calI}^{[n+1]}.
$$
For an $\Ox_X$-module $\calE,$ we consider the $\Ox_X$-module structure on $\calP_{X/S}$ given by the second (resp. first) projection $P_{X/S} \ra X$ to form the tensor product $\calP_{X/S}\otimes_{\Ox_X}\calE$ (resp. $\calE \otimes_{\Ox_X} \calP_{X/S}$), which we then consider as an $\Ox_X$-module via the first (resp. second) projection $P_{X/S} \ra X.$
If $\calE$ and $\F$ are $\Ox_X$-modules, an $\Ox_X$-linear morphism
$$
\calP_{X/S}^{\{n\}}\otimes_{\Ox_X}\calE \ra \F\qquad \left (\text{resp.}\ \calP_{X/S}\otimes_{\Ox_X}\calE \ra \F\right ),
$$
is called a \emph{differential operator of order $\le n$} (resp. \emph{hyperdifferential operator}). We denote by $\calD^n_{X/S}\left (\calE,\F \right )$ the sheaf of differential operators of order $\le n$ and $\calD_{X/S}\left (\calE,\F \right )=\bigcup_{n\ge 0}\calD^n_{X/S}\left (\calE ,\F\right ).$ If $\calE=\F=\Ox_X,$ we will use the notation $\calD_{X/S}$ instead. Differential operators of a given order $\le n$ will be considered as a hyperdifferential operators via the canonical projection $\calP_{X/S} \ra \calP_{X/S}^{\{n\}}$ and both will be referred to simply by differential operators.

The PD-envelope $P_{X/S}$ has a natural groupoid structure given by the morphism $P_{X/S}\times_XP_{X/S} \ra P_{X/S}$ induced by the $(1,3)$-projection $X\times_SX\times_SX \ra X\times_SX,$ the canonical immersion $X\ra P_{X/S}$ and the morphism $P_{X/S} \ra P_{X/S}$ that exchanges factors. This corresponds to a natural Hopf algebra structure on $\calP_{X/S}$ (\cite{DXU19} 4.2) which defines a ring structure on its dual $\widehat{\calD}_{X/S}=\calP_{X/S}^\vee,$ where $\calP_{X/S}$ is considered as an $\Ox_X$-module via the first projection. This ring structure is explicitly given as follows : let
$$
\delta:\calP_{X/S} \ra \calP_{X/S}\otimes_{\Ox_X}\calP_{X/S}
$$
be the morphism induced by $P_{X/S}\times_XP_{X/S} \ra P_{X/S}.$
For differential operators $f,g:\calP_{X/S} \ra \Ox_X,$ their product $f\circ g$ is given by the composition
$$
\calP_{X/S} \xrightarrow{\delta} \calP_{X/S}\otimes_{\Ox_X}\calP_{X/S} \xrightarrow{\op{Id} \otimes g} \calP_{X/S} \xrightarrow{f} \Ox_X.
$$
\end{parag}

\begin{exmp}\label{exmp24}
Consider the split exact sequence of $\Ox_X$-modules
$$0 \ra \omega^1_{X/S}=\ov{\calI}/\ov{\calI}^{[2]} \ra \calP_{X/S}/\ov{\calI}^{[2]}=\calP_{X/S}^{\{1\}} \ra \calP_{X/S}/\ov{\calI}=\Ox_X \ra 0.$$
The splitting given by the first projection $P_{X/S} \ra X$ yields an isomorphism
$$\calP_{X/S}^{\{1\}} \xrightarrow{\sim}\Ox_X\oplus \omega^1_{X/S}.$$
A logarithmic derivation $\partial:\Ox_X\times \calM_X^{gp} \ra \Ox_X$ corresponds to an $\Ox_X$-linear map $\partial:\omega^1_{X/S} \ra \Ox_X.$
Composing with the projection
$\calP_{X/S}^{\{1\}}=\Ox_X\oplus \omega^1_{X/S}\ra \omega^1_{X/S},$
we obtain a differential operator of degree $\le 1,$ abusively denoted by
$\partial: \calP_{X/S}^{\{1\}} \ra \Ox_X.$
Via this construction, we will often consider sections of the tangent sheaf $\calT_{X/S}=\mathscr{Hom}_{\Ox_X}(\omega^1_{X/S},\Ox_X)$ as differential operators.
\end{exmp}

\begin{proposition}\label{KhamineiExact}
Let $\Delta:X \ra Y$ be an exact closed immersion of integral logarithmic schemes with ideal $\calI.$ Then the sequence
\begin{equation}
0 \ra \Delta^{-1}(1+\calI) \xrightarrow{\lambda} \Delta^{-1}\calM_{Y} \xrightarrow{\Delta^{\flat}} \calM_{X} \ra 0
\end{equation}
is an exact sequence of monoids i.e. $\lambda$ is injective, $\Delta^{\flat}$ is surjective and for any local sections $m$ and $m'$ of $\Delta^{-1}\calM_{Y}$ such that $\Delta^{\flat}(m)=\Delta^{\flat}(m'),$ there exists a unique local section $a$ of $\Delta^{-1}(1+\calI)$ such that
$
m+\lambda(a)=m'.
$
\end{proposition}

\begin{proof}
The morphism $\lambda$ (resp. $\Delta^{\flat}$) is injective (resp. surjective) by definition.
The immersion $\Delta$ is exact so it is strict. It follows that $\calM_{X}=\Delta^*\calM_{Y}$ is equal to the direct sum
$$
\begin{tikzcd}
\gamma^{-1}\Ox_{X}^* \ar[hook]{r} \ar[swap]{d}{\gamma} & \Delta^{-1}\calM_{Y}\ar{d} \\
\Ox_{X}^* \ar{r} & \Delta^{-1}\calM_{Y}\oplus_{\gamma^{-1}\Ox_{X}^*}\Ox_{X}^*,
\end{tikzcd}
$$
where $\gamma:\Delta^{-1}\calM_{Y} \xrightarrow{\Delta^{-1}\alpha_{Y}}\Delta^{-1}\Ox_{Y} \xrightarrow{\Delta^{\#}} \Ox_{X}.$
The sequence \eqref{era2exactseqprime} becomes isomorphic to
$$
0 \ra \Delta^{-1}(1+\calI) \xrightarrow{\lambda} \Delta^{-1}\calM_{Y} \xrightarrow{\Delta^{\flat}} \Delta^{-1}\calM_{Y}\oplus_{\gamma^{-1}\Ox_{X}^*}\Ox_{X}^*,
$$
where $\Delta^{\flat}(m)=(m,1)$ for any local section $m$ of $\Delta^{-1}\calM_{Y}.$
Let $\ov{x} \ra X$ be a geometric point and $m,m'\in \calM_{Y,\ov{x}}$ such that
$
\Delta^{\flat}_{\ov{x}}(m)=\Delta^{\flat}_{\ov{x}}(m').
$
By definition of amalgamated sums in the category of monoids, there exists $t,t'\in \left (\gamma^{-1}\Ox_{X}^*\right )_{\ov{x}}$ such that
$$
\begin{cases}
m+t=m'+t' \\
\gamma(t)=\gamma(t').
\end{cases}
$$
Since $\gamma(t)\in \Ox_{X}^*$ and $\Ox_{X,\ov{x}}=\Ox_{Y,\ov{x}}/\calI_{\ov{x}},$ we get
$
\alpha_{Y,\ov{x}}(t)\in \Ox_{Y,\ov{x}}^*.
$
It follows that $t\in \calM_{Y,\ov{x}}^*$ and so
$
m=m'+t'-t \in \calM_{Y,\ov{x}}.
$
Since
$
t-t'=\alpha_{Y,\ov{x}}^{-1} \left (\alpha_{Y,\ov{x}}(t)\alpha_{Y,\ov{x}}(t')^{-1} \right )
$
and
$$
\Delta^{\#}_{\ov{x}} \left (\alpha_{Y,\ov{x}}(t)\alpha_{Y,\ov{x}}(t')^{-1} \right )=\gamma(t)\gamma(t')^{-1}=1,
$$
we get
$
\alpha_{Y,\ov{x}}(t)\alpha_{Y,\ov{x}}(t')^{-1}\in 1+\calI_{\ov{x}}
$
and
$
t-t'=\lambda_{\ov{x}} \left (\alpha_{Y,\ov{x}}(t)\alpha_{Y,\ov{x}}(t')^{-1} \right ).
$
\end{proof}

\begin{parag}\label{P2}
The canonical immersion $\iota:X\ra P_{X/S}$ induces an exact sequence
$$
0 \ra \iota^{-1}(1+\ov{\calI}) \xrightarrow{\lambda} \iota^{-1} \calM_{P_{X/S}} \ra \calM_X \ra 0.
$$
For a local section $m$ of $\calM_X,$ there exists a unique local section $\eta(m)$ of $\ov{\calI}$ such that $\lambda(1+\eta(m))+p_1^{\flat}m=p_2^{\flat}m,$ where $p_1,p_2:P_{X/S} \ra X$ are the canonical projections.
If $X$ is smooth over $S$ and $m_1,\hdots,m_d\in \Gamma(U,\calM_X)$ are local coordinates over an étale $X$-scheme $U$ (i.e. $(\op{dlog}m_i)_{1\le i\le d}$ is a basis of $\omega^1_{U/S}$) then we have an isomorphism of $\Ox_U$-PD-algebras (\cite{Kat89} 6.5) given by
$$\begin{array}[t]{clc}\Ox_U\langle T_1,\hdots,T_d\rangle & \ra & \calP_{U/S}\\ T_i& \mapsto &\eta(m_i)\end{array}$$
where $\Ox_U\langle T_1,\hdots,T_d\rangle$ is the PD-polynomial algebra on the PD-ring $\Ox_U$ and $\calP_{U/S}$ is an $\Ox_U$-algebra by the first projection $p_1:P_{U/S}\ra U.$ We denote by $(\partial_I)_{I\in \mathbb{N}^d}$ the dual basis of $(T^{[I]})_{I\in\mathbb{N}^d},$ so that for any multi-index $I\in\mathbb{N}^d$
\begin{equation}\label{partialI}
\partial_I:\calP_{U/S}\ra \Ox_U,\ T^{[J]}\mapsto \delta_{IJ}=\begin{cases}1\ \op{if}\ I=J\\ 0\ \op{else} \end{cases}\end{equation}
where $T^{[J]}=\prod_{i=1}^dT_i^{[J_i]}.$
\end{parag}

\begin{proposition}\label{prop17}
Suppose that $X\ra S$ has local coordinates $m_1,\hdots,m_d\in \Gamma(X,\calM_X).$ For any multi-index $I\in \mathbb{N}^d$ and any $1\le i\le d,$ denote by $\partial_I$ the differential operator defined in \ref{P2} and let $\epsilon_i=(0,\hdots,0,1,0\hdots,0)$ be the multi-index whose all coefficients are zero except for the $i$th which is equal to $1.$ Then, for all local sections $m$ of $\calM_X,$ multi-indices $I,J\in \N^d,$ $1\le i\le d$ and $k\in \N,$ we have
\begin{alignat}{2}
\qquad & \delta(\eta(m)) = 1\otimes \eta(m) + \eta(m)\otimes 1 + \eta(m)\otimes \eta(m) \\
\qquad & \partial_J\circ \partial_{\epsilon_i} = \partial_{J+\epsilon_i}+J_i\partial_J \\
\qquad & \partial_{k\epsilon_i} = \prod_{j=0}^{k-1}(\partial_{\epsilon_i}-j) \\
\qquad & \text{If}\ I_j=0\ \text{then}\ \partial_I\circ \partial_{k\epsilon_j} = \partial_{I+k\epsilon_j} \\
\qquad & \partial_{\epsilon_i}\circ \partial_J = \partial_{J+\epsilon_i}+J_i\partial_J=\partial_J\circ \partial_{\epsilon_i} \\
\qquad & \partial_I\circ \partial_J = \partial_J\circ \partial_I \\
\qquad & \partial_{(pn+r)\epsilon_i}=(\partial_{p\epsilon_i})^n \circ \partial_{r\epsilon_i} \label{eqiran2} \\
\qquad & \partial_I\circ \partial_{pJ}=\partial_{I+pJ}.
\end{alignat}
\end{proposition}

\begin{proof}
The first equality is proved in (\cite{Ogus94} (1.1.4.2)). The second equality is proved at the end of the proof of (\cite{Ogus94} 1.1.5). The others follow. We just prove the two last ones : Since we work in characteristic $p,$
\begin{alignat*}{2}
\partial_{(pn+r)\epsilon_i} &= \prod_{k=0}^{pn+r-1}(\partial_{\epsilon_i}-k)
= \prod_{k=0}^{pn-1}(\partial_{\epsilon_i}-k) \circ \prod_{k=pn}^{pn+r-1}(\partial_{\epsilon_i}-k)\\
&= \prod_{l=0}^{n-1} \prod_{k=pl}^{p(l+1)-1}(\partial_{\epsilon_i}-k) \circ \prod_{k=pn}^{pn+r-1}(\partial_{\epsilon_i}-k) \\
&= \prod_{l=0}^{n-1} \prod_{k=0}^{p-1}(\partial_{\epsilon_i}-k) \circ \prod_{k=0}^{r-1}(\partial_{\epsilon_i}-k)
= (\partial_{p\epsilon_i})^n \circ \partial_{r\epsilon_i}.
\end{alignat*}
Now the second equality. We have
\begin{alignat*}{2}
\partial_{I+pJ} &= \prod_{k=1}^d\partial_{(I_k+pJ_k)\epsilon_k}
= \prod_{k=1}^d\left (\partial_{p\epsilon_k}\right )^{J_k}\circ \partial_{I_k\epsilon_k} \\
&= \prod_{k=1}^d\left (\partial_{p\epsilon_k}\right )^{J_k} \circ \prod_{k=1}^d\partial_{I_k\epsilon_k}
= \prod_{k=1}^d\partial_{pJ_k\epsilon_k} \circ \prod_{k=1}^d\partial_{I_k\epsilon_k}
= \partial_{pJ}\circ \partial_I,
\end{alignat*}
where the second and forth lines result from \eqref{eqiran2}.
\end{proof}

\section{Connections and stratifications}

\begin{parag}\label{parKhaminei}
Let $f:X\ra S$ be a morphism of logarithmic schemes, $\calE$ an $\Ox_X$-module and $\lambda\in \Gamma(X,\Ox_X).$ A \emph{$\lambda$-connection on $\calE$} is an $f^{-1}\Ox_S$-linear morphism $\nabla:\calE\ra \calE\otimes_{\Ox_X}\omega^1_{X/S}$ satisfying the modified Leibniz rule
$$\nabla(ax)=a\nabla(x)+\lambda x\otimes da,$$
for all local sections $a$ and $x$ of $\Ox_X$ and $\calE$ respectively, and where $d:\Ox_X\ra \omega^1_{X/S}$ denotes the universal derivation. 
A $\lambda$-connection $\nabla$ on $\calE$ induces, for any positive integer $i,$ a morphism
$$\nabla^i:\calE\otimes_{\Ox_X}\omega^i_{X/S}\ra \calE\otimes_{\Ox_X}\omega^{i+1}_{X/S}$$
defined for any local sections $x$ and $\omega$ of $\calE$ and $\omega^i_{X/S}$ respectively by
$$\nabla^i(x\otimes \omega)=\nabla(x)\wedge \omega+\lambda x\otimes d\omega$$
where $\nabla(x)\wedge \omega$ denotes the image of $\nabla(x)\otimes \omega$ by the canonical morphism $\calE\otimes_{\Ox_X} \omega^1_{X/S}\otimes_{\Ox_X}\omega^i_{X/S}\ra \calE\otimes_{\Ox_X}\omega^{i+1}_{X/S}.$

The \emph{curvature of $\nabla$,} denoted by $K(\nabla),$ is the composition $\nabla^1\circ \nabla,$ which is $\Ox_X$-linear. We say that the $\lambda$-connection $\nabla$ is \emph{integrable} if $K(\nabla)=0.$ We denote by $\boldsymbol{\lambda\text{-}\op{MIC}(X/S)}$ the category of $\Ox_X$-modules equipped with an integrable $\lambda$-connection.

$1$-connections are simply called \emph{connections} and integrable $0$-connections are called \emph{Higgs fields.} 

If $\partial:\Ox_X\times\calM_X\ra \Ox_X$ is a logarithmic derivation and $\nabla:\calE \ra \calE \otimes_{\Ox_X}\omega^1_{X/S}$ is a connection, we denote by $\nabla(\partial)$ the composition
\begin{equation}\label{nablacomp}
\calE\xrightarrow{\nabla}\calE\otimes_{\Ox_X}\omega^1_{X/S}\xrightarrow{\op{Id}_{\calE}\otimes \partial}\calE
\end{equation}
where $\partial:\omega^1_{X/S}\ra \Ox_X$ is the $\Ox_X$-linear homomorphism induced by the logarithmic derivation $\partial$. It is clear that $\nabla(\partial)$ is $f^{-1}\Ox_S$-linear.

If $\partial_1,\partial_2:\Ox_X\times\calM_X\ra \Ox_X$ are two logarithmic derivations, we denote by $K(\nabla)(\langle \partial_1,\partial_2\rangle)$ the composition
$$\calE\xrightarrow{K(\nabla)}\calE\otimes_{\Ox_X}\omega^2_{X/S}\xrightarrow{\op{Id}_{\calE}\otimes (\langle\partial_1, \partial_2\rangle)} \calE$$
where $\langle\partial_1, \partial_2\rangle$ denotes the morphism
$$\langle\partial_1, \partial_2\rangle:\begin{array}[t]{clc}\omega^2_{X/S} & \ra & \Ox_X\\ \omega_1\wedge \omega_2 & \mapsto & \partial_1(\omega_1)\partial_2(\omega_2)-\partial_1(\omega_2)\partial_2(\omega_1). \end{array}$$
\end{parag}

\begin{parag}\label{parag23}
The algebra $\op{Der}_{X/S}(\Ox_X)$ of logarithmic derivations $\Ox_X\times \calM_X\ra \Ox_X$ relative to $S$ has a Lie algebra structure given by the Lie bracket 
$$[\partial_1,\partial_2]=([D_1,D_2],D_1\delta_2-D_2\delta_1)$$ 
for all logarithmic derivations $\partial_i=(D_i,\delta_i):\Ox_X\times\calM_X\ra \Ox_X,\ i=1,2$ (\cite{Ogus2018} V 2.1.2). 
A straight forward computation shows that for a given $\lambda$-connection $\nabla:\calE \ra \calE \otimes_{\Ox_X}\omega^1_{X/S},$ we have
\begin{equation}\label{art1371}
\lambda \nabla([\partial_1,\partial_2])=[\nabla(\partial_1),\nabla(\partial_2)]-K(\nabla)(\langle\partial_1, \partial_2\rangle).
\end{equation}
It follows that $\nabla$ is integrable if and only if $\nabla([\partial_1,\partial_2])=[\nabla(\partial_1),\nabla(\partial_2)]$ for all logarithmic derivations $\partial_1$ and $\partial_2$ i.e. $\nabla:\op{Der}_{X/S}(\Ox_X)\ra \op{End}_{f^{-1}\Ox_S}(\calE)$ is a morphism of Lie algebras.
\end{parag}

\begin{parag}\label{P46}
In the remaining of this section, we fix a smooth morphism $f:X\ra S$ of logarithmic schemes of positive characteristic $p.$
For any logarithmic derivation $\partial=(D,\delta):\Ox_X\times\calM_X\ra \Ox_X,$ let
$$\partial^{(p)}=(D^p,F_X^{\#}\circ \delta+D^{p-1}\circ \delta):\Ox_X \times \calM_X \ra \Ox_X,$$
where $D^p$ is the composition of $D$ with itself $p$ times, $F_X:X\ra X$ is the absolute Frobenius morphism of $X$ \eqref{Not3} and $F_X^{\#}:\Ox_X\ra \Ox_X$ is the corresponding homomorphism of structural rings. Then, $\partial^{(p)}$ is also a logarithmic derivation and the $p$-operation $\partial\mapsto \partial^{(p)}$ defines a restricted Lie algebra structure on $\op{Der}_{X/S}(\Ox_X)$ (\cite{Ogus94} 1.2.1). The differential operators $\partial_{\epsilon_i},$ defined in \ref{prop17}, satisfy $\partial_{\epsilon_i}^{(p)}=\partial_{\epsilon_i}$ (\cite{Ogus94} 1.2.2).

We define the $p$-curvature $\psi_{\nabla}$ of a connection $\nabla:\calE \ra \calE \otimes_{\Ox_X} \omega^1_{X/S}$ on an $\Ox_X$-module $\calE,$ as in the non logarithmic case : 
\begin{equation}\label{Mojtaba2}
\psi_{\nabla}:\begin{array}[t]{clc}\op{Der}_{X/S}(\Ox_X) & \ra & \op{End}_{f^{-1}\Ox_S}(\calE)\\ \partial & \mapsto & \nabla(\partial)^p-\nabla(\partial^{(p)}).\end{array}
\end{equation}
For any logarithmic derivation $\partial=(D,\delta):\Ox_X\times \calM_X\ra \Ox_X,$ the endomorphism $\psi_{\nabla}(\partial)$ is $\Ox_X$-linear. Indeed, for any local sections $a$ and $x$ of $\Ox_X$ and $\calE$ respectively,
\begin{alignat}{2}
\nabla(\partial)(ax) &= a\nabla(\partial)(x)+D(a)x \label{eq461}\\
\nabla(\partial^{(p)})(ax) &= a\nabla(\partial^{(p)})(x)+D^p(a)x.
\end{alignat} 
We can prove by induction that, for all positive integers $k,$
\begin{equation}\label{eq463}
\nabla(\partial)^k(ax)=\sum_{i=0}^k\begin{pmatrix}k\\ i\end{pmatrix}D^{k-i}(a)\nabla(\partial)^i(x).
\end{equation}
For $k=p,$ we get
$$\nabla(\partial)^p(ax)=a\nabla(\partial)^p(x)+D^p(a)x.$$
The $\Ox_X$-linearity of $\psi_{\nabla}(\partial)$ follows:
\begin{alignat*}{2}
\psi_{\nabla}(\partial)(ax) &= \nabla(\partial)^p(ax)-\nabla(\partial^{(p)})(ax) \\
&= a\nabla(\partial)^p(x)+D^p(a)x-a\nabla(\partial^{(p)})(x)-D^p(a)x = a\psi_{\nabla}(\partial)(x).
\end{alignat*}
\end{parag}

\begin{definition}\label{defSa}
Let $\calE$ be an $\Ox_X$-module, $P_{X/S}(2)$ the logarithmic PD-envelope of the diagonal immersion $X\ra X\times_SX\times_SX$ and $p_{ij}:P_{X/S}(2) \ra P_{X/S}$ the canonical projections. A \emph{stratification on $\calE$} is a sequence of $\calP_{X/S}^{\{n\}}$-linear isomorphisms $\left (\varepsilon_n:\calP^{\{n\}}_{X/S}\otimes_{\Ox_X}\calE\ra \calE\otimes_{\Ox_X}\calP^{\{n\}}_{X/S} \right )_{n\ge 0}$ satisfying the following conditions :
\begin{enumerate}
\item $\varepsilon_0=\op{Id}_{\calE}.$
\item The following diagram is commutative for all $n\ge m$
$$\begin{tikzcd}
\calP^{\{n\}}_{X/S}\otimes_{\Ox_X}\calE\ar{r}{\varepsilon_n}\ar{d} & \calE\otimes_{\Ox_X}\calP^{\{n\}}_{X/S}\ar{d} \\
\calP^{\{m\}}_{X/S}\otimes_{\Ox_X}\calE\ar{r}{\varepsilon_m} & \calE\otimes_{\Ox_X}\calP^{\{m\}}_{X/S}
\end{tikzcd}$$
where the vertical arrows are the canonical homomorphisms.
\item $p_{13}^{*}\varepsilon_n=p_{23}^{*}\varepsilon_n\circ p_{12}^{*}\varepsilon_n$ for all $n.$ 
\end{enumerate}
\end{definition}

\begin{definition}
Let $\calE$ be an $\Ox_X$-module. A \emph{hyperstratification on $\calE$} is a $\calP_{X/S}$-linear isomorphism 
$$\varepsilon:\calP_{X/S}\otimes_{\Ox_X}\calE\ra \calE\otimes_{\Ox_X}\calP_{X/S}$$ satisfying the following conditions :
\begin{enumerate}
\item $\varepsilon$ is equal to the identity $\op{Id}_{\calE}$ modulo the ideal $\calI\subset \calP_{X/S}$ of $X$ in $P_{X/S}.$
\item $p_{13}^*\varepsilon=p_{23}^*\varepsilon\circ p_{12}^*\varepsilon,$ where $p_{ij}$ are given in \ref{defSa}.
\end{enumerate}
\end{definition}

\begin{proposition}\label{prop39}
Let $\calE$ be a quasi-coherent $\Ox_X$-module and suppose that $X\ra S$ is log smooth. The following data are equivalent :
\begin{enumerate}
\item A stratification $(\varepsilon_n)_{n\ge 0}$ on $\calE.$
\item A sequence of $\Ox_X$-linear morphisms $\theta_n:\calE\ra \calE\otimes_{\Ox_X}\calP^{\{n\}}_{X/S}$ satisfying the following conditions:
\begin{itemize}
\item $\theta_0=\op{Id}_{\calE}.$
\item For any integer $n\ge 0,$ the morphism $\theta_n$ is equal to the composition
$$\calE\xrightarrow{\theta_{n+1}}\calE\otimes_{\Ox_X}\calP^{\{n+1\}}_{X/S}\ra \calE\otimes_{\Ox_X}\calP^{\{n\}}_{X/S}$$
where the second arrow is the canonical morphism.
\item For all integers $n,m\ge 0$
$$\begin{tikzcd}
\calE\ar{r}{\theta_{n+m}}\ar{d}{\theta_m} & \calE\otimes_{\Ox_X}\calP^{\{n+m\}}_{X/S}\ar{d}{\op{Id}\otimes \delta^{n,m}}\\
\calE\otimes_{\Ox_X}\calP^{\{m\}}_{X/S}\ar{r}{\theta_n\otimes \op{Id}}& \calE\otimes_{\Ox_X}\calP^{\{n\}}_{X/S}\otimes_{\Ox_X}\calP^{\{m\}}_{X/S} 
\end{tikzcd}$$
\end{itemize}
\item A sequence of compatible $\calP_{X/S}$-linear morphisms
$$\nabla_n:\calD_{X/S}^n(\Ox_X,\Ox_X)\ra \calD_{X/S}^n(\calE,\calE)$$
satisfying for all integers $n,m\ge 0$ and $(f,g)\in \calD_{X/S}^n(\Ox_X,\Ox_X)\times \calD_{X/S}^m(\Ox_X,\Ox_X)$ $$\nabla_{n+m}(f\circ g)=\nabla_n(f)\circ \nabla_m(g)$$ 
and $\nabla_0(\op{Id}_{\Ox_X})=\op{Id}_{\calE}.$
\item An integrable connection $\nabla:\calE\ra \calE\otimes_{\Ox_X}\omega^1_{X/S}.$
\end{enumerate}
\end{proposition}

\begin{proof}
The equivalences between (1) and (2) and between (3) and (4) follow from (\cite{Ber74} II 1.4.4) and (\cite{Ber74} II 4.1.3) respectively. The equivalence between (1) and (4) is poved in (\cite{Ogus94} 1.1.8).
\end{proof}

\begin{definition}\label{nil}
Let $X\ra S$ be a smooth morphism of logarithmic schemes and consider local coordinates $m_1,\hdots,m_d$ of $X$ with respect to $S$ \eqref{P2}. For all $1\le i\le d,$ denote by $\partial_i$ the logarithmic derivation associated to $m_i.$ Let $\nabla$ be an integrable connection on an $\Ox_X$-module $\calE.$ By \eqref{art1371}, the endomorphisms $\nabla(\partial_i)$ and $\nabla(\partial_j)$ of $\calE$ commute for all $1\le i,j\le d.$ Then we can define, for a multi-index $I=(I_1,\hdots,I_d)\in\mathbb{N}^d,$ an endomorphism $\nabla(\partial_I)$ of $\calE$ as follows:
$$\nabla(\partial_I)=\prod_{k=1}^d\prod_{l=0}^{I_k-1}(\nabla(\partial_k)-l\op{Id}_{\calE}).$$
We say that $\nabla$ is \emph{quasi-nilpotent} if for any open subset $U\subset X$ and for any $x\in \Gamma(U,\calE),$ there exists, locally on $U,$ an integer $N$ such that $\nabla(\partial_I)(x)=0$ for all multi-index $I\in\mathbb{N}^d$ satisfying $|I|\ge N.$ We denote by $\boldsymbol{\op{MIC}^{qn}(X/S)}$ the full subcategory of $\boldsymbol{\op{MIC}(X/S)}$ consisting of $\Ox_X$-modules equipped with a quasi-nilpotent integrable connection.
\end{definition}

\begin{proposition}\label{prop412}
Let $\nabla:\calE\ra \calE\otimes_{\Ox_X}\omega^1_{X/S}$ be an integrable connection. Denote by $\psi_{\nabla}$ its $p$-curvature \eqref{Mojtaba2} and consider the differential operators $\partial_I$ defined in \eqref{partialI}.
\begin{enumerate}
\item $\partial_{\epsilon_i}^p-\partial_{\epsilon_i}^{(p)}=\partial_{\epsilon_i}^p-\partial_{\epsilon_i}=\partial_{p\epsilon_i}.$
\item $\nabla(\partial_{p\epsilon_i})=\psi_{\nabla}(\partial_{\epsilon_i}).$
\item Let $I=(I_1,\hdots,I_d)\in \N^d$ be a multi-index. If $\psi_{\nabla}=0$ and there exists $1\le i\le d$ such that $I_i\ge p,$ then
$$\nabla(\partial_I)=0.$$
\end{enumerate}
\end{proposition}

\begin{proof}
We have the following equality of polynomials in $\mathbb{F}_p[x]:$
$$x(x-1)\hdots (x-p+1)=x^p-x.$$
By (\ref{prop17}), we get
$\partial_{\epsilon_i}^p-\partial_{\epsilon_i}=\partial_{\epsilon_i}(\partial_{\epsilon_i}-1)\hdots (\partial_{\epsilon_i}-p+1)=\partial_{p\epsilon_i}.$
Then, by \cite{Ogus94} remark 1.2.2, $\partial_{\epsilon_i}^{(p)}=\partial_{\epsilon_i}$ and so
$\nabla(\partial_{p\epsilon_i})=\nabla(\partial_{\epsilon_i}^p)-\nabla(\partial_{\epsilon_i}^{(p)})=\psi_{\nabla}(\partial_{\epsilon_i}).$
For the third point, by \eqref{prop17}, we have
\begin{alignat*}{2}
\nabla(\partial_I) &= \prod_{j\neq i}\nabla(\partial_{I_j\epsilon_j})\circ \nabla(\partial_{I_i\epsilon_i}) = \prod_{j\neq i}\nabla(\partial_{I_j\epsilon_j})\circ \prod_{k=p}^{n-1}(\nabla(\partial_{\epsilon_i})-k)\circ \nabla(\partial_{p\epsilon_i}) = 0.
\end{alignat*}
\end{proof}

\section{Frames on logarithmic schemes}

\begin{parag}\label{parag43}
Let $P$ be an fs monoid. We define a presheaf of sets $[P]$ on the category $\boldsymbol{\op{L}}$ of fs logarithmic schemes \eqref{parag42} as follows :
\begin{equation}
[P]:\begin{array}[t]{clc} \boldsymbol{\op{L}} & \ra & \boldsymbol{\op{Sets}}\\ T & \mapsto & \op{Hom}\left (P,\Gamma(T,\ov{\calM}_T)\right )\end{array}
\end{equation}
where $\ov{\calM}_T$ is defined in \ref{Not5}. Consider the logarithmic scheme $A[P]$ defined in \ref{Not4}. We have a canonical morphism of presheaves $A[P]\ra [P]$ induced by the tautological map $$P\ra \Gamma(A[P],\calM_{A[P]}).$$
If $X$ is an fs logarithmic scheme, a morphism $X\ra [P]$ of presheaves of sets on $\boldsymbol{\op{L}},$ is said to be \emph{strict} if, for every geometric point $\ov{x}$ of $X,$ there exists an étale neighborhood $U$ of $\ov{x}$ such that the induced morphism $U\ra [P]$ factors into
$$\begin{tikzcd}
U\ar{r}\ar{d} & \left [P\right ]\\
A\left [P\right ]\ar{ur} &
\end{tikzcd}$$
where $U\ra A\left [P\right ]$ is a strict morphism of logarithmic schemes and $A\left [P\right ]\ra [P]$ is the canonical morphism. A \emph{frame} on an fs logarithmic scheme $X$ is a strict morphism $X\ra [P].$ An fs logarithmic scheme $X$ equipped with a frame $X\ra [P]$ will be called a \emph{framed logarithmic scheme} and will be denoted by $(X,P).$ A morphism of framed logarithmic schemes $(X,Q)\ra (S,P)$ is a pair consisting of a morphism of logarithmic schemes $X\ra Y$ and a morphism of monoids $P\ra Q,$ such that the diagram
$$\begin{tikzcd}
X\ar{r}\ar{d} & Y\ar{d} \\
\left [Q\right ]\ar{r} & \left [P\right ]
\end{tikzcd}$$
is commutative, where $[Q]\ra [P]$ is the morphism of presheaves induced by $P\ra Q.$ Note that a frame $X\ra [Q]$ on an fs logarithmic scheme $X$ is equivalent to a morphism $Q \ra \Gamma \left ( X,\ov{\calM}_X\right )$ that lifts, étale locally on $X,$ to a chart $Q \ra \Gamma \left (X,\calM_X\right ).$

Finally, given a morphism of fs monoids $Q\ra P$ and an fs logarithmic scheme $X$ equipped with a morphism $X\ra [Q],$ we denote by $X\times_{[Q]}^{\op{log}}[P]$ the presheaf on $\boldsymbol{\op{L}}$ defined, for any fs logarithmic scheme $T,$ by
$$(X\times_{[Q]}^{\op{log}}[P])(T)=X(T) \times _{[Q](T)} [P](T).$$
Also, given fs logarithmic schemes $X,\ Y$ and $S,$ an fs monoid $Q$ and morphisms $X\ra S\times [Q]$ and $Y\ra S\times [Q],$ we denote by $X\times_{S,[Q]}^{\op{log}}Y$ the presheaf of $\boldsymbol{\widehat{\op{L}}}$ defined, for any fs logarithmic scheme $T,$ by
$$(X\times_{S,[Q]}^{\op{log}}Y)(T)= X(T)\times_{S(T)\times [Q](T)}Y(T).$$ 
\end{parag}

\begin{proposition}\label{etaleframelift}
Let $f:(X,Q) \ra (S,P)$ be a log smooth morphism of framed fs logarithmic schemes. Suppose that $S\ra [P]$ lifts to a chart $S\ra A[P].$ Etale locally on $X,$ there exists a chart $P\ra M$ of $f$ fitting into a commutative diagram
$$
\begin{tikzcd}
X \ar{r} \ar{d} & S \ar{d} \\
A\left [M\right ] \ar{d} \ar{r} & A\left [P\right ] \ar{d} \\
\left [Q \right ] \ar{r} & \left [P \right ],
\end{tikzcd}
$$
and satisfying the following conditions:
\begin{enumerate}
\item $P^{gp} \ra M^{gp}$ is injective and the torsion subgroup of its cokernel is of finite order invertible in $\Ox_X.$
\item The morphism $X \ra S\times_{A[P]}A[M],$ induced by $f$ and the chart $X\ra A[M],$ is étale and strict.
\end{enumerate}
\end{proposition}

\begin{proof}
Let $\ov{x} \ra X$ be a geometric point and $\ov{s}=f(\ov{x}).$ By (\cite{Ogus2018} IV 3.3.1), after restricting to an étale neighborhood of $\ov{x},$ there exists a chart $P\ra M$ of $f$ satisfying the conditions (1) and (2) and such that the chart $X \ra A[M]$ is exact at $\ov{x}.$ By (\cite{Ogus2018} II 2.3.1), the morphism $\gamma:\ov{M} \ra \ov{\calM}_{X,\ov{x}}$ is an isomorphism. We have the commutative diagram
\begin{equation}\label{diagtak1}
\begin{tikzcd}
X \ar{r} \ar{d} & S \ar{d} \\
A[M] \ar{r} & A[P].
\end{tikzcd}
\end{equation}
It is then sufficient to prove the existence of a morphism $A[M] \ra [Q]$ such that the diagram
\begin{equation}\label{diagtak2}
\begin{tikzcd}
A[M] \ar{r} \ar{d} & A[P] \ar{d} \\
\left [Q\right ] \ar{r} & \left [ P\right ]
\end{tikzcd}
\end{equation}
is commutative.
The chart $S\ra A[P]$ induces a morphism $\beta:\ov{P} \ra \ov{\calM}_{S,\ov{s}}.$ The frame $X\ra [Q]$ corresponds to a morphism $Q\ra \Gamma(X,\ov{\calM}_X).$ Consider the composition
$$
\delta:Q \ra \Gamma(X,\ov{\calM}_X) \ra \ov{\calM}_{X,\ov{x}} \xrightarrow{\gamma^{-1}} \ov{M}.
$$
The outer rectangle and the lower square of the following diagram
$$
\begin{tikzcd}
P \ar{r} \ar{d} & Q \ar{d}{\delta} \\
\ov{P} \ar{r} \ar{d} & \ov{M} \ar{d}{\gamma} \\
\ov{\calM}_{S,\ov{s}} \ar{r}{\ov{f^{\flat}_{\ov{x}}}} & \ov{\calM}_{X,\ov{x}}
\end{tikzcd}
$$
are commutative. Since $\gamma$ is an isomorphism,
the upper square is commutative. The result follows.
\end{proof}

\begin{proposition}[\cite{Saito04} 4.2.3]
Let $Q$ be an fs monoid, $X,$ $Y$ and $S$ fs logarithmic schemes and $X\ra S\times [Q],\ Y\ra S\times [Q]$ morphisms of presheaves. Then, the presheaf $X\times_{S,[Q]}^{\op{log}}Y$ is representable by an fs logarithmic scheme which is log étale and affine over $X\times_S^{\op{log}}Y.$
\end{proposition}

\begin{proposition}[\cite{Saito04} 4.2.8]\label{prop45}
Let $(X,Q)\ra (S,P)$ be a morphism of framed fs logarithmic schemes such that $P$ and $Q$ are fs monoids and consider the factorization $$X\xrightarrow{\Delta}X\times_{S,[Q]}^{\op{log}}X\xrightarrow{g}X\times_S^{\op{log}}X$$
of the diagonal morphism $X\ra X\times_S^{\op{log}}X.$ Then $g$ is log étale and $\Delta$ is an exact immersion. Furthermore, if $X\xrightarrow{\Delta_U} U \ra X\times_{S,[Q]}^{\op{log}}X$ is a factorization of $\Delta$ into a closed immersion and an open one and if $\calI$ is the ideal of $\Delta_U,$ then there exists a canonical isomorphism of $\Ox_X$-modules
\begin{equation}\label{iso}
\Delta_U^{-1}\left (\calI/\calI^2\right )\ra \omega^1_{X/S}.
\end{equation}
\end{proposition}

\begin{proof}
Set $Y=X\times_{S,[Q]}^{\op{log}}X$ and denote by $p_1,p_2:X\times_S^{\op{log}}X\ra X$ and $q_1,q_2:Y\ra X$ the canonical projections. By (\cite{Saito04} 4.2.3), the morphism $g$ is log étale and by (\cite{Saito04} 4.2.5.2) the immersion $\Delta$ is strict and hence exact. Consider $Y$ (resp. $X\times_S^{\op{log}}X$) as a scheme over $X$ by the first projection $q_1$ (resp. $p_1$). The conormal exact sequence associated to $X\xrightarrow{\Delta} Y\ra X$ yields a canonical isomorphism 
$\Delta^{-1}\left (\calI/\calI^2\right )\xrightarrow{\sim}\Delta^*\omega^1_{Y/X}.$
Since $g$ is log étale and $p_2$ induces an isomorphism $p_2^*\omega^1_{X/S}\xrightarrow{\sim}\omega^1_{X\times_S^{\op{log}}X/X},$ the morphism $q_2$ induces an isomorphism 
$q_2^*\omega^1_{X/S}\xrightarrow{\sim}\omega^1_{Y/X}.$
It follows that we have an isomorphism 
$
\Delta^{-1}(\calI/\calI^2)\xrightarrow{\sim}\omega^1_{X/S}.
$
\end{proof}

\begin{parag}\label{parag46}
Keep the same notation as in the proof of \ref{prop45}. Let $m$ be a local section of $\calM_X.$ The local section $\op{dlog}m$ of $\omega^1_{X/S}$ is sent by the isomorphism $\omega^1_{X/S}\xrightarrow{\sim}\Delta^*\omega^1_{Y/X}$ to $\op{dlog}q_2^{\flat}m.$
The functor $M \mapsto M^{gp}$ is exact so, by \ref{KhamineiExact}, the exact closed immersion $\Delta$ gives rise to the following exact sequence of abelian groups
$$0\ra \Delta^{-1}(1+\calI)\ra \Delta^{-1}\calM_Y^{gp}\ra \calM_X^{gp}\ra 0.$$
Let $\ov{x}\ra X$ be a geometric point and $\ov{y}\ra Y$ its image. By the previous exact sequence, there exists a unique $\mu(m)\in 1+\calI_{\ov{y}}$ such that
$\alpha_{Y,\ov{y}}^{-1}(\mu(m))=q^{\flat}_{2,\ov{y}}(m)-q_{1,\ov{y}}^{\flat}(m).$
In $\omega^1_{Y/X,\ov{y}},$ we have
\begin{alignat*}{2}
\op{dlog}q_{2,\ov{y}}^{\flat}m &= \op{dlog}(q_{2,\ov{y}}^{\flat}m-q_{1,\ov{y}}^{\flat}m)
= \op{dlog} \alpha_{Y,\ov{y}}^{-1}(\mu(m))
= \mu(m)^{-1}d\mu(m).
\end{alignat*}
It follows that in $(\Delta^*\omega^1_{Y/X})_{\ov{x}}$ we have
$1\otimes \op{dlog}q_{2,\ov{y}}^{\flat}m = d(\mu(m)-1),$
so the section $\mu(m)-1$ is sent, by the isomorphism \eqref{iso}, to the section $\op{dlog}m.$
\end{parag}

\begin{theorem}\label{thm410}
Let $\theta:P\ra Q$ be a morphism of fs monoids and $(f,\theta):(X,Q)\ra (S,P)$ a morphism of framed logarithmic schemes of characteristic $p.$ Recall the exact relative Frobenius diagram \eqref{diag51}. Set
$F_P:P \ra P,\ x\mapsto px,$ $Q''=(Q\oplus_{P,F_P}P)^{int}$ and let $Q'$ be the inverse image of $Q$ by
$v:Q''^{gp} \ra Q^{gp},\quad (x,y) \mapsto px+\theta^{gp}(y).$
Then there exists a canonical frame $X' \ra [Q'].$ In addition, if we set
\begin{equation}\label{FQP}
\pi_{Q/P}:Q \ra Q',\ x\mapsto (x,0),\qquad F_{Q/P}:Q' \ra Q,\ (x,y)\mapsto px+\theta^{gp}(y),
\end{equation}
then $(F,F_{Q/P}):(X,Q)\ra (X',Q')$ and $(\pi,\pi_{Q/P}):(X',Q') \ra (X,Q)$ are morphisms of framed logarithmic schemes.
\end{theorem}

\begin{proof}
Denote by $F:X\ra X'$ the exact relative Frobenius. 
The canonical projections $X''\ra X$ and the canonical morphism $X''\ra S$ induce morphisms
$$\Gamma \left (X,\ov{\calM}_X \right )\ra \Gamma \left (X'',\ov{\calM}_{X''} \right )\leftarrow \Gamma \left (S,\ov{\calM}_S \right ).$$
The frames $X\ra [Q]$ and $S\ra [P]$ correspond to morphisms
$$
Q\ra \Gamma(X,\ov{\calM}_X),\ P\ra \Gamma \left (S,\ov{\calM}_S \right ).
$$
We obtain a morphism
$Q''=(Q\oplus_{P,F_P}P)^{int}\ra \Gamma \left (X'',\ov{\calM}_{X''} \right ),$
and hence a morphism
$X''\ra [Q''].$
The exactness of $F^{\flat}:F^{-1}\calM_{X'} \ra \calM_X$ yields a cartesian square
$$
\begin{tikzcd}
\Gamma \left (X',\ov{\calM}_{X'} \right ) \ar{r} \ar[hook]{d} & \Gamma \left (X,\ov{\calM}_{X} \right ) \ar[hook]{d} \\
\Gamma \left (X',\ov{\calM}_{X'}^{gp} \right ) \ar{r} & \Gamma \left (X,\ov{\calM}_{X}^{gp} \right ).
\end{tikzcd}
$$
Similarly, the canonical morphism $Q' \ra Q$ is exact and hence the square
$$
\begin{tikzcd}
Q' \ar{r} \ar[hook]{d} & Q \ar[hook]{d} \\
Q'^{gp} \ar{r} & Q^{gp}
\end{tikzcd}
$$
is cartesian. Consider the composition
$$
Q'^{gp}=Q''^{gp} \ra \Gamma \left (X'',\ov{\calM}_{X''} \right ) \ra \Gamma \left (X',\ov{\calM}_{X'} \right ),
$$
where the second arrow is induced by the canonical morphism $X' \ra X''.$
There exists a unique morphism $Q' \ra \Gamma \left (X',\ov{\calM}_{X'} \right )$ fitting into the commutative diagram
$$
\begin{tikzcd}
Q' \ar{rr} \ar[hook]{dd} \ar{dr} & & Q \ar{dr} \ar[hook,dashed]{dd} & \\
 & \Gamma \left (X',\ov{\calM}_{X'} \right ) \ar{rr} \ar[hook]{dd} & & \Gamma \left (X,\ov{\calM}_{X} \right ) \ar[hook]{dd} \\
Q'^{gp} \ar[dashed]{rr} \ar{dr} & & Q^{gp} \ar[dashed]{dr} & \\
 & \Gamma \left (X',\ov{\calM}_{X'}^{gp} \right ) \ar{rr} & & \Gamma \left (X,\ov{\calM}_{X}^{gp} \right ).
\end{tikzcd}
$$
This morphism $Q' \ra \Gamma \left (X',\ov{\calM}_{X'} \right )$ yields the desired frame. The fact that $(F,F_{Q/P}):(X,Q)\ra (X',Q')$ and $(\pi,\pi_{Q/P}):(X',Q') \ra (X,Q)$ are morphisms of framed logarithmic schemes is clear.
\end{proof}

\section{Logarithmic formal schemes}

\begin{parag}\label{parag62}
The basic definitions and results for logarithmic schemes extend to formal logarithmic schemes. In particular, we have the notions of prelogarithmic and logarithmic structures on formal schemes, logarithmic structure associated to a prelogarithmic structure, as well as the notions of pullback structures, strict morphisms and charts (see \cite{Kat89}).
\end{parag}

\begin{definition}
An \emph{adic} (resp. \emph{$p$-adic}) \emph{logarithmic formal scheme} is a logarithmic formal scheme whose underlying formal scheme is adic (\cite{Ahmed2010} 2.1.24) (resp. $p$-adic i.e. an adic formal scheme such that the ideal $(p)$ generated by $p$ is an ideal of definition).
\end{definition}

\begin{parag}
Let $P$ be a monoid. We denote by $\Z_p\langle P\rangle$ the $p$-adic completion of the ring $\Z[P];$
\begin{equation}\label{BM}
\Z_p\langle P\rangle=\lim\limits_{\substack{\longleftarrow\\ n\ge 1}}\left (\Z/p^n\Z\right )[P].
\end{equation}
We denote by $B\langle P \rangle$ the $p$-adic formal scheme $\op{Spf}(\Z_p\langle P\rangle)$ equipped with the logarithmic structure induced by the canonical morphism $P\ra \Z_p\langle P\rangle.$ 
The following lemma is a variant of (\cite{Ogus2018} III 1.2.4) for $p$-adic logarithmic formal schemes:
\end{parag}

\begin{lemma}\label{Lemlog18}
Let $\frakX$ be a logarithmic $p$-adic formal scheme and $P$ a monoid. Then, we have a canonical bijection
$$\op{Hom}(\frakX,B\langle P \rangle)\stackrel{\sim}{\rightarrow} \op{Hom}_{\boldsymbol{\op{Mon}}}(P,\Gamma(\frakX,\calM_{\frakX})).$$
\end{lemma}

\begin{proof}
Consider the morphism $\alpha:\op{Hom}(\frakX,B\langle P \rangle)\ra \op{Hom}_{\boldsymbol{\op{Mon}}}(P,\Gamma(\frakX,\calM_{\frakX}))$ defined by composition with the canonical morphism
$P\ra \Gamma(B\langle P \rangle,\calM_{B\langle P \rangle}).$
Conversely, we define a morphism $\beta:\op{Hom}_{\boldsymbol{\op{Mon}}}(P,\Gamma(\frakX,\calM_{\frakX}))\ra \op{Hom}(\frakX,B\langle P \rangle)$ as follows : if $\theta:P\ra \Gamma(\frakX,\calM_{\frakX})$ is a morphism of monoids, we consider a covering $\frakX=\bigcup_{i\in I}\op{Spf}A_i$ of $\frakX$ by affine formal schemes, where $A_i$ is a complete $\Z_p$-algebra. For any $i\in I,$ let $\theta_i$ be the composition
$P\ra \Gamma(\frakX,\calM_{\frakX})\ra \Gamma(\op{Spf}A_i,\Ox_{\frakX})=A_i.$
Since $A_i$ is a complete $\Z_p$-algebra, the morphism $\theta_i$ induces a continuous morphism $\Z_p\langle P\rangle\ra A_i$ and so a morphism of formal schemes $f_i:\op{Spf}A_i\ra B\langle P \rangle.$ The morphism $P\xrightarrow{\theta} \Gamma(\frakX,\calM_{\frakX})\ra \Gamma(\op{Spf}A_i,\calM_{\frakX})$ induces a morphism $f_i^{\flat}:\calM_{B\langle P \rangle}\ra f_{i*}\calM_{\op{Spf}A_i}.$ The commutativity of the diagram
$$\begin{tikzcd}
P\ar{r}\ar{d} & \Gamma(\op{Spf}A_i,\calM_{\frakX})\ar{d} \\
\Z_p\langle P\rangle \ar{r} & A_i
\end{tikzcd}$$ 
implies that $(f_i,f_i^{\flat})$ is a morphism of logarithmic formal schemes. These morphisms glue into a morphism $\beta(\theta):\frakX\ra B\langle P \rangle$ and this morphism is independant of the choice of the covering. The maps $\alpha$ and $\beta$ are inverse to each other. 
\end{proof}

\begin{parag}\label{parag63}
We say that a logarithmic $p$-adic formal scheme $\frakX$ is \emph{quasi-coherent} (resp. \emph{coherent}, resp. \emph{fine}, resp. \emph{finitely-generated and saturated}, or \emph{fs} for short) if étale locally on $\frakX,$ there exists a strict morphism $\frakX\ra B\langle P \rangle$ for a monoid (resp. a finitely generated monoid, resp. a fine monoid, resp. an fs monoid) $P.$ We denote by $\boldsymbol{\op{LFS}}$ the category of fs logarithmic $p$-adic formal schemes.
\end{parag}

\begin{parag}\label{parag27}
Let $\frakX\ra \frakS$ and $\frakY\ra \frakS$ be morphisms in the category $\boldsymbol{\op{LFS}}.$ Following the conventions of \ref{parag42}, we denote by $\frakX\times_{\frakS}^{\op{log}}\frakY$ the fiber product of $\frakX \ra \frakS$ and $\frakY \ra \frakS$ in the category $\boldsymbol{\op{LFS}}$ and we keep the notation $\frakX \times_{\frakS} \frakY$ for the fiber product in the category of logarithmic formal schemes. The two are compared as follows: locally, there exist charts $\frakX\ra B\langle M \rangle,$ $\frakY\ra B\langle N \rangle$ and $\frakS\ra B\langle P \rangle$ for fs monoids $M,\ N$ and $P$ (\ref{parag63}), that fit into a commutative diagram (\cite{Ogus2018} III 1.2.7.3)
$$\begin{tikzcd}
\frakX\ar{r}\ar{d} & \frakS\ar{d} & \frakY\ar{l}\ar{d} \\
B\langle M \rangle\ar{r} & B\langle P \rangle & B\langle N \rangle\ar{l}
\end{tikzcd}$$ 
Let $Q=M\oplus_PN$ and $Q^{sat}$ its saturation. The commutative diagram above induces the following commutative diagram of monoids:
$$\begin{tikzcd}
P\ar{r}\ar{d} & M\ar{d} \\
N\ar{d} & \Gamma(\frakX,\calM_{\frakX})\ar{d} \\
\Gamma(\frakY,\calM_{\frakY})\ar{r} & \Gamma(\frakX\times_{\frakS}\frakY,\calM_{\frakX\times_{\frakS}\frakY})
\end{tikzcd}$$
and so we get a morphism $M\oplus_PN\ra \Gamma(\frakX\times_{\frakS}\frakY,\calM_{\frakX\times_{\frakS}\frakY}),$ which, by \ref{Lemlog18}, is equivalent to a morphism $B\langle M\oplus_PN\rangle \ra \frakX\times_{\frakS}\frakY.$ By definition of the logarithmic structure on a fiber product (\cite{Ogus2018} III 2.1), this morphism is a chart. 
Then the underlying formal scheme of $\frakX\times_{\frakS}^{\op{log}}\frakY$ is $(\frakX\times_{\frakS}\frakY)\times_{B\langle M\oplus_PN \rangle}B\langle Q^{sat} \rangle$ and its logarithmic structure is the pullback of that of $B\langle Q^{sat} \rangle.$
\end{parag}

\begin{parag}\label{parag17}
Let $(\frakX_n,\Ox_{\frakX_n},\calM_{\frakX_n})_{n\ge 1}$ be an inductive system of logarithmic schemes such that for every $n\ge 1,$ $\frakX_n$ is a $\Z/p^n\Z$-scheme, the diagram
$$\begin{tikzcd}
\frakX_n\ar{r}\ar{d}\ar{r} & \Sp \Z/p^n\Z \ar{d}\\
\frakX_{n+1}\ar{r} & \Sp \Z/p^{n+1}\Z
\end{tikzcd}$$
is cartesian in the category of schemes and $\frakX_n\ra\frakX_{n+1}$ is strict for all $n\ge 1.$  
Then the inductive limit $(\frakX,\Ox_{\frakX})$ of $((\frakX_n,\Ox_{\frakX_n}))_{n\ge 1}$ is a $p$-adic formal scheme (\cite{Ahmed2010} 2.1.24). Let $f_n:\frakX_n\ra \frakX$ be the canonical morphisms. The morphisms $f_n$ are homeomorphisms of the underlying topological spaces and so we use them to identify the topological spaces $\frakX_n$ and $\frakX.$ Let $\calM_{\frakX}=\lim\limits_{\longleftarrow}\calM_{\frakX_n}.$ Note that the projective limit commutes with the forgetful functor to the category of sheaves of sets (\cite{Ogus2018} I 1.1). Let $\alpha_{\frakX}:\calM_{\frakX}\ra \Ox_{\frakX}$ be the morphism induced by the morphisms $\alpha_{\frakX_n}.$ Since $\Ox_{\frakX}^*=\lim\limits_{\longleftarrow}\Ox_{\frakX_n}^*,$ the morphism $\alpha_{\frakX}:\calM_{\frakX}\ra \Ox_{\frakX}$ defines a logarithmic structure on $\frakX.$ Then $\frakX$ is a logarithmic $p$-adic formal scheme and $(\frakX,\calM_{\frakX})$ is the inductive limit of $(\frakX_n,\calM_{\frakX_n})$ in the category logarithmic formal schemes. 
Conversely, we have the following proposition:
\end{parag}

\begin{proposition}\label{propkey}
Let $\frakX$ be an integral logarithmic $p$-adic formal scheme and denote, for any integer $n\ge 1,$ by $\frakX_n$ the scheme $(\frakX,\Ox_{\frakX}/(p^n))$ and by $f_n:\frakX_n\ra \frakX$ the canonical morphism. Equip $\frakX_n$ with the logarithmic structure $\calM_{\frakX_n}=f_n^*\calM_{\frakX}.$ Then the canonical morphism
\begin{equation}\label{canmo1}
\varphi:\calM_{\frakX}\ra \lim_{\substack{\longleftarrow\\ n\ge 1}}\calM_{\frakX_n},\ m\mapsto (f_n^{\flat}(m))_{n\ge 1}
\end{equation}
is an isomorphism.
\end{proposition}

\begin{proof}
Let $\alpha_n$ be the composition $\calM_{\frakX}\xrightarrow{\alpha_{\frakX}} \Ox_{\frakX}\ra \Ox_{\frakX_n},$ where $\Ox_{\frakX}\ra \Ox_{\frakX_n}$ is the reduction modulo $p^n.$
By the definition of pullback of logarithmic structures (\cite{Ogus2018} III 1.1.5.2),
\begin{equation}
\calM_{\frakX_n}=f_n^*\calM_{\frakX}=\calM_{\frakX}\oplus_{\alpha_n^{-1}(\Ox_{\frakX_n}^*)}\Ox_{\frakX_n}^*.
\end{equation}
By this equality, the morphism (\ref{canmo1}) is given by
\begin{equation}
\varphi(m)=((m,1))_{n\ge 1}.
\end{equation}
Let $m$ and $m'$ be local sections of $\calM_{\frakX}$ such that $\varphi(m)=\varphi(m').$ For every positive integer $n,$ $(m,1)=(m',1)$ in $\calM_{\frakX_n}=\calM_{\frakX}\oplus_{\alpha_n^{-1}(\Ox_{\frakX}^*)}\Ox_{\frakX_n}^*.$ So, by (\cite{Kat89} (1.3)), there exists local sections $t_n,t_n'$ of $\alpha_n^{-1}(\Ox^*_{\frakX})$ such that
\begin{equation}\label{amsum}
\begin{cases}m+t_n=m'+t_n' \\ \alpha_n(t_n)=\alpha_n(t_n').\end{cases}
\end{equation} 
The local sections $\alpha_{\frakX}(t_n)$ and $\alpha_{\frakX}(t_n')$ are invertible modulo $p^n$ and so they are both invertible in $\Ox_{\frakX}.$ Since the morphism $\alpha_{\frakX}^{-1}(\Ox_{\frakX}^*)\rightarrow\Ox_{\frakX}^*$ induced by $\alpha_{\frakX}$ is an isomorphism, $t_n$ and $t_n'$ are invertible in $\calM_{\frakX}.$ Furthermore, $m+t_n-t_n'=m'$ and $\calM_{\frakX}$ is integral, so $t_n-t_n'=t_k-t'_k$ for all positive integers $n$ and $k.$ Set $d=t_1-t'_1.$ By (\ref{amsum}), $\alpha_{\frakX}(d)=1$ in $\Ox_{\frakX_n}$ for every positive integer $n$ and so $\alpha_{\frakX}(d)=1$ in $\Ox_{\frakX},$ $d=0$ and $t_1=t_1'.$ Again, since $\calM_{\frakX}$ is integral and by (\ref{amsum}), $m=m'.$ Hence (\ref{canmo1}) is injective.

For the surjectivity, let $\frakU=\op{Spf}A$ be an affine open formal subscheme of $\frakX$ and $(m_n)_{n\ge 1}$ a sequence of sections of $\Gamma \left ( \frakU, \calM_{\frakX} \right )$ and, for every $n\ge 1,$ $a_n\in A/(p^n)$ such that $((m_n,a_n))_{n\ge 1}$ yields a local section of $\lim\limits_{\substack{\longleftarrow\\ n\ge 1}}\calM_{\frakX_n}$ over $\frakU.$ For every positive integer $n\ge 1,$ let $b_n\in A$ such that, modulo $p^n,$ $b_n=a_n.$ First, note that $a_1$ is invertible in $A/(p)$ so $b_1$ is invertible in $A.$ Then there exists $t_1\in \Gamma \left (\frakU,\calM^*_{\frakX}\right )$ such that $\alpha_{\frakX}(t_1)=b_1.$ Setting $m=m_1+t_1,$ we get by (\cite{Kat89} (1.3))
\begin{equation}\label{eq683}
(m_1,a_1)=(m_1,\alpha_1(t_1))=(m,1).
\end{equation}
We now construct by induction on $n$ a sequence $(c_n)_{n\ge 1} \in \lim\limits_{\longleftarrow}\left (A/(p^n)\right )^*=A^*$ such that, in $\Gamma \left ( \frakU,\calM_{\frakX_n}\right ),$
$(m_n,a_n)=(m,c_n)\ \forall n\ge 1.$
By (\ref{eq683}), we set $c_1=1.$ Fix $n\ge 1$ and suppose $c_n$ is defined. Let $\ov{x} \ra \frakX_n$ be a geometric point. In $\calM_{\frakX_n,\ov{x}}=\calM_{\frakX,\ov{x}}\oplus_{\alpha_n^{-1}(\Ox_{\frakX_n}^*)_{\ov{x}}}\Ox_{\frakX_n,\ov{x}}^*,$ we have $(m_{n+1},a_{n+1})=(m_n,a_n)=(m,c_n).$
Then, again by (\cite{Kat89} (1.3)), there exist a neighborhood $V$ of $\ov{x}$ and local sections $t,t' \in \Gamma \left ( V,\alpha_n^{-1}(\Ox_{\frakX_n}^*) \right )$ over $V$ such that
$$
\begin{cases}
m_{n+1}+t=m+t' \\
a_{n+1}\alpha_n(t')=c_n\alpha_n(t)\ (\op{mod} p^n).
\end{cases}
$$
We set $c_{n+1}=a_{n+1}\alpha_{n+1}(t'-t).$ Note that $t'-t$ is independant of the choice of $t$ and $t'$ and we can hence suppose that $t'-t \in \Gamma\left (\frakU,\calM_{\frakX_n} \right ).$ Then
$$(m_{n+1},a_{n+1})=(m+t'-t,a_{n+1})=(m,a_{n+1}\alpha_{n+1}(t'-t))=(m,c_{n+1}).$$
The compatibility condition follows from the following equality modulo $p^n:$
$$c_{n+1}=a_{n+1}\alpha_n(t'-t)=c_n.$$
Let $c=(c_n)_{n\ge 1} \in \lim\limits_{\longleftarrow} \left (A/(p^n)\right )^*=A^*$ and $s\in \Gamma\left (\frakU,\calM_{\frakX} \right )$ such that $\alpha_{\frakX}(s)=c.$ Then $((m_n,a_n))_{n\ge 1}$ is the image of $m+s$ by $\varphi$ \eqref{canmo1}.
\end{proof}

\begin{corollaire}\label{corkey}
Under the same hypothesis of \ref{propkey}, the logarithmic formal scheme $\frakX$ is the inductive limit, in the category of logarithmic formal schemes, of $(\frakX_n)_{n\ge 1}.$
\end{corollaire}

\begin{parag}\label{parag68}
Let $f:\frakX\ra \frakY$ be a morphism of logarithmic $p$-adic formal schemes. Denote, for any integer $m\ge n\ge 1,$ by $\frakX_n$ (resp. $\frakY_n$) the scheme $(\frakX,\Ox_{\frakX}/p^n)$ (resp. $(\frakY,\Ox_{\frakY}/p^n)$) and by $i_n:\frakX_n\ra \frakX,$ $j_n:\frakY_n\ra \frakY,$ $i_{n,m}:\frakX_n \ra \frakX_m$ and $j_{n,m}:\frakY_n \ra \frakY_m$ the canonical morphisms. Equip $\frakX_n$ (resp. $\frakY_n$) with the logarithmic structure $i_n^*\calM_{\frakX}$ (resp. $j_n^*\calM_{\frakY}$). By (\cite{Ahmed2010}, 2.2.1), the morphism $f$ induces an inductive system of morphisms of schemes $(f_n:\frakX_n\ra \frakY_n).$ Then, by the commutativity of the diagram
$$\begin{tikzcd}
\frakX_n\ar{r}{f_n}\ar{d}{i_n} & \frakY_n\ar{d}{j_n}\\
\frakX\ar{r}{f} & \frakY
\end{tikzcd}$$ 
and since $i_n$ and $j_n$ are strict, the morphism $f^{\flat}:f^*\calM_{\frakY}\ra \calM_{\frakX}$ induces a morphism $f_n^*\calM_{\frakY_n}\ra \calM_{\frakX_n}$ for all $n\ge 1,$ making $f_n$ into a morphism of logarithmic schemes. Similarly, since $i_{n,m}^*\calM_{\frakX_m}=i_n^*\calM_{\frakX}$ (resp. $j_{n,m}^*\calM_{\frakY_m}=j_n^*\calM_{\frakY}$), the morphism $i_n^{\flat}$ (resp. $j_n^{\flat}$) turns $i_{n,m}$ (resp. $j_{n,m}$) into a morphism of logarithmic schemes.
\end{parag}

\begin{proposition}\label{prop18}
Keep the same hypothesis of \ref{parag68}. The map $f\mapsto (f_n)_{n\ge 1}$ defined in \ref{parag68} is a bijection between the set of morphisms of logarithmic formal schemes $f:\frakX\ra \frakY$ and the set of sequences of morphisms of logarithmic schemes $(f_n:\frakX_n\ra \frakY_n)$ making the following diagram commutative for all $n\ge m$
$$
\begin{tikzcd}
\frakX_m\ar{r}{f_m}\ar{d} & \frakY_m\ar{d}\\
\frakX_n\ar{r}{f_n} & \frakY_n
\end{tikzcd}$$
\end{proposition}

\begin{proof}
Let $(f_n:\frakX_n \ra \frakY_n)_{n\ge 1}$ a sequence of morphisms of logarithmic schemes such that
$$\begin{tikzcd}
\frakX_n\ar{r}{f_n}\ar{d}{i_{n,m}} & \frakY_n\ar{d}{j_{n,m}}\\
\frakX_m\ar{r}{f_m} & \frakY_m
\end{tikzcd}$$ 
is commutative for every integers $m\ge n\ge 1.$
By (\cite{Ahmed2010} 2.2.2), the morphisms $f_n$ induce a morphism of formal schemes. In addition, by \ref{propkey}, $\calM_{\frakY}=\lim\limits_{\longleftarrow} \calM_{\frakY_n}$ and $\calM_{\frakX}=\lim\limits_{\longleftarrow} \calM_{\frakX_n}.$ Then the compositions
$f^{-1}\calM_{\frakY} \ra f^{-1}\calM_{\frakY_n} \xrightarrow {f_n^{\flat}} \calM_{\frakX_n}$
induce a morphism $f^{-1}\calM_{\frakY} \ra \calM_{\frakX}$ which turns $f$ into a morphism of logarithmic formal schemes. This construction is inverse to that of \ref{parag68}.
\end{proof}

\begin{parag}\label{parag69}
Keep the same notations of \ref{parag68}.
For every integer $n\ge 1,$ denote by $i_{n,n+1}:\frakX_n\ra \frakX_{n+1}$ the canonical morphism and by $(d_n,\delta_n):\Ox_{\frakX_n}\times \calM_{\frakX_n}\ra \omega^1_{\frakX_n/\frakY_n}$
the universal derivation. Then the morphisms
$\Ox_{\frakX_{n+1}}\xrightarrow{i_{n,n+1}^{\#}}\Ox_{\frakX_n}\xrightarrow{d_n}\omega^1_{\frakX_n / \frakY_n}$
and
$\calM_{\frakX_{n+1}}\xrightarrow{i_{n,n+1}^{\flat}}\calM_{\frakX_n}\xrightarrow{\delta_n}\omega^1_{\frakX_n / \frakY_n}$
define a logarithmic derivation $\Ox_{\frakX_{n+1}}\times\calM_{\frakX_{n+1}}\ra \omega^1_{\frakX_n / \frakY_n}$ and thus a morphism
$\omega^1_{\frakX_{n+1} / \frakY_{n+1}}\ra \omega^1_{\frakX_n / \frakY_n}.$
These morphisms define a projective system $(\omega^1_{\frakX_n/\frakY_n})_{n\ge 1}.$ We set
\begin{equation}\label{eq681}
\omega^1_{\frakX/\frakY}=\lim_{\substack{\longleftarrow \\ n\ge 1}}\omega^1_{\frakX_n / \frakY_n},
\end{equation}
that we call \emph{the sheaf of logarithmic differentials of $\frakX$ over $\frakY$}.
\end{parag}

\begin{proposition}\label{prop69}
Let $\frakX$ be a logarithmic $p$-adic formal scheme and denote, for any integer $n\ge 1,$ by $\frakX_n$ the scheme $(\frakX,\Ox_{\frakX}/p^n)$ and, for any $1\le n\le k,$ denote by $f_n:\frakX_n\ra \frakX$ and $f_{n,k}:\frakX_{n}\ra \frakX_{k}$ the canonical morphisms. Equip $\frakX_n$ with the logarithmic structure $f_n^*\calM_{\frakX}.$ Then for any $1\le n\le k,$ the morphisms $f_n^{\flat}:\calM_{\frakX}\ra \calM_{\frakX_n}$ and $f_{n,k}^{\flat}:\calM_{\frakX_k}\ra \calM_{\frakX_n}$ are surjective.
\end{proposition}

\begin{proof}
Let $k\ge n\ge 1.$ Since $f_n^{\flat}=f_{n,k}^{\flat}\circ f_k^{\flat},$ it is sufficient to prove that $f_n^{\flat}$ is surjective. By definition of the logarithmic structure on $\frakX_n,$ the morphism $f_{n}$ is strict so $$\calM_{\frakX_n}=f_{n}^*\calM_{\frakX}=\calM_{\frakX}\oplus_{\alpha_n^{-1}(\Ox_{\frakX_n}^*)}\Ox_{\frakX_n}^*,$$
where $\alpha_n:\calM_{\frakX}\ra \Ox_{\frakX_n}$ is the composition
$\calM_{\frakX}\xrightarrow{\alpha_{\frakX}}\Ox_{\frakX}\xrightarrow{f_{n}^{\#}}\Ox_{\frakX_{n}}.$ Let $m'$ and $a_n$ be local sections of $\calM_{\frakX}$ and $\Ox_{\frakX_{n}}^*$ respectively. Set $m_n=(m',a_n),$ local section of $\calM_{\frakX_{n}}.$ Let $a$ be a local lifting of $a_n$ to a local section of $\Ox_{\frakX}$ over an affine formal subscheme of $\frakX.$ Since $\Ox_{\frakX}=\lim\limits_{\substack{\longleftarrow\\ n\ge 1}}\Ox_{\frakX_{n}}$ and $a_n$ is invertible in $\Ox_{\frakX_n},$ it follows that $a$ is also invertible in $\Ox_{\frakX}$ and so $a=\alpha_{\frakX}(t)$ for a local section $t$ of $\calM_{\frakX}.$ It follows that
$$m_n=(m',a_n)=(m',\alpha_{n}(t))=(t+m',1)=f_{n}^{\flat}(t+m').$$
\end{proof}

\begin{parag}\label{Khamineipar414}
Let $\Delta:\frakX \ra \frakY$ be an exact closed immersion of integral logarithmic $p$-adic locally Noetherian formal schemes with ideal $\calI.$ For any $x\in \frakX,$ the ideal $\calI_x$ is a proper ideal of $\Ox_{\frakY,x}$ so $1+\calI_x \subset \Ox_{\frakY,x}^*.$ It follows that $\Delta^{-1}(1+\calI)\subset \Delta^{-1}\Ox_{\frakY}^*.$ Denote by $\lambda:\Delta^{-1}(1+\calI) \ra \Delta^{-1}\calM_{\frakY}$ the composition of $\Delta^{-1}\alpha_{\frakY}^{-1}:\Delta^{-1}\Ox_{\frakY}^* \ra \Delta^{-1}\calM_{\frakY}$ with the canonical embedding $\Delta^{-1}(1+\calI) \hookrightarrow \Delta^{-1}\Ox_{\frakY}^*.$ Consider the sequence of monoids
\begin{equation}\label{era2exactseq}
0 \ra \Delta^{-1}(1+\calI) \xrightarrow{\lambda} \Delta^{-1}\calM_{\frakY} \xrightarrow{\Delta^{\flat}} \calM_{\frakX} \ra 0.
\end{equation}
\end{parag}

\begin{proposition}
Keep the assumptions of \ref{Khamineipar414}. The sequence \eqref{era2exactseq} is an exact sequence of monoids i.e. $\lambda$ is injective, $\Delta^{\flat}$ is surjective and for any local sections $m$ and $m'$ of $\Delta^{-1}\calM_{\frakY}$ such that $\Delta^{\flat}(m)=\Delta^{\flat}(m'),$ there exists a local section $a$ of $\Delta^{-1}(1+\calI)$ such that
$
m+\lambda(a)=m'.
$
\end{proposition}

\begin{proof}
The morphism $\lambda$ (resp. $\Delta^{\flat}$) is injective (resp. surjective) by definition.
For a positive integer $n,$ let $\Delta_n:\frakX_n \ra \frakY_n$ be the immersion obtained from $\Delta$ by reduction modulo $p^n,$ as in \ref{prop69} and denote by $\calI_n$ its ideal. By \ref{propkey}, it is sufficient to prove that the sequence
\begin{equation}\label{era2exactseqprime}
0 \ra \Delta_n^{-1}(1+\calI_n) \xrightarrow{\lambda_n} \Delta_n^{-1}\calM_{\frakY_n} \xrightarrow{\Delta_n^{\flat}} \calM_{\frakX_n}
\end{equation}
is exact for all positive integers $n.$
This follows from \ref{KhamineiExact}.
\end{proof}

\begin{definition}\label{erafsmoothdef}
A morphism $f:\frakX\ra \frakY$ of logarithmic $p$-adic formal schemes is said to be \emph{log étale} (resp. \emph{log smooth}) if $f$ is locally of finite presentation (resp. locally of finite type) (\cite{Ahmed2010} 2.3.13 and 2.3.15), étale locally on $\frakX$ and $\frakY,$ there exists a chart $\theta:P\ra Q$ of $f$ such that
\begin{enumerate}
\item $\theta^{gp}$ is injective and $\op{coker}\theta^{gp}$ (resp. the torsion subgroup of $\op{coker}\theta^{gp}$) is finite and of order coprime with $p.$
\item The morphism $\frakX \ra \frakY\times_{B\langle P\rangle}B\langle Q\rangle,$ induced by $f$ and the chart $\frakX \ra B\langle Q\rangle$ is étale (\cite{Ahmed2010} 2.4.5).
\end{enumerate}
It is clear that if $f$ is log étale (resp. log smooth), then, for every integer $n\ge 1,$ the induced morphism of logarithmic schemes $f_n:\frakX_n\ra \frakY_n$ (\ref{parag68}) is log étale (resp. log smooth).
\end{definition}

\begin{proposition}\label{proplift}
Let $f:\frakX \ra \frakY$ be a morphism of integral logarithmic $p$-adic formal schemes. Suppose we are given a commutative diagram
\begin{equation}\label{diag6161}
\begin{tikzcd}
 & & \frakX \ar{d}{f} \\
T \ar{r}\ar[bend right=-30]{urr} & \frakT \ar{r} \ar[dashed]{ur} & \frakY
\end{tikzcd}
\end{equation}
where $\frakT$ is a logarithmic $p$-adic formal scheme, $T$ is the special fiber of $\frakT$ equipped with the pullback structure of $\frakT$ and $T \ra \frakT$ is the canonical strict morphism. If $f$ is log étale (resp. log smooth) then the dashed arrow exists and is unique (resp. exists locally on $\frakT$).
\end{proposition}

\begin{proof}
For every integer $n\ge 1,$ we denote by $\frakX_n$ (resp. $\frakY_n,$ resp. $\frakT_n$) the logarithmic scheme obtained from $\frakX$ (resp. $\frakY,$ resp. $\frakT$) by reduction modulo $p^n,$ as in \ref{propkey}.
Suppose $f$ is log étale. We construct, by induction on $n,$ a sequence of compatible morphisms $(g_n:\frakT_n\ra \frakX_n).$ For $n=1,$ we take the morphism $g_1:T \ra \frakX_1$ given in (\ref{diag6161}). Let $n$ be a positive integer and suppose we are given a morphism $g_n:\frakT_n \ra \frakX_n$ such that the diagram
\begin{equation}
\begin{tikzcd}
 & & \frakX_n \ar{d}{f_n} \\
T \ar{r}\ar[bend right=-30]{urr} & \frakT_n \ar{r} \ar{ur}{g_n} & \frakY_n
\end{tikzcd}
\end{equation}
is commutative. We obtain the following solid commutative diagram
$$
\begin{tikzcd}
\frakX_n \ar{rr} & & \frakX_{n+1} \ar{d}{f_{n+1}} \\
\frakT_n \ar{u}{g_n} \ar{r} & \frakT_{n+1} \ar[dashed]{ur}{g_{n+1}} \ar{r} & \frakY_{n+1}
\end{tikzcd}
$$
The existence and uniqueness of the dashed arrow $g_{n+1}:\frakT_{n+1} \ra \frakX_{n+1}$ follows from the fact that $f_{n+1}$ is log étale and $\frakT_n \ra \frakT_{n+1}$ is a strict thickening. The morphisms $g_n$ are compatible by construction, so they induce a morphism
$g:\lim\limits_{\substack{\longrightarrow \\ n\ge 1}}\frakT_n \ra \lim\limits_{\substack{\longrightarrow \\ n\ge 1}}\frakX_n.$
And by \ref{propkey}, $\lim\limits_{\substack{\longrightarrow \\ n\ge 1}}\frakT_n=\frakT,\ \lim\limits_{\substack{\longrightarrow \\ n\ge 1}}\frakX_n=\frakX.$
The proof in the log étale case is complete.
The log smooth case can be proved in a similar way by restricting to affine formal subschemes of $\frakY,$ $\frakX$ and $\frakT.$ 
\end{proof}

\begin{definition}\label{logflatdef}
Let $f:\frakX \ra \frakY$ be a morphism of fine logarithmic $p$-adic formal schemes and, for all positive integers $n,$ $f_n:\frakX_n \ra \frakY_n$ the morphism obtained from $f$ by reduction modulo $p^n$ as in \ref{parag68}. The morphism $f$ is said to be \emph{log flat} if, for every positive integer $n,$ $f_n$ is log flat (\cite{Ogus2018} IV 4.1.1).
\end{definition}

\begin{proposition}\label{Wflat}
Let $W(k)$ be the ring of Witt vectors of a perfect field $k$ of characteristic $p$ and $\frakS=\op{Spf}W$ equipped with the trivial logarithmic structure. If $\frakT$ is an fs logarithmic $p$-adic formal scheme log flat over $\frakS,$ then $\frakT$ is flat over $\frakS$ \eqref{dxuflat}. 
\end{proposition}

\begin{proof}
For any positive integer $n,$ $\frakT_n \ra \frakS_n$ is log flat. Let $n$ be a positive integer. Fppf locally on $\frakT_n,$ there exists a chart $\frakT_n \ra A_n[ M]$ such that $\frakT_n \ra \frakS_n\times_{\op{Spec}(\Z/p^n\Z)} B[ M]=\op{Spec}(W(k)/(p^n)[M])$ is flat. Since $W(k)/(p^n) \ra W(k)/(p^n)[ M]$ is flat, we get that $\frakT_n$ is flat over $\frakS_n.$ It follows that $\frakT$ is flat over $\frakS.$
\end{proof}

\begin{proposition}\label{logflatfiber}
Let $f:\frakX \ra \frakY$ and $g:\frakY \ra \frakZ$ be two morphisms of fine logarithmic $p$-adic formal schemes, $x\in \frakX,$ and $z=g\circ f(x)$ such that $f$ and $g$ are locally of finite presentation (\cite{Ahmed2010} 2.3.15) and $g\circ f$ and $f_z:\frakX_z\ra \frakY_z$ are log flat in a neighborhood of $x.$ Then $f$ is log flat in a neighborhood of $x.$
\end{proposition}

\begin{proof}
This is a consequence of (\cite{Ogus2018} IV 4.2.2).
\end{proof}

\begin{lemma}\label{era3proplogstrformal}
Let $f:\frakX\ra \frakY$ and $g:\frakZ\ra \frakY$ be morphisms of logarithmic formal schemes such that $g$ is strict. Suppose that there exists a morphism of formal schemes $h:\frakX \ra \frakZ$ such that $f=g\circ h$ in the category of formal schemes. Then there exists a unique morphism of sheaves of monoids $h^{\flat}:h^{-1}\calM_{\frakZ} \ra \calM_{\frakX}$ making $h$ a morphism of logarithmic formal schemes and such that $f=g\circ h$ in the category of logarithmic formal schemes.
\end{lemma}

\begin{proof}
The proof is similar to \ref{era3proplogstr}.
\end{proof}

\section{Frames on logarithmic formal schemes}
\begin{parag}\label{parag31}
In this section, we generalize the notion of frames, introduced by Kato and Saito in \cite{Saito04}, to fs logarithmic formal schemes. Let $P$ be an fs monoid. Recall (\ref{parag43}) that we have defined a presheaf of sets $[P]$ on the category $\boldsymbol{\op{L}}$ of fs logarithmic schemes by
$$[P](T)=\op{Hom}_{\boldsymbol{\op{Mon}}}(P,\Gamma(T,\ov{\calM}_{T})).$$
We consider $\boldsymbol{\op{L}}$ as a full subcategory of the category of fs logarithmic formal schemes in the canonical way (\cite{SP} \href{https://stacks.math.columbia.edu/tag/0AHY}{ 0AHY}).
By \ref{Lemlog18}, we have a canonical morphism of presheaves $B\langle P \rangle\ra [P].$ In this section, all logarithmic formal schemes are supposed to be fs. We start with the following proposition:
\end{parag}

\begin{proposition}
The functor
\begin{equation}
F:\boldsymbol{\op{LFS}} \ra \boldsymbol{\widehat{\op{L}}},\ \frakX \mapsto \op{Hom}(-,\frakX)
\end{equation}
is fully faithful.
\end{proposition}

\begin{proof}
Let $f,g:\frakX \ra \frakY$ be morphisms in $\boldsymbol{\op{LFS}}$ such that $F(f)=F(g)$ and consider the notations of \ref{parag68}. For any integer $n\ge 1,$ the equality $f\circ i_n=F(f)(i_n)=F(g)(i_n)=g\circ i_n$ is equivalent to $j_n \circ f_n=j_n \circ g_n.$ Since $i_n:\frakX_n \ra \frakX$ is equal to the identity on the underlying topological spaces, $f=g$ on the underlying topological spaces. On the level of structural sheaves, the compositions $$\Ox_{\frakY} \xrightarrow{j_n^{\#}} \Ox_{\frakY_n} \xrightarrow{f_n^{\#}} f_*\Ox_{\frakX_n},\ \Ox_{\frakY} \xrightarrow{j_n^{\#}} \Ox_{\frakY_n} \xrightarrow{g_n^{\#}} g_*\Ox_{\frakX_n}$$
are equal. And since $j_n^{\#}$ is surjective, $f_n^{\#}=g_n^{\#}.$
On the level of monoids, the compositions
$$\calM_{\frakY}\xrightarrow{j_n^{\flat}}\calM_{\frakY_n}\xrightarrow{f_n^{\flat}}f_*\calM_{\frakX_n},\ \calM_{\frakY}\xrightarrow{j_n^{\flat}}\calM_{\frakY_n}\xrightarrow{g_n^{\flat}}g_*\calM_{\frakX_n}$$
are equal. And since $j_n^{\flat}$ is surjective (\ref{prop69}), $f_n^{\flat}=g_n^{\flat}.$ We conclude that $f_n=g_n$ as morphisms of logarithmic schemes and it follows by \ref{prop18} that $f=g.$ The functor $F$ is thus faithful.

Now let $\varphi: \op{Hom}(-,\frakX) \ra \op{Hom}(-,\frakY)$ be a morphism of presheaves on $\boldsymbol{\op{L}}.$ For every integer $n\ge 1,$ let $h_n=\varphi(i_n):\frakX_n \ra \frakY.$ The morphisms $h_n$ are compatible so, by \ref{corkey}, they induce a morphism $h:\frakX \ra \frakY$ of logarithmic formal schemes. Let us prove that $\varphi=F(h).$ Let $T$ be an fs logarithmic scheme and $f:T\ra \frakX$ be a morphism. 
Let $\op{Spf}A$ be a formal affine subset of $\frakX$ and $\op{Spec}B$ a formal affine subset of $T$ such that $f(\op{Spec}B)\subset \op{Spf}A.$ Let $u:A\ra B$ be the corresponding continuous morphism of adic rings. Since $p^n\longrightarrow 0$ in $A,$ $u(p)^n\longrightarrow 0$ in $B.$ But since $0$ is an open ideal in $B,$ it follows that $u(p)^n=0$ for some positive integer $n.$ Hence the morphism of formal schemes $\op{Spec}B \xrightarrow{f} \frakX$ factors through $\frakX_n.$ And since $i_n:\frakX_n \ra \frakX$ is strict, this factorization actually holds for logarithmic formal schemes.
Suppose there exists a covering $(T_i\xrightarrow{f_i} T)_{i\in I}$ such that, for every $i\in I,$
\begin{equation}\label{eq721}
\varphi_{T_i}(f\circ f_i)=h \circ f\circ f_i
\end{equation}
Considering the commutative diagram
$$
\begin{tikzcd}
\op{Hom}(T,\frakX) \ar{r}{\varphi_T} \ar{d} & \op{Hom}(T,\frakY) \ar{d} \\
\op{Hom}(T_i,\frakX) \ar{r}{\varphi_{T_i}} & \op{Hom}(T_i,\frakY)
\end{tikzcd}
$$
where the vertical arrows are induced by $f_i.$
We get $\varphi_{T_i}(f\circ f_i)=\varphi_T(f) \circ f_i.$ By the hypothesis (\ref{eq721}), $h \circ f\circ f_i=\varphi_T(f)\circ f_i,$ and since $(f_i)_{i\in I}$ is a covering, we conclude that $h\circ f=\varphi_T(f).$ It is thus sufficient to work locally on $T$ and we may suppose that there exists $f_n:T \ra \frakX_n$ such that $f$ is equal to the composition
$T\xrightarrow{f_n} \frakX_n \xrightarrow{i_n} \frakX.$
Consider the commutative diagram
$$
\begin{tikzcd}
\op{Hom}(\frakX_n,\frakX) \ar{r}{\varphi_{\frakX_n}} \ar{d} & \op{Hom}(\frakX_n,\frakY) \ar{d} \\
\op{Hom}(T,\frakX) \ar{r}{\varphi_{T}} & \op{Hom}(T,\frakY)
\end{tikzcd}
$$
where the vertical arrows are induced by $f_n.$ We get $\varphi_T(f)=h_n\circ f_n=h\circ f.$ We conclude that $F$ is full.
\end{proof}

\begin{definition}\label{def72}
Let $\frakX$ be an fs logarithmic $p$-adic formal scheme, $P$ an fs monoid and $\frakX\ra[P]$ a morphism of presheaves. We say that $\frakX\ra [P]$ is a \emph{frame} if, for any geometric point $\ov{x}\ra \frakX,$ there exists an étale neighborhood $\frakU$ of $\ov{x}$ such that $\frakU\ra [P]$ factors as a composition $\frakU\ra B\langle P\rangle \ra [P]$ with $\frakU\ra B\langle P\rangle$ strict and $B\langle P\rangle\ra [P]$ is the canonical morphism of presheaves (\ref{parag31}). An fs logarithmic $p$-adic formal scheme $\frakX$ equipped with a frame $\frakX\ra [P]$ will be denoted by $(\frakX,P)$ and called a \emph{framed logarithmic formal scheme}. A morphism $(\frakX,P)\ra (\frakS,Q)$ of framed logarithmic formal schemes is a pair of morphisms $(\frakX\ra \frakS,Q\ra P)$ such that the following diagram is commutative:
$$\begin{tikzcd}
\frakX\ar{r}\ar{d} & \frakS\ar{d}\\
\left [P\right ]\ar{r} & \left [Q\right ]
\end{tikzcd}$$
where $[P]\ra [Q]$ is the morphism induced by $Q\ra P.$ 
\end{definition}

\begin{parag}\label{parag73}
Let $(\frakX,Q)$ be a framed fs logarithmic formal scheme and $n\ge 1.$ Let $\frakX_n$ be the scheme $(\frakX,\Ox_{\frakX}/p^n)$ equipped with the logarithmic structure pullback of that of $\frakX.$ We get a frame on $\frakX_n$ by composing $\frakX\ra [Q]$ with the canonical morphism $\frakX_n\ra \frakX$ which is strict.
Let $\frakX\ra \frakS$ and $\frakY\ra \frakS$ be morphisms in $\boldsymbol{\op{LFS}},$ $Q\ra P$ a morphism of fs monoids and $\frakX\ra [Q]$ and $\frakY\ra [Q]$ morphisms of presheaves. We denote by $\frakX\times_{[Q]}^{\op{log}}[P]$ and $\frakX\times^{\op{log}}_{\frakS,[Q]}\frakY$ the presheaves of sets
\begin{equation}\label{eq731}
\frakX\times_{[Q]}^{\op{log}}[P]:\begin{array}[t]{clc}\boldsymbol{\op{L}} & \ra & \boldsymbol{\op{Sets}}\\ T & \mapsto & \frakX(T)\times_{[Q](T)}[P](T),\end{array}
\end{equation}
\begin{equation}\label{eq732}
\frakX\times^{\op{log}}_{\frakS,[Q]}\frakY:\begin{array}[t]{clc}\boldsymbol{\op{L}} & \ra & \boldsymbol{\op{Sets}}\\ T & \mapsto & \frakX(T)\times_{\frakS(T)\times [P](T)}\frakY(T).\end{array}
\end{equation}
We have the following proposition, which is analogous to (\cite{Saito04} 4.2.1):
\end{parag}

\begin{proposition}\label{formalframelift}
Let $f:(\frakX,Q) \ra (\frakS,P)$ be a morphism of framed fs logarithmic $p$-adic formal schemes. Etale locally on $\frakX$ and $\frakS,$ there exists a chart $B\langle M\rangle \ra B\langle P \rangle$ of $f$ lifting the diagram
$$
\begin{tikzcd}
\frakX \ar{r} \ar{d} & \frakS \ar{d} \\
\left [Q \right ] \ar{r} & \left [P \right ].
\end{tikzcd}
$$
\end{proposition}

\begin{proposition}\label{formaletaleframelift}
Equip $\op{Spf}\Z_p$ with the trivial logarithmic structure. Let $f:(\frakX,Q) \ra (\frakS,P)$ be a log smooth morphism of framed fs logarithmic $p$-adic formal schemes. We make one of the following two assumptions : either $\frakS$ is log flat over $\op{Spf}\Z_p,$ or $\frakS$ is flat over $\op{Spf} \Z_p$ and $f$ is integral (\cite{Ogus2018} III 2.5.1). Then, étale locally on $\frakX$ and $\frakS,$ there exists a chart $B\langle M \rangle \ra B\langle P \rangle$ of $f,$ fitting into a commutative diagram
$$
\begin{tikzcd}
\frakX \ar{r} \ar{d} & \frakS \ar{d} \\
B\langle M\rangle \ar{r} \ar{d} & B\langle P \rangle \ar{d} \\ 
\left [Q \right ] \ar{r} & \left [P\right ],
\end{tikzcd}
$$
and satisfying the following conditions:
\begin{enumerate}
\item $P^{gp} \ra M^{gp}$ is injective and the torsion subgroup of its cokernel is of finite order invertible in $\Ox_{\frakX}.$
\item The morphism $\frakX \ra \frakS\times_{B\langle P \rangle}B\langle M \rangle,$ induced by $f$ and the chart $\frakX \ra B\langle M\rangle,$ is étale and strict.
\end{enumerate}
\end{proposition}

\begin{proof}
Let $X$ and $S$ be the special fibers of $\frakX$ and $\frakS$ respectively and $f_1:X\ra S$ the morphism induced by $f.$ Let $\ov{x} \ra X$ be a geometric point. After shrinking $\frakS$ to an étale neighborhood of $f_1(\ov{x}),$ we can suppose that the frame $\frakS \ra [P]$ factors through a chart $\frakS \ra B\langle P \rangle.$ By \ref{etaleframelift}, there exists an étale neighborhood $U$ of $\ov{x}$ and a chart $P \ra M$ of the restriction $U\ra S$ of $f_1:X \ra S,$ fitting into a commutative diagram
$$
\begin{tikzcd}
U \ar{r}{f_1} \ar{d} & S \ar{d} \\
A_1[M] \ar{r} \ar{d} & A_1[P] \ar{d} \\
\left [Q \right ] \ar{r} & \left [ P \right ],
\end{tikzcd}
$$
such that the torsion subgroup of the cokernel of $P^{gp} \ra M^{gp}$ is of finite order invertible in $\Ox_{U}$ and the morphism $U \ra S\times_{A_1[P]}A_1[ M ],$ induced by $f_1$ and the chart $U\ra A_1[ M],$ is étale and strict. We can furthermore suppose that $U$ is affine.
Since $B\langle M\rangle \ra B\langle P \rangle$ is log smooth, the morphism $\frakS \times_{B\langle P \rangle}B\langle M\rangle \ra \frakS$ is log smooth modulo $p^n$ for every $n\ge 1.$ Let $\frakU \ra \frakX$ be an étale morphism such that $U=\frakU\times_{\frakX}X.$ There exists a morphism $\frakU\ra \frakS \times_{B\langle P \rangle} B\langle M\rangle$ fitting into the commutative diagram
$$
\begin{tikzcd}
S\times_{A_1[P]} A_1[M] \ar{rr} & & \frakS\times_{B\langle P\rangle} B\langle M\rangle \ar{d} \\
U\ar{u} \ar{r} & \frakU \ar{ur} \ar{r} & \frakS.
\end{tikzcd}
$$
Since the morphisms $U\ra \frakU,$ $U\ra S\times_{A_1[P]}A_1[M]$ and $S\times_{A_1[P]}A_1[M] \ra \frakS\times_{B\langle P\rangle}B\langle M\rangle$ are all strict, the morphism $\frakU \ra \frakS\times_{B\langle P \rangle}B\langle M\rangle$ is also strict. It remains to prove that the morphism $\frakU \ra \frakS\times_{B\langle P\rangle}B\langle M \rangle$ is étale. The morphism $B\langle M\rangle \ra B\langle P\rangle$ is log smooth hence log flat, thus so is $\frakS\times_{B\langle P\rangle} B\langle M\rangle \ra \frakS.$ Since $U \ra S\times_{A_1[P]}A_1[M]$ is étale, by (\cite{Ahmed2010} 2.4.11), there exists an étale strict morphism $\frakZ\ra \frakS\times_{B\langle P\rangle} B\langle M\rangle$ of $p$-adic formal schemes fitting into a cartesian diagram
\begin{equation}\label{diag771}
\begin{tikzcd}
U \ar{r} \ar{d} & \frakZ \ar{d} \\
S\times_{A_1[P]}A_1[M] \ar{r} & \frakS\times_{B\langle P\rangle} B\langle M \rangle.
\end{tikzcd}
\end{equation}
Then, by the étaleness of $\frakZ \ra \frakS\times_{B\langle P\rangle} B\langle M \rangle,$ there exists a unique morphism
$\frakU \ra \frakZ,$ which is automatically strict, fitting into the commutative diagram
$$
\begin{tikzcd}
 & & \frakZ \ar{d} \\
U \ar{r} \ar[bend right=-20]{urr} & \frakU \ar{ur} \ar{r} & \frakS\times_{B\langle P\rangle} B\langle M \rangle.
\end{tikzcd}
$$
By \eqref{diag771}, $\frakU \ra \frakZ$ is an isomorphism modulo $p.$
\begin{itemize}
\item If $\frakS$ is log flat over $\op{Spf}\Z_p$ then so is $\frakX.$ Since $\frakU$ and $\frakZ$ are log flat over $\frakX$ and $\frakS$ respectively, $\frakU$ and $\frakZ$ are log flat, hence flat \eqref{Wflat}, over $\op{Spf}\Z_p.$ 
\item If $\frakS$ is flat over $\op{Spf}\Z_p$ and $\frakX\ra \frakS$ is integral, then so is $U\ra S.$ Then, by (\cite{Ogus2018} III 2.5.3) and the commutative square
$$
\begin{tikzcd}
U \ar{r} \ar{d} & \frakZ \ar{d} \\
S\ar{r} & \frakS, 
\end{tikzcd}
$$
the canonical morphism $\frakZ \ra \frakS$ is integral. Since it is also log smooth and hence log flat, it follows that it is flat by (\cite{Ogus2018} IV 4.3.5). By the same result, $\frakX \ra \frakS$ is flat.
\end{itemize}
We have thus proven that, in both cases, $\frakU$ and $\frakZ$ are flat over $\op{Spf}\Z_p.$ We conclude by (\cite{DXU19} 7.2) that $\frakU \ra \frakZ$ is an isomorphism of formal schemes. Since it is clearly strict, it is an isomorphism of logarithmic formal schemes.
\end{proof}

\begin{proposition}\label{prop73}
Let $P$ and $Q$ be fs monoids, $\theta:Q\ra P$ a morphism of monoids such that $\theta^{gp}$ is surjective and $(\frakX,Q)$ an fs framed logarithmic $p$-adic formal scheme \eqref{parag73}. Then
\begin{enumerate}
\item The presheaf $\frakX\times_{[Q]}^{\op{log}}[P]$ \eqref{eq731} is representable by a logarithmic formal scheme which is affine and log étale over $\frakX.$
\item If $\frakX \ra [Q]$ lifts to a chart $\frakX\ra B\langle Q\rangle,$ then $\frakX\times_{[Q]}^{\op{log}}[P]$ is representable by $\frakX\times_{B\langle Q\rangle}^{\op{log}} B\langle \tilde{P}\rangle,$ where $\tilde{P}$ is the inverse image of $P$ by $Q^{gp}\ra P^{gp}.$
\end{enumerate}
\end{proposition}

\begin{proof}
The proof is similar to the case of schemes (\cite{Saito04} 4.2.1).
\end{proof}

\begin{corollaire}\label{prop74}
Let $(\frakX,Q)\ra (\frakS,P)\leftarrow (\frakY,Q)$ be morphisms of framed fs logarithmic $p$-adic formal schemes \eqref{def72}. The presheaf $\frakX\times^{\op{log}}_{\frakS,[Q]}\frakY$ \eqref{eq732} is representable by an fs logarithmic $p$-adic formal scheme which is log étale over $\frakX\times_{\frakS}^{\op{log}}\frakY.$
\end{corollaire}

\begin{proof}
We apply the previous proposition to the presheaf
$\frakX\times_{\frakS,[Q]}^{\op{log}}\frakY=\frakX\times_{\frakS}^{\op{log}}\frakY\times_{[Q\oplus Q]}^{\op{log}}[Q].$
\end{proof}

\begin{corollaire}\label{cor76}
Let $(\frakX,Q)\ra (\frakS,P)\leftarrow (\frakY,Q)$ be morphisms of framed fs logarithmic $p$-adic formal schemes \eqref{def72}. Then the projection $\frakX\times_{\frakS,[Q]}^{\op{log}}\frakY\ra \frakY$ is strict.
\end{corollaire}

\begin{proof}
By the proof of \ref{prop73}, the map $\frakX\times_{\frakS,[Q]}^{\op{log}}\frakY\ra [Q]$ is strict. Since $\frakY\ra [Q]$ is also strict, it follows that $\frakX\times_{\frakS,[Q]}^{\op{log}}\frakY\ra \frakY$ is strict. 
\end{proof}

\begin{proposition}\label{prop76}
Let $(\frakX,Q)\ra (\frakS,P)$ be a morphism of framed fs logarithmic $p$-adic locally Noetherian formal schemes \eqref{def72}. Then the diagonal immersion
$\frakX \ra \frakX\times_{\frakS}^{\op{log}}\frakX$
factors canonically into a strict immersion $\frakX\ra \frakX\times^{\op{log}}_{\frakS,[Q]}\frakX$ followed by a log étale morphism $\frakX\times^{\op{log}}_{\frakS,[Q]}\frakX\ra \frakX\times_{\frakS}^{\op{log}}\frakX.$
\end{proposition}

\begin{proof}
The facts that $\frakX\ra \frakX\times^{\op{log}}_{\frakS,[Q]}\frakX$ is strict and $\frakX\times^{\op{log}}_{\frakS,[Q]}\frakX\ra \frakX\times_{\frakS}^{\op{log}}\frakX$ is log étale are consequences of \ref{cor76} and \ref{prop74}. 
\end{proof}

\begin{corollaire}\label{cor79}
Let $(\frakX,Q)\ra (\frakS,P)$ be a morphism of framed fs logarithmic $p$-adic locally Noetherian formal schemes. Denote by $\frakX_{\frakS,[Q]}^{r+1}$ the fiber product $\frakX\times^{\op{log}}_{\frakS,[Q]}\frakX\times^{\op{log}}_{\frakS,[Q]}\hdots \times^{\op{log}}_{\frakS,[Q]}\frakX\ (r+1\ \op{times}).$ Then the canonical morphism $\frakX^{r+1}_{\frakS,[Q]}\ra \frakX^{r+1}$ is étale and $\frakX\ra \frakX^{r+1}_{\frakS,[Q]}$ is a strict immersion, so that the diagonal immersion $\frakX\ra \frakX^{r+1}$ factors as the composition of an étale morphism and a strict immersion.
\end{corollaire}

\begin{parag}\label{parag77}
Let $(\frakX,Q)\ra (\frakS,P)$ be a morphism of framed fs logarithmic $p$-adic locally Noetherian formal schemes (\ref{def72}). Then, by \ref{prop76}, the diagonal immersion $\frakX \ra \frakX \times_{\frakS}^{\op{log}} \frakX$ factors into a strict immersion $\frakX \ra \frakY=\frakX \times_{\frakS,[Q]}^{\op{log}}\frakX$ followed by a log étale morphism $\frakY\ra \frakX \times_{\frakS}^{\op{log}} \frakX.$ Let $\frakX \xrightarrow{\Delta} \frakU \ra \frakY$ be a factorization of $\frakX \ra \frakY$ into a closed immersion followed by an open immersion. For every integer $n\ge 1,$ denote by $\frakX_n$ (resp. $\frakY_n$) the logarithmic scheme obtained from $\frakX$ (resp. $\frakY$) by reduction modulo $p^n,$ as in \ref{prop69}, and denote by $\Delta_n:\frakX_n \ra \frakY_n$ the morphism induced by $\Delta$ and denote by $\calI$ (resp. $\calI_n$) the ideal of the closed immersion $\Delta:\frakX \ra \frakU$ (resp. $\Delta_n:\frakX_n \ra \frakU_n$). By \ref{prop45}, we have a canonical isomorphism $$\Delta_n^{-1} \left (\calI_n/\calI_n^2 \right ) \xrightarrow{\sim} \omega^1_{\frakX_n / \frakS_n}.$$
It follows, by (\ref{eq681}), that we have a canonical isomorphism
\begin{equation}\label{isoomega1}
\Delta^{-1} \left (\calI / \calI^2 \right ) \xrightarrow{\sim} \omega^1_{\frakX / \frakS}.
\end{equation}
Denote by $p_1,p_2:\frakU \ra \frakX$ and $p_{1,n},p_{2,n}:\frakU_n \ra \frakX_n$ the canonical projections. Note that $\Delta^{-1}(1+\calI) \subset \Delta^{-1}\Ox_{\frakU}^*.$
We have the following commutative diagram
$$
\begin{tikzcd}
0 \ar{r} & \Delta^{-1}(1+\calI) \ar{r}{\lambda} \ar{d} & \Delta^{-1} \calM_{\frakU} \ar{d}\ar{r}{\Delta^{\flat}} & \calM_{\frakX} \ar{d}\ar{r} & 0 \\
0 \ar{r} & \Delta_n^{-1}(1+\calI_n) \ar{r}{\lambda_n} & \Delta_n^{-1} \calM_{\frakU_n} \ar{r}{\Delta_n^{\flat}} & \calM_{\frakX_n}\ar{r} & 0
\end{tikzcd}
$$
with exact rows (\ref{era2exactseq}) and where $\lambda$ (resp. $\lambda_n$) is induced by $\alpha_{\frakU}^{-1}$ (resp. $\alpha_{\frakU_n}^{-1}$). Given a local section $m\in \Gamma(\frakU,\calM_{\frakX})$ (resp. $m\in \Gamma(U,\calM_{\frakX_n})$) over an étale $\frakX$-formal scheme $\frakU$ (resp. an étale $X$-scheme U), we denote by $\mu(m)$ (resp. $\mu_n(m)$) the unique section of $\Gamma(\frakU,\Delta^{-1}(1+\calI))$ (resp. $\Gamma(U,\Delta_n^{-1}(1+\calI_n))$) such that $(\Delta^{-1}p_1^{\flat})m+\lambda(\mu(m))=(\Delta^{-1}p_2^{\flat})m$ (resp. $(\Delta_n^{-1}p_{1,n}^{\flat})m+\lambda_n(\mu_n(m))=(\Delta_n^{-1}p_{2,n}^{\flat})m$) and $\eta(m)=\mu(m)-1\in \Delta^{-1}\calI$ (resp. $\eta_n(m)=\mu_n(m)-1\in \Delta_n^{-1}\calI_n$).
 
Suppose $\frakX\ra \frakS$ is smooth and suppose we are given local coordinates $m_{1,1},\hdots,m_{1,d}\in \Gamma(U,\calM_{\frakX_1})$ of $\frakX_1$ over $\frakS_1$ (\ref{P2}), where $U$ is an étale $\frakX_1$-scheme. By \ref{prop69}, after eventually shrinking $U,$ there exists an étale $\frakX$-scheme $\frakU$ and, for every $1\le i\le d,$ a local section $m_i\in \Gamma(\frakU,\calM_{\frakX})$ that lifts $m_{1,i}.$ Since $\eta_1(m_{1,1}),\hdots,\eta_1(m_{1,d}),$ modulo $\calI_1^2,$ form a local basis for $\omega^1_{\frakX_1 / \frakS_1},$ the sections $\eta(m_1),\hdots,\eta(m_d)$ modulo $\calI^2$ also form a local basis for $\omega^1_{\frakX / \frakS}.$ Such local sections $m_i$ will be refered to as \emph{local coordinates for $\frakX$ over $\frakS$}. Note that $\eta(m_1),\hdots,\eta(m_d)$ locally generate the ideal $\calI.$
\end{parag}

\begin{proposition}\label{era2prop611}
Keep the same hypothesis and notation of \ref{parag77}. We identify $\frakY \times_{\frakX}^{\op{log}}\frakY=\frakY \times_{\frakX}\frakY$ with $\frakY(2):=\frakX \times_{\frakS,[Q]}^{\op{log}} \frakX\times_{\frakS,[Q]}^{\op{log}} \frakX$ via the canonical isomorphism. The open formal subscheme $\frakU\times_{\frakX}^{\op{log}}\frakX$ identifies then with an open formal subscheme $\frakU(2)$ of $\frakY(2)$ and the diagonal embedding $\frakX \ra  \frakY(2)$ factors into an exact closed immersion $\Delta(2): \frakX \ra \frakU(2)$ followed by the open immersion $\frakU(2)\ra \frakY(2).$ Let $\w{\delta}$ be the morphism of structural rings associated with the $(1,3)$-projection $p_{13}:\frakU(2) \rightarrow \frakU.$ Then, for any local section $m$ of $\calM_{\frakX},$
$$(\Delta(2)^{-1}\w{\delta})(\eta(m))=\eta(m)\otimes \eta(m)+\eta(m)\otimes 1+1\otimes \eta(m).$$
\end{proposition}

\begin{proof}
For $1\le i\le 3$ and $1\le j\le 2,$ denote by $\pi_i:\frakU(2) \ra \frakX$ and $p_j:\frakU \ra \frakX$ the projections on the $i$th and $j$th factor respectively and, for $1\le i<j\le 3,$ denote by $p_{ij}:\frakU(2) \ra \frakX$ the projection on the $(i,j)$th factor. 
Let $\calI$ and $\calI(2)$ be the ideals of the exact immersions $\Delta:\frakX \ra \frakU$ and $\Delta(2):\frakX \ra \frakU(2)$ respectively. By \eqref{era2exactseq}, we have exact sequences fitting into the commutative diagram
$$
\begin{tikzcd}
0 \ar{r} & \Delta^{-1}(1+\calI) \ar{r}{\lambda} \ar{d}{\Delta(2)^{-1}p_{ij}^{\#}} & \Delta^{-1}\calM_{\frakU} \ar{r}{\Delta^{\flat}} \ar{d}{\Delta(2)^{-1}p_{ij}^{\flat}} & \calM_{\frakX} \ar{r} \ar[equal]{d} & 0 \\
0 \ar{r} & \Delta(2)^{-1}(1+\calI(2)) \ar{r}{\lambda_2} & \Delta(2)^{-1}\calM_{\frakU(2)} \ar{r}{\Delta(2)^{\flat}} & \calM_{\frakX} \ar{r} & 0.
\end{tikzcd}
$$
For $1\le i<j\le 3,$ let $\eta_{ij}(m)$ be the unique local section of $\Delta(2)^{-1}(\calI(2))$ satisfying
$$
\lambda_2(1+\eta_{ij}(m))+\pi_i^{\flat}(m)=\pi_j^{\flat}(m).
$$
We drop the notation $\Delta(2)^{-1}$ and $\Delta^{-1}$ to lighten the notation.
We have
\begin{alignat*}{2}
\lambda_2(p_{13}^{\#}(1+\eta(m)))+\pi_1^{\flat}(m) &= p_{13}^{\flat}(\lambda(1+\eta(m)))+p_{13}^{\flat}p_1^{\flat}(m) = p_{13}^{\flat}(p_2^{\flat}(m)) \\
&= \pi_3^{\flat}(m) = \lambda_2(1+\eta_{23}(m))+\lambda_2(1+\eta_{12}(m))+\pi_1^{\flat}(m) \\
&= \lambda_2\left (1+\eta_{12}(m)+\eta_{23}(m)+\eta_{12}(m)\eta_{23}(m) \right )+\pi_1^{\flat}(m).
\end{alignat*}
Since $\calM_{\frakU}$ is integral and $\lambda_2$ is injective, we get
$$p_{13}^{\#}(\eta(m))=\eta_{12}(m)+\eta_{23}(m)+\eta_{12}(m)\eta_{23}(m).$$
We conclude by the following fact: let $q_1,q_2:\frakU\times_{\frakX}\frakU \ra \frakU$ be the canonical projections and $\Delta':\frakX \ra \frakU\times_{\frakX}\frakU$ the morphism induced by $\Delta:\frakX \ra \frakU.$ Since we identify $\frakU\times_{\frakX}\frakU$ and $\frakU(2)$ via the canonical isomorphism then $\Delta'$ identifies with $\Delta(2),$ $\eta_{12}(m)$ identifies with $(\Delta'^{-1}q_1^{\#})(\eta(m))$ and $\eta_{23}(m)$ identifies with $(\Delta'^{-1}q_2^{\#})(\eta(m)).$
\end{proof}

\begin{proposition}\label{era2prop612}
Keep the same hypothesis and notation of \ref{parag77}. Let $\sigma:\frakY \ra \frakY$ be the morphism exchanging the $\frakX$ factors. For any local section $m$ of $\calM_{\frakX},$ we have
$$(\Delta^{-1}\sigma^{\#})(\eta(m))=(1+\eta(m))^{-1}-1.$$ 
\end{proposition}

\begin{proof}
We have the following commutative diagram
$$
\begin{tikzcd}
\Delta^{-1}(1+\calI) \ar{r}{\lambda} \ar[swap]{d}{\Delta^{-1}\sigma^{\#}} & \Delta^{-1}\calM_{\frakY} \ar{d}{\Delta^{-1}\sigma^{\flat}} \\
\Delta^{-1}(1+\calI) \ar{r}{\lambda} & \Delta^{-1}\calM_{\frakY}.
\end{tikzcd}
$$
We get
\begin{alignat*}{2}
\lambda\left (\Delta^{-1}\sigma^{\#}(\mu(m))\right ) + \left (\Delta^{-1}p_2^{\flat}\right )(m) &= \lambda\left (\Delta^{-1}\sigma^{\#}(\mu(m))\right ) + \left (\Delta^{-1} \sigma^{\flat} \right ) \left (\Delta^{-1}p_1^{\flat} \right )(m) \\
&= \left (\Delta^{-1}\sigma^{\flat} \right ) (\lambda(\mu(m))) + \left (\Delta^{-1} \sigma^{\flat} \right ) \left (\Delta^{-1}p_1^{\flat} \right )(m) \\
&= \left (\Delta^{-1}\sigma^{\flat} \right ) \left ( \Delta^{-1}p_2^{\flat} \right )(m) \\
&= \left (\Delta^{-1}p_1^{\flat} \right )(m)  \\
&= \lambda\left (\mu(m)^{-1} \right ) + \left (\Delta^{-1}p_2^{\flat}\right )(m).
\end{alignat*}
The morphism $\lambda$ is injective and $\Delta^{-1}\calM_{\frakY}$ is integral, it follows that
$
\left (\Delta^{-1}\sigma^{\#}\right )(\mu(m))=\mu(m)^{-1}.
$
\end{proof}

\section{A logarithmic Shiho functor}

\begin{parag}\label{parag48}
In this section, we fix a log smooth morphism $f:\frakX\ra \frakS$ of fs logarithmic $p$-adic locally Noetherian formal schemes flat over $\Z_p$ (\ref{parag63}). For any integer $n\ge 1,$ we denote by $\frakX_n$ and $\frakS_n$ the logarithmic schemes obtained from $\frakX$ and $\frakS$ respectively by reduction modulo $p^n,$ as in \ref{prop69}. Set $X=\frakX_1$ and $S=\frakS_1.$ We use the same notations as in \ref{PFrob}. Namely, we denote by $F_S:S\ra S$ the absolute Frobenius morphism of $S$ \eqref{Not3}, by $F_{X/S}:X\ra X''$ the relative Frobenius morphism of $X$ with respect to $S,$ by $F_1:X\ra X'$ the exact relative Frobenius of $X$ with respect to $S$ and by $F_1^{\#}:\Ox_{X'}\ra F_{1*}\Ox_X$ the associated morphism of structural rings. Denote also by $G:X'\ra X''$ and $\pi:X'\ra X$ the morphisms defined in \eqref{diag51}.
\end{parag}

\begin{parag}
In the rest of this section, we suppose that we have an $\frakS$-lifting $F:\frakX\ra \frakX'$ of the exact relative Frobenius $F_1:X\ra X'.$ In other words, we suppose that there exists an fs logarithmic $p$-adic formal scheme $\frakX'$ over $\frakS$ and an $\frakS$-morphism $F:\frakX\ra \frakX'$ that fit into cartesian squares
$$\begin{tikzcd}
X'\ar{r}\ar{d} & S\ar{d} & & X\ar{r}{F_1}\ar{d} & X'\ar{d}\\
\frakX'\ar{r} & \frakS & & \frakX\ar{r}{F} & \frakX'
\end{tikzcd}$$
Note that, by (\cite{Ahmed2010} 2.1.36), $\frakX'$ is locally Noetherian.
We also suppose that $F$ is log flat \eqref{thmlogflat}.
For any integer $n\ge 1,$ we denote by $F_n:\frakX_n\ra \frakX_n'$ the base change of $F:\frakX\ra \frakX'$ by $\frakS_n\ra \frakS.$
\end{parag}

\begin{lemma}\label{lem12}
Consider the composition $X' \xrightarrow{\pi} X \ra \frakX$ where $\pi$ is given in \eqref{diag51}. Let $\frakU \ra \frakX$ be an étale morphism of $p$-adic formal schemes and $\frakU' \ra \frakX'$ the unique étale morphism such that $\frakU \times_{\frakX} X'=\frakU' \times_{\frakX'}X'.$ Let $m$ be a local section of $\calM_{\frakX}$ (resp. $\calM_{\frakX_n}$) over $\frakU$ and $m_1$ its image in $\Gamma \left (\frakU\times_{\frakX}X,\calM_X \right ).$ By \ref{prop69}, after eventually reducing $\frakU,$ there exists a lifting $m'$ of $\pi^{\flat}m_1\in \Gamma \left (\frakU\times_{\frakX}X',\calM_{X'} \right )$ to $\Gamma \left (\frakU\times_{\frakX}X',\calM_{\frakX'} \right )$ (resp. $\Gamma \left (\frakU\times_{\frakX}X',\calM_{\frakX'_n} \right )$). Then there exists an invertible local section $u$ of $\calM_{\frakX}$ (resp. $\calM_{\frakX_n}$) such that $u+pm=F^{\flat}(m')$ (resp. $u+pm=F_n^{\flat}(m')$) and $\alpha_{\frakX}(u)=1+pb$ (resp. $\alpha_{\frakX_n}(u)=1+pb$) for a local section $b$ of $\Ox_{\frakX}$ (resp. $\Ox_{\frakX_n}$).
\end{lemma}

\begin{proof}
We prove the lemma for $\frakX.$ The proof for $\frakX_n$ is similar.
Let $\gamma:\calM_{\frakX}\ra \Ox_X$ be the composition $\calM_{\frakX}\xrightarrow{\alpha_{\frakX}} \Ox_{\frakX}\ra \Ox_X,$ where $\Ox_{\frakX}\ra \Ox_X$ is the reduction modulo $p.$ The morphism $X\ra \frakX$ is strict (\ref{prop69}) so $\calM_X=\calM_{\frakX}\oplus_{\gamma^{-1}(\Ox_X^*)}\Ox_X^*.$ The left commutative diagram gives rise to the right commutative diagram of sheaves
$$\begin{tikzcd}
X\ar{r}\ar{d}{F_1} & \frakX\ar{d}{F}\\
X'\ar{r} & \frakX'
\end{tikzcd} \qquad
\begin{tikzcd}
\calM_X & \calM_{\frakX}\ar{l}\\
\calM_{X'}\ar{u}{F_1^{\flat}} & \calM_{\frakX'}\ar{u}{F^{\flat}}\ar{l}.
\end{tikzcd}$$
It follows that, in $\calM_X,$ we have the equality
$(pm,1)=(F^{\flat}(m'),1).$
By the definition of amalgamated sums in the category of monoids (\cite{Kat89}  (1.3)), there exist local sections $v$ and $v'$ of $\gamma^{-1}(\Ox_X^*)$ such that $F^{\flat}(m')+v'=pm+v$ and $\gamma(v')=\gamma(v).$ Since $\gamma(v')=\gamma(v),$ there exists a local section $c$ of $\Ox_{\frakX}$ such that $\alpha_{\frakX}(v)=\alpha_{\frakX}(v')+pc.$ The local sections $\gamma(v)$ and $\gamma(v')$ are both invertible so $\alpha_{\frakX}(v)$ and $\alpha_{\frakX}(v')$ are also invertible and so $v$ and $v'$ are invertible in $\calM_{\frakX}.$ It follows that we can take $u=v-v'$ and $\alpha_{\frakX}(u)=\alpha_{\frakX}(v)\alpha_{\frakX}(v')^{-1}=1+pb$ where $b=\alpha_{\frakX}(v')^{-1}c.$
\end{proof}

\begin{lemma}\label{lemcalc}
Suppose that $\frakX$ and $\frakS$ are equipped with frames $\frakX \ra [Q]$ and $\frakS\ra [P],$ where $P$ and $Q$ are fs monoids, that $f$ underlies a morphism of framed logarithmic formal schemes and that there exists a morphism $F:(\frakX,Q) \ra (\frakX',Q')$ of framed logarithmic formal schemes, over $(\frakS,P),$ lifting the exact relative Frobenius $F_1:X \ra X'$ \eqref{diag51}, where $Q'$ is defined in \ref{PFrob}. Let $\Delta:\frakX \ra \frakY:=\frakX\times_{\frakS,[Q]}^{\op{log}}\frakX$ and $\Delta':\frakX' \ra \frakX' \times_{\frakS,[Q']}^{\op{log}}\frakX'$ be the strict diagonal immersions \eqref{prop76}. Denote by $p_1,p_2:\frakY\ra \frakX$ and $p_1',p_2':\frakY'\ra \frakX'$ the canonical projections. Denote by $G:\frakY\ra \frakY'$ the morphism induced by $F:(\frakX,Q)\ra (\frakX',Q').$ Let $\frakX \ra \frakU \ra \frakY$ and $\frakX' \ra \frakU' \ra \frakY'$ be factorizations of $\Delta$ and $\Delta'$ respectively into a closed immersion followed by an open one and denote by $\calI$ and $\calI'$ the ideals of $\frakX\ra \frakU$ and $\frakX' \ra \frakU'$ respectively. Let $m_1$ be a local section of $\calM_X,$ $m'_1$ its image in $\calM_{X'}$ by $\pi^{\flat}:\calM_X\ra \pi_*\calM_{X'}$ \eqref{diag51}, $m$ (resp. $m'$) a local lifting of $m_1$ (resp. $m'_1$) to $\calM_{\frakX}$ (resp. $\calM_{\frakX'}$).
Let $\eta(m)\in \Delta^{-1}\calI,\ \eta(m')\in \Delta'^{-1}\calI',$ $\mu(m)=\eta(m)+1$ and $\mu(m')=\eta(m')+1$ as defined in \ref{parag77}. By \ref{lem12}, there exists a local invertible section $u$ of $\calM_{\frakX}$ and a local section $b$ of $\Ox_{\frakX}$ such that $F^{\flat}(m')=pm+u$ and $\alpha_{\frakX}(u)=1+pb.$
Then
\begin{equation}\label{eq1321}
\Delta^{-1}G^{\#}\left ( \eta(m')\right ) = \left (\eta(m)^p+\sum_{k=1}^{p-1}\begin{pmatrix}p\\k \end{pmatrix}\eta(m)^k+1 \right )\alpha_{\frakY}(p_2^{\flat}u-p_1^{\flat}u)-1
\end{equation}
and
\begin{equation}\label{eq1322}
\alpha_{\frakY}(p_2^{\flat}u-p_1^{\flat}u)=\frac{1+pp_2^{\#}(b)}{1+pp_1^{\#}(b)}.
\end{equation}
\end{lemma}

\begin{proof}
By \ref{parag77}, we have the following commutative diagram with exact rows
$$
\begin{tikzcd}
0 \ar{r} & \Delta'^{-1}(1+\calI') \ar{r}{\lambda'} \ar{d}{\Delta^{-1}G^{\#}} & \Delta'^{-1}\calM_{\frakY'} \ar{d}{\Delta^{-1}G^{\flat}} \ar{r} & \calM_{\frakX'} \ar{d}{F^{\flat}} \ar{r} & 0 \\
0 \ar{r} & \Delta^{-1}(1+\calI) \ar{r}{\lambda}  & \Delta^{-1}\calM_{\frakY} \ar{r} & \calM_{\frakX} \ar{r} & 0.
\end{tikzcd}
$$
In the rest of this proof, we drop $\Delta^{-1}$ and $\Delta'^{-1}$ to lighten the notation. By definition, we have
$p_1'^{\flat}(m')+\lambda'(\mu(m'))=p_2'^{\flat}(m').$
Applying $G^{\flat},$ we get
$p_1^{\flat}F^{\flat}(m')+\lambda(G^{\#}(\mu(m'))=p_2^{\flat}F^{\flat}(m').$
So we have
$pp_1^{\flat}(m)+p_1^{\flat}(u)+\lambda(G^{\#}(\mu(m')))=pp_2^{\flat}(m)+p_2^{\flat}(u).$
By definition,
$p_1^{\flat}(m)+\lambda(\mu(m))=p_2^{\flat}(m).$
It follows that
$\lambda(G^{\#}(\mu(m')))=\lambda(\mu(m)^p)+p_2^{\flat}u-p_1^{\flat}u.$
But $\lambda(\alpha_{\frakY}(p_2^{\flat}u-p_1^{\flat}u))=p_2^{\flat}u-p_1^{\flat}u$ by definition of $\lambda.$ The equality \eqref{eq1321} then follows from the injectivity of $\lambda.$ For \eqref{eq1322}, we have
\begin{equation*}
\alpha_{\frakY}(p_2^{\flat}u-p_1^{\flat}u) = \frac{p_2^{\#}\alpha_{\frakX}(u)}{p_1^{\#}\alpha_{\frakX}(u)}=\frac{1+pp_2^{\#}(b)}{1+pp_1^{\#}(b)}.
\end{equation*}
\end{proof}

\begin{parag}
We have the following commutative diagram
$$\begin{tikzcd}
F_1^*G^*\omega^1_{X''/S}\ar{rd}{dF_{X/S}}\ar{d}{F_1^*dG} & \\
F_1^*\omega^1_{X'/S}\ar{r}{dF_1} & \omega^1_{X/S}
\end{tikzcd}.$$
Since $dF_{X/S}=0$ and $dG$ is an isomorphism (because $G$ is étale), we deduce that $dF_1=0.$
Let $n\ge 1$ and denote by $F_n:\frakX_n\ra \frakX'_n$ the reduction of $F$ modulo $p^n.$ By the following commutative diagram
$$\begin{tikzcd}
F_n^*\omega^1_{\frakX_n'/\frakS_n}\ar{d}\ar{r}{dF_n} & \omega^1_{\frakX_n/\frakS_n} \ar{d}\\
F_1^*\omega^1_{X'/S}\ar{r}{dF_1} & \omega^1_{X/S}
\end{tikzcd}$$
(where the vertical arrows are reduction modulo $p$) and the fact that $dF_1=0,$ we get that $\op{Im}(dF_n)\subset p\omega^1_{\frakX_n/\frakS_n}.$
Now by the flatness of $\frakX_{n+1}$ over $\op{Spec} \Z/p^{n+1}\Z,$ the sequence
$0\ra \Ox_{\frakX_n}\xrightarrow{\times p} \Ox_{\frakX_{n+1}}\rightarrow \Ox_{\frakX_1} \ra 0$
is exact.
Since $f:\frakX\ra \frakS$ is smooth, the $\Ox_{\frakX_{n+1}}$-module $\omega^1_{\frakX_{n+1}/\frakS_{n+1}}$ is locally free and so we get the exact sequence
$$0\ra \omega^1_{\frakX_n/\frakS_n}\xrightarrow{\times p}\omega^1_{\frakX_{n+1}/\frakS_{n+1}}\ra \omega^1_{\frakX_1/\frakS_1}\ra 0.$$
We deduce the existence of a unique $\Ox_{\frakX_n}$-linear morphism $p^{-1}dF_{n+1}:F_n^*\omega^1_{\frakX_n'/\frakS_n}\ra \omega^1_{\frakX_n/\frakS_n}$ making the following diagram commutative
\begin{equation}\label{surp}
\begin{tikzcd}
F_{n+1}^*\omega^1_{\frakX_{n+1}'/\frakS_{n+1}}\ar{d}\ar{r}{dF_{n+1}} & p\omega^1_{\frakX_{n+1}/\frakS_{n+1}} \\
F_n^*\omega^1_{\frakX'_n/\frakS_n}\ar{r}{p^{-1}dF_{n+1}} & \omega^1_{\frakX_n/\frakS_n}\ar{u}{\times p}
\end{tikzcd}
\end{equation}
Let us give an explicit expression of $p^{-1}dF_{n+1}$ in light of \ref{lem12}.
Let $U\ra \frakX_{n+1}$ be an étale morphism, $\tilde{m}\in \Gamma(U,\calM_{\frakX_{n+1}}),$ $m$ and $\ov{m}$ the images of $\tilde{m}$ in $\Gamma(U\times_{\frakX_{n+1}}\frakX_n,\calM_{\frakX_n})$ and $\Gamma(U\times_{\frakX_{n+1}}X,\calM_X)$ respectively. Let $\tilde{m}'$ be a local lifting of $\pi^{\flat}\ov{m}$ to $\calM_{\frakX_{n+1}'}$ and denote by $m'$ its image in $\calM_{\frakX_n'}$ \eqref{prop69}. By \ref{lem12}, there exists an invertible local section $\tilde{u}$ of $\calM_{\frakX_{n+1}}$ and a local section $\tilde{b}$ of $\Ox_{\frakX_{n+1}}$ such that $p\tilde{m}+\tilde{u}=F_{n+1}^{\flat}(\tilde{m}')$ and $\alpha_{\frakX_{n+1}}(\tilde{u})=1+p\tilde{b}.$ Let $b$ (resp. $u$) be the image of $\tilde{b}$ (resp. $\tilde{u}$) in $\Ox_{\frakX_n}$ (resp. $\calM_{\frakX_n}$). 
By the commutativity of \eqref{surp}, we deduce that in $\omega^1_{\frakX_{n+1}/\frakS_{n+1}}:$
\begin{alignat*}{2}
p\times (p^{-1}dF_{n+1})(F_n^*\op{dlog}m') &= dF_{n+1}(F_{n+1}^*\op{dlog}\tilde{m}') =  \op{dlog}(F_{n+1}^{\flat}\tilde{m}') \\
&= \op{dlog}(p\tilde{m}+\tilde{u}) = p\op{dlog}\tilde{m}+\op{dlog}\tilde{u}\\
&= p\op{dlog}\tilde{m}+\frac{d\alpha_{\frakX_{n+1}}(\tilde{u})}{\alpha_{\frakX_{n+1}}(\tilde{u})}= p\op{dlog}\tilde{m}+p\frac{d\tilde{b}}{1+p\tilde{b}}. 
\end{alignat*}
By the injectivity of $\omega^1_{\frakX_n/\frakS_n}\xrightarrow{\times p}\omega^1_{\frakX_{n+1}/\frakS_{n+1}},$ we conclude that, in $\omega^1_{\frakX_n/\frakS_n},$ we have:
\begin{equation}\label{surp2}
(p^{-1}dF_{n+1})(F_n^*\op{dlog}m') = \op{dlog}m+\frac{db}{1+pb}= \op{dlog}m+(1-pb+p^2b^2-\hdots )db.
\end{equation}
\end{parag}

\begin{lemma}\label{lem85}
Let $n\ge 1$ be an integer and suppose that $\frakX_n \ra \frakS_n$ is log smooth. Let $(\calE',\nabla')$ be an $\Ox_{\frakX_n'}$-module equipped with a $p$-connection \eqref{parKhaminei} and $\zeta$ the composition
$$\zeta:F_n^*(\calE'\otimes_{\Ox_{\frakX_n'}}\omega^1_{\frakX_n'/\frakS_n})\xrightarrow{\sim} F_n^*(\calE')\otimes_{\Ox_{\frakX_n}}F_n^*(\omega^1_{\frakX_n'/\frakS_n})\xrightarrow{\op{Id}_{F_n^*(\calE')}\otimes p^{-1}dF_{n+1}}F_n^*(\calE')\otimes_{\Ox_{\frakX_n}}\omega^1_{\frakX_n/\frakS_n},$$
where the first arrow is the canonical isomorphism.
Consider the morphism 
\begin{equation}\label{conn}
\nabla:\begin{array}[t]{lc}F_n^*\calE'\ra F_n^*\calE'\otimes_{\Ox_{\frakX_n}}\omega^1_{\frakX_n/\frakS_n}\\ x\otimes a\mapsto a\zeta( F_n^*\nabla'(x))+F_n^*x\otimes da,\end{array}
\end{equation}
where $x$ and $a$ are local sections $\calE'$ and $\Ox_{\frakX_n}$ respectively.
Then $\nabla$ a well-defined connection on $F_n^*\calE'.$ In addition, if $\nabla'$ is integrable then so is $\nabla.$
\end{lemma}

\begin{proof}
It is clear that for any local sections $x$ and $y$ of $\calE',$ $a$ and $b$ of $\Ox_{\frakX_n}$ and $\alpha$ of $\Ox_{\frakX_n'},$
$\nabla((x+y)\otimes a)=\nabla(x\otimes a)+\nabla(y\otimes a)$ and 
$\nabla(x\otimes (a+b))=\nabla(x\otimes a)+\nabla(x\otimes b).$
In addition,
$$
\nabla((\alpha x)\otimes b)=F_n^{\#}(\alpha) b\zeta(F_n^*\nabla'(x))+F_n^*x\otimes bdF_n^{\#}(\alpha)+F_n^*(\alpha x)\otimes db
$$
and
$
\nabla(x\otimes(F_n^{\#}(\alpha) b)) = F_n^{\#}(\alpha) b\zeta(F_n^*\nabla'(x))+F_n^*x\otimes d(F_n^{\#}(\alpha) b)=\nabla((\alpha x)\otimes b).
$
It follows that $\nabla$ is well-defined and by (\ref{conn}), it is a connection on the $\Ox_{\frakX_n}$-module $F_n^*\calE'.$ 

Now suppose that $\nabla'$ is integrable and let us check that $\nabla$ is also integrable. Let $x$ be a local section of $\calE'.$ Let $(\ov{m}_i)_{1\le i\le d}$ be local coordinates for $X\ra S$ (\ref{P2}). For $1\le i\le d,$ let $m_i$ (resp. $m_i'$) be a local lifting of $\ov{m}_i$ (resp. $\pi^{\flat}\ov{m}_i$) to $\frakX_n$ (resp. $\frakX_n'$). Then $(\op{dlog}m_i')_{1\le i\le d}$ generates the $\Ox_{\frakX_n'}$-module $\omega^1_{\frakX_n'/\frakS_n}.$ Let $x_i,\ x_{ij},\ i,j=1,\hdots,d$ be local sections of $\calE'$ such that
$$
\nabla'(x)=\sum_{i=1}^dx_i\otimes \op{dlog}m_i',\qquad
\nabla'(x_i)=\sum_{j=1}^dx_{ij}\otimes \op{dlog}m_j'.
$$
It follows that the curvature $K(\nabla')$ of $\nabla'$ is given by $$K(\nabla')(x)=\sum_{i,j}x_{ij}\otimes \op{dlog}m_i'\wedge \op{dlog}m_j'.$$ On the other hand, by the definition \ref{conn}, we have
$$\nabla(aF_n^*x)=\sum_{i=1}^daF_n^*x_i\otimes p^{-1}F_{n+1}(\op{dlog}m_i')+F_n^*x\otimes da,$$
for any local sections $a$ and $x$ respectively of $\Ox_{\frakX_n}$ and $\calE'.$ It follows that the curvature $K(\nabla)$ of $\nabla$ is given, for any local sections $a$ and $x$ respectively of $\Ox_{\frakX_n}$ and $\calE',$ by
\begin{alignat*}{2}
K(\nabla)(aF_n^*x)&=\nabla^1\circ \nabla(aF_n^*x)=\sum_{i=1}^d\nabla^1(aF_n^*x_i\otimes (p^{-1}F_{n+1})(\op{dlog}m_i'))+\nabla^1(F_n^*x\otimes da)\\
&=\sum_{i=1}^d\nabla(aF_n^*x_i)\wedge (p^{-1}F_{n+1})(\op{dlog}m_i')+aF_n^*x_i\otimes d((p^{-1}F_{n+1})(\op{dlog}m_i'))\\ 
&\ +\nabla(F_n^*x)\wedge da.
\end{alignat*}
Finally, we get:
\begin{alignat}{2}
K(\nabla)(aF_n^*x)&=\sum_{i,j}aF_n^*x_{ij}\otimes (p^{-1}dF_{n+1})(\op{dlog}m_j')\wedge (p^{-1}dF_{n+1})(\op{dlog}m_i') \label{856} \\
&\ +\sum_iF_n^*x_i\otimes da\wedge (p^{-1}dF_{n+1})(\op{dlog}m_i') \label{857}\\
&\ +\sum_iaF_n^*x_i\otimes d((p^{-1}dF_{n+1})(\op{dlog}m_i')) \label{858}\\
&\ +\nabla(F_n^*x)\wedge da. \label{859}
\end{alignat} 
The sum \eqref{856} is simply the image of $K(\nabla')(x)$ by the morphism $F_n^*\otimes p^{-1}dF_{n+1}\wedge p^{-1}dF_{n+1}$ and is thus equal to zero since $\nabla'$ is integrable. We also have
\begin{equation*}
\nabla(F_n^*x)\wedge da = \sum_iF_n^*x_i\otimes (p^{-1}dF_{n+1})(\op{dlog}m'_i)\wedge da
\end{equation*}
so \eqref{857} and \eqref{859} cancel out. It remains to prove that (\ref{858}) vanishes. By (\ref{surp2}), there exists a local section $b_i$ of $\Ox_{\frakX_n}$ such that 
$$(p^{-1}dF_{n+1})(F_{n+1}^*\op{dlog}m_i')=\op{dlog}m_i+(1-pb_i+p^2b_i^2-\hdots )db_i.$$
It is therefore a closed form and so (\ref{858}) vanishes.
We conclude that $K(\nabla)(x)=0$ and so $\nabla$ is integrable. 
\end{proof}

\begin{parag}
Keep the same hypothesis of \ref{lem85}. By \ref{lem85}, we have a functor
\begin{equation}\label{krazphin}
\Phi_n:\begin{array}[t]{clc}
p\op{-MIC}(\frakX_n'/\frakS_n) & \ra & \op{MIC}(\frakX_n/\frakS_n)\\ 
(\calE',\nabla') & \mapsto & (F_n^*\calE',\nabla)
\end{array}
\end{equation}
from the category $p\op{-MIC}(\frakX_n'/\frakS_n)$ of $\Ox_{\frakX_n'}$-modules with an integrable $p$-connection to the category $\op{MIC}(\frakX_n/\frakS_n)$ of $\Ox_{\frakX_n}$-modules with an integrable connection.
\end{parag}

\section{Dilatations}

In this section, we fix a perfect field $k$ of characteristic $p$ and denote by $W(k)$ its ring of Witt vectors. Let $\frakS=\op{Spf}W(k)$ equipped with the trivial logarithmic structure. For any logarithmic $p$-adic formal scheme $\frakY$ and any positive integer $n,$ we denote by $\frakY_n$ the logarithmic scheme obtained from $\frakY$ by reduction modulo $p^n$ (as in \ref{prop69}).

\begin{parag}
Let $n$ be a positive integer, $\frakX$ a logarithmic $p$-adic formal scheme flat and locally of finite type over $\frakS$ (\ref{dxuflat}) and $\calI$ an open ideal of finite type of $\Ox_{\frakX}$ containing $p^n$ (\cite{Ahmed2010} 2.1.19). Let $\frakY \ra \frakX$ be the admissible blow-up of $\calI$ in $\frakX$ (\cite{Ahmed2010} 3.1.2). We equip it with the logarithmic structure pullback of that of $\frakX.$ By (\cite{Ahmed2010} 3.1.4), the ideal $\calI\Ox_{\frakY}$ is invertible and $\frakY$ is flat over $\Z_p.$ We denote by $\frakX_{(\calI/p^n)}$ the dilatation of $\calI$ with respect to $p^n$ i.e. the largest open formal subscheme of $\frakY$ such that the restriction of $\calI \Ox_{\frakY}$ on $\frakX_{(\calI/p^n)}$ is generated by $p^n$ (\cite{Ahmed2010} 3.2.3.4 and 3.2.7). The formal scheme $\frakX_{(\calI/p^n)}$ is flat over $\frakS$ (\ref{dxuflat}).
If $\frakX=\op{Spf}A$ and $I$ is the open ideal of $A$ corresponding to $\calI,$ and $(a_0,\hdots,a_r)$ is a set of generators of $I$ such that $a_0=p^n,$ then $\frakX_{(\calI/p^n)}=\op{Spf}B$ where $B$ is the $p$-adic completion of the ring $A_0/(p^n\text{-tor}),$ where
\begin{equation}\label{dilaff}
A_0=A\left [x_1,\hdots,x_r\right ]/(p^nx_1-a_1,\hdots,p^nx_r-a_r)
\end{equation}
and $(p^n\text{-tor})$ is the $p^n$-torsion ideal of $A_0.$ In particular, we see that the canonical morphism $\frakX_{(\calI/p^n)} \ra \frakX$ is affine.
\end{parag}

\begin{proposition}\label{erafdil1}
Let $n$ be a positive integer, $\frakY$ a logarithmic $p$-adic formal scheme flat and locally of finite type over $\frakS$ and $i:T \ra \frakY_n$ a strict immersion. There exists a strict morphism of logarithmic $p$-adic formal schemes flat over $\frakS,$ $g:\frakY_{(T/p^n)} \ra \frakY,$ unique up to a canonical isomorphism, satisfying the following conditions:
\begin{enumerate}
\item The morphism $\left (\frakY_{(T/p^n)}\right )_n \ra \frakY_n$ factors uniquely through $T \ra \frakY_n$ and the resulting morphism $\left (\frakY_{(T/p^n)}\right )_n \ra T$ is affine.
\item If $\frakZ$ is a logarithmic $p$-adic formal scheme flat over $\frakS$ and $f:\frakZ \ra \frakY$ is an $\frakS$-morphism such that $\frakZ_n \ra \frakY_n$ factors through $T\ra \frakY_n$ then there exists a unique $\frakS$-morphism $f':\frakZ \ra \frakY_{(T/p^n)}$ such that $f=g\circ f'.$ In addition, if $T \ra \frakY$ and $f$ are closed immersions, then so is $f'.$
\end{enumerate}
\end{proposition}

\begin{proof}
The uniqueness follows from the second assertion, so we can work locally and assume that the immersion $i:T \ra \frakY_n$ is closed. We denote by $\calI$ the ideal of $T \ra \frakY.$ We take $\frakY_{(T/p^n)}$ to be the dilatation of $\calI$ with respect to $p^n,$ we equip it with the logarithmic structure pull back of that of $\frakY$ and take $g:\frakY_{(T/p^n)} \ra \frakY$ to be the canonical morphism. By \ref{era3proplogstrformal}, to prove that the morphism of logarithmic formal schemes $\left (\frakY_{(T/p^n)} \right )_n \ra \frakY_n$ factors through $i,$ it is sufficient to prove the factorization for the underlying formal schemes, which is clear. For the second assertion, again by \ref{era3proplogstrformal}, it is sufficient to prove the existence of $f'$ on the underlying formal schemes. The rest of the proof is the same as in (\cite{DXU19} 3.5).
\end{proof}

\begin{proposition}\label{erafdil3}
Let $n$ be a positive integer, $\frakY$ a logarithmic $p$-adic formal scheme flat and locally of finite type over $\frakS$ and $i:T \ra \frakY_n$ a strict immersion. For any integer $k\ge n,$ there exists a canonical morphism
$$\frakY_{(T/p^{k+1})} \ra \frakY_{(T/p^{k})}.$$
\end{proposition}

\begin{proof}
We have the commutative diagram on the left:
$$
\begin{tikzcd}
\frakY_{(T/p^{k+1}),k} \ar{r} \ar{d} & \frakY_k \ar{d} \\
\frakY_{(T/p^{k+1}),k+1} \ar{r} \ar{d} & \frakY_{k+1} \\
T \ar{ur} & 
\end{tikzcd}
\begin{tikzcd}
\frakY_{(T/p^{k+1}),k} \ar{r} \ar{d} & \frakY_k \\
\frakY_{(T/p^{k+1}),k+1} \ar{d} &  \\
T \ar{uur} & 
\end{tikzcd}
$$
Since $T \ra \frakY_{k+1}$ decomposes as $T \ra \frakY_k \ra \frakY_{k+1}$ and $\frakY_k \ra \frakY_{k+1}$ is an immersion, hence a monomorphism, we obtain the commutative diagram on the right.
The desired morphism is then obtained by the flatness of $\frakY_{(T/p^{k+1})}$ over $\frakS$ and by applying the universal property of $\frakY_{(T/p^k)}$ to the canonical morphism
$\frakY_{(T/p^{k+1})} \ra \frakY.$
\end{proof}

\begin{proposition}[\cite{Oyama} 1.1.4 and \cite{DXU19} 3.7]\label{propdiletale}
Let $n$ be a positive integer, $\frakX$ and $\frakY$ two logarithmic $p$-adic formal schemes flat and locally of finite type over $\frakS$ and $f:\frakX \ra \frakY$ a log étale $\frakS$-morphism. Suppose that there exists two strict $\frakS_n$-immersions $i:T \ra \frakX_n$ and $j:T \ra \frakY_n$ such that $f_n\circ i=j.$ Then $f$ induces a canonical isomorphism of logarithmic formal schemes
$$\frakX_{(T/p^n)} \xrightarrow{\sim} \frakY_{(T/p^n)}.$$
\end{proposition}

\begin{proof}
Denote by $g_{\frakX}:\frakX_{(T/p^n)} \ra \frakX$ the canonical morphism. Since the formal scheme $\frakX_{(T/p^n)}$ is flat over $\frakS,$
the universal property of $\frakY_{(T/p^n)}$ applied to the morphism
$\frakX_{(T/p^n)} \xrightarrow{g_{\frakX}} \frakX \xrightarrow{f} \frakY$
implies the existence of a unique morphism $f':\frakX_{(T/p^n)} \ra \frakY_{(T/p^n)}$ such that the diagram
$$
\begin{tikzcd}
\frakX_{(T/p^n)} \ar{r}{g_{\frakX}} \ar{d}{f'} & \frakX \ar{d}{f} \\
\frakY_{(T/p^n)} \ar[swap]{r}{g_{\frakY}} & \frakY
\end{tikzcd}
$$
is commutative.
The log étaleness of $f:\frakX \ra \frakY$ implies the existence of a unique morphism $h:\frakY_{(T/p^n)} \ra \frakX$ fitting into the commutative diagram
$$
\begin{tikzcd}
T \ar{r}{i} & \frakX_n \ar{r} & \frakX \ar{d}{f} \\
\left (\frakY_{(T/p^n)} \right )_n \ar{r} \ar{u} & \frakY_{(T/p^n)} \ar{r} \ar{ur}{h} & \frakY
\end{tikzcd}
$$
The universal property of $\frakX_{(T/p^n)}$ applied to $h:\frakY_{(T/p^n)} \ra \frakX$ implies the existence of a unique morphism $f'':\frakY_{(T/p^n)} \ra \frakX_{(T/p^n)}$ such that the diagram
\begin{equation}\label{eqKham741}
\begin{tikzcd}
\frakY_{(T/p^n)} \ar{r}{h} \ar[swap]{d}{f''} & \frakX \\
\frakX_{(T/p^n)} \ar[swap]{ur}{g_{\frakX}} & 
\end{tikzcd}
\end{equation}
is commutative. Consider the commutative diagram
\begin{equation}\label{eqKham742}
\begin{tikzcd}
\frakY_{(T/p^n)} \ar{r} \ar[swap]{d}{f''} & \frakX \ar{r}{f} & \frakY \\
\frakX_{(T/p^n)} \ar[swap]{d}{f'} \ar{ur} & & \\
\frakY_{(T/p^n)} \ar{ruu}{h} \ar{uurr} & & 
\end{tikzcd}
\end{equation}
The universal property of $\frakY_{(T/p^n)}$ implies that
$f' \circ f'' = \op{Id}_{\frakY_{(T/p^n)}}.$
Now let $u=f''\circ f'.$ The commutative diagram
$$
\begin{tikzcd}
\frakX_{(T/p^n),n} \ar{dr} \ar{r}{f'_n} & \frakY_{(T/p^n),n} \ar{d} \ar{r}{h_n} & \frakX_n \\
 & T\ar[swap,hook]{ur}{i} & 
\end{tikzcd}
$$
implies that $h_n \circ f'_n=g_{\frakX,n}.$
Then, by \eqref{eqKham742}, we have the following commutative diagram :
$$
\begin{tikzcd}
 & \frakX_n \ar[hook]{rr} & & \frakX \ar{d}{f} \\
\left ( \frakX_{(T/p^n)} \right )_n \ar{ur}{g_{\frakX,n}} \ar{r} & \frakX_{(T/p^n)} \ar{urr}{h\circ f'} \ar{r} & \frakX \ar{r}{f} & \frakY
\end{tikzcd}
$$
By the log étaleness of $f,$ we deduce that $h \circ f'=g_{\frakX}.$ By \eqref{eqKham741}, $h=g_{\frakX}\circ f''$ so we get
$g_{\frakX}\circ  f''\circ f'=g_{\frakX}.$
We deduce by the universal property of $\frakX_{(T/p^n)}$ that
$f'' \circ f'=\op{Id}_{\frakX_{(T/p^n)}}.$
\end{proof}

\begin{proposition}[\cite{Oyama} 1.1.5 and \cite{DXU19} 3.8]\label{propdilflat}
Let $n$ be a positive integer, $\frakX$ and $\frakY$ two logarithmic $p$-adic formal schemes flat and locally of finite type over $\frakS,$ $f:\frakX \ra \frakY$ an $\frakS$-morphism which is flat on the underlying formal schemes, $T \ra \frakY_n$ a strict immersion and $S=\frakX\times_{\frakY}T.$ Then $f$ induces a canonical isomorphism
$$\frakY_{(T/p^n)}\times_{\frakY}\frakX \xrightarrow{\sim} \frakX_{(S/p^n)}.$$
\end{proposition}

\begin{proof}
Denote by $g_{\frakX}:\frakX_{(S/p^n)} \ra \frakX$ and $g_{\frakY}:\frakY_{(T/p^n)} \ra \frakY$ the canonical morphisms.
The morphism $f:\frakX \ra \frakY$ is flat on the underlying formal schemes, hence so is the projection $\frakY_{(T/p^n)}\times_{\frakY}\frakX \ra \frakY_{(T/p^n)}.$ We deduce that $\frakY_{(T/p^n)}\times_{\frakY}\frakX$ is flat over $\frakS.$
By the commutativity of the diagram
$$
\begin{tikzcd}
\left (\frakX_{(S/p^n)} \right )_n \ar{r} \ar{d} & \frakX_n \ar{r} & \frakY_n \\
S \ar{ur} \ar{r} & T\ar{ur} &
\end{tikzcd}
$$
and the universal property of $\frakY_{(T/p^n)},$ there exists a unique morphism $\frakX_{(S/p^n)} \ra \frakY_{(T/p^n)}$ fitting into the commutative diagram
$$
\begin{tikzcd}
\frakX_{(S/p^n)} \ar{r} \ar[swap]{d}{g_{\frakX}} & \frakY_{(T/p^n)} \ar{d}{g_{\frakY}} \\
\frakX \ar{r}{f} & \frakY
\end{tikzcd}
$$
It follows that there exists a unique morphism $f':\frakX_{(S/p^n)} \ra \frakY_{(T/p^n)}\times_{\frakY}\frakX$ making the following diagram commutative:
$$
\begin{tikzcd}
\frakX_{(S/p^n)} \ar{dr}{f'} \ar[bend right=-20]{drr} \ar[swap, bend right=30]{ddr}{g_{\frakX}}  & & \\
 & \frakY_{(T/p^n)}\times_{\frakY}\frakX \ar{r}{\alpha} \ar{d}{\beta} & \frakY_{(T/p^n)} \ar{d}{g_{\frakY}} \\
 & \frakX \ar{r}{f} & \frakY,
\end{tikzcd}
$$
where $\alpha$ and $\beta$ are the canonical projection.
Applying the universal property of $\frakX_{(S/p^n)}$ to the projection
$\beta:\frakY_{(T/p^n)}\times_{\frakY}\frakX \ra \frakX$
proves the existence of a unique morphism
$f'':\frakY_{(T/p^n)}\times_{\frakY}\frakX \ra \frakX_{(S/p^n)}$
such that the following diagram is commutative
$$
\begin{tikzcd}
\frakY_{(T/p^n)}\times_{\frakY}\frakX \ar{r}{\beta} \ar{d}{f''} & \frakX \\
\frakX_{(S/p^n)} \ar[swap]{ur}{g_{\frakX}} & 
\end{tikzcd}
$$
The commutativity of the diagram
$$
\begin{tikzcd}
\frakX_{(S/p^n)} \ar{r}{f'} \ar{d}{g_{\frakX}} & \frakY_{(T/p^n)} \times_{\frakY} \frakX \ar{d}{f''} \\
\frakX & \frakX_{(S/p^n)} \ar{l}{g_{\frakX}}
\end{tikzcd}
$$
and the universal property of $\frakX_{(S/p^n)}$ prove that
$f'' \circ f' =\op{Id}_{\frakX_{(S/p^n)}}.$
The commutativity of the diagram
$$
\begin{tikzcd}
\frakY_{(T/p^n)}\ar{dr} \times_{\frakY} \frakX \ar{r}{f''} \ar[swap]{dd}{\alpha} & \frakX_{(S/p^n)} \ar{d}{g_{\frakX}} \ar[bend right=-30]{rr} \ar{r}{f'} & \frakY_{(T/p^n)} \times_{\frakY} \frakX \ar{r}{\alpha} & \frakY_{(T/p^n)} \ar{dd}{g_{\frakY}} \\
& \frakX \ar{drr} & & \\
\frakY_{(T/p^n)} \ar{rrr}{g_{\frakY}} & & & \frakY
\end{tikzcd}
$$
and the universal property of $\frakY_{(T/p^n)}$ imply that
$\alpha \circ f' \circ f''=\alpha.$
We also have
$\beta \circ f' \circ f''=g_{\frakX} \circ f''=\beta.$
This proves that
$f' \circ f'' =\op{Id}_{\frakY_{(T/p^n)} \times_{\frakY} \frakX}.$
\end{proof}

\section{Groupoids}

\begin{parag}\label{paragJapon1}
In this section, we fix a perfect field $\kappa$ of characteristic $p$ and denote by $W$ its ring of Witt vectors. We equip $\op{Spf}W$ with the trivial logarithmic structure. For a logarithmic $p$-adic formal scheme $\frakX$ over $\op{Spf}W$ and a positive integer $n,$ we denote by $\frakX_n$ the logarithmic scheme obtained from $\frakX$ by reduction modulo $p^n.$ We use the corresponding roman letter $X$ for $\frakX_1.$
\end{parag} 

\begin{parag}\label{parag86}
In this section, let $P \ra Q$ be a morphism of fs monoids, $f:(\frakX,Q)\ra (\frakS,P)$ a log smooth morphism of framed fs logarithmic $p$-adic formal schemes \eqref{def72}, locally of finite type over $\op{Spf}W$ \eqref{dxuflat}. We make one of the following two assumptions : either $\frakS$ is log flat over $\op{Spf}W,$ or $\frakS$ is flat over $\op{Spf}W$ and $f$ is integral (\cite{Ogus2018} III 2.5.1). Note that, in both cases, by \ref{Wflat} and (\cite{Ogus2018} IV 4.3.5), both $\frakX$ and $\frakS$ are flat over $\op{Spf}W$ \eqref{dxuflat}.

Let $r$ be a nonnegative integer. By \ref{cor79}, we have a factorization
$$\begin{tikzcd}
 & \frakX_{\frakS,\left [Q\right ]}^{r+1}\ar{d} \\
\frakX \ar{r} \ar{ur}{\Delta(r)}&  \frakX^{r+1}_{\frakS}
\end{tikzcd}$$
where $\frakX^{r+1}_{\frakS}$ denotes the fiber product, in the category $\boldsymbol{\op{LFS}}$ \eqref{parag63}, of $\frakX$ over $\frakS$ with itself $r+1$ times, $\Delta(r):\frakX\ra \frakX_{\frakS,[Q]}^{r+1}$ is a strict immersion and $\frakX_{\frakS,[Q]}^{r+1}\ra \frakX_{\frakS}^{r+1}$ is log étale. Set $\frakY(r)=\frakX_{\frakS,[Q]}^{r+1}.$ The canonical morphism $\frakY(r) \ra \frakS$ is log smooth and, in particular, locally of finite type, so $\frakY(r)$ is locally of finite type over $\op{Spf}W.$ The projections $\frakY(r) \ra \frakX$ are strict (\ref{cor76}) and log smooth hence smooth. It follows that $\frakY(r)$ is flat over $\op{Spf}W.$ 

For any integer $n\ge 1,$ denote by $R_{\frakX,n}(r)$ the dilatation $\frakY(r)_{(\frakX_n/p^n)}$ defined in \ref{erafdil1}. By the universal property, there exists a unique strict immersion of logarithmic formal schemes $\frakX\ra R_{\frakX,n}(r)$ making the following diagrams commutative
$$
\begin{tikzcd}
 & R_{\frakX,n}(r)\ar{d} \\
 & \frakY(r)\ar{d} \\
\frakX\ar{uur}\ar{ur}\ar{r} & \frakX^{r+1}_{\frakS}
\end{tikzcd}
$$
We extend the notation to $n=0$ by setting $R_{\frakX,0}(r)=\frakY(r).$ For every positive integer $k,$ we denote by $\frakX_k$ and $R_{\frakX,n}(r)_k$ the logarithmic schemes obtained from $\frakX$ and $R_{\frakX,n}(r)$ respectively, by reduction modulo $p^k$ (as in \ref{prop69}), by $P_{\frakX/\frakS,n}(r)_k$ the PD-envelope of the strict immersion $\frakX_k\ra R_{\frakX,n}(r)_k$ and by $\ov{\calI}_{\frakX/\frakS,n,k}(r)$ the PD-ideal of $P_{\frakX/\frakS,n}(r)_k.$ We equip $P_{\frakX/\frakS,n}(r)_k$ with the logarithmic structure pullback of that of $R_{\frakX,n}(r)_k.$ By extension of scalars (\cite{Ber74} I 2.8.2 and \cite{Ogus78} 3.20.8), the canonical morphism
$$
P_{\frakX/\frakS,n}(r)_k \ra P_{\frakX/\frakS,n}(r)_{k+1}\times_{\frakS_{k+1}}\frakS_k
$$
is an isomorphism for any positive integer $k.$
The schemes $(P_{\frakX/\frakS,n}(r)_k)_{k\ge 1}$ form then a $p$-adic system of schemes. We denote by $P_{\frakX/\frakS,n}(r)$ its limit, equip it with the logarithmic structure pullback of that of $\frakY(r).$ We obtain the commutative diagram:
$$
\begin{tikzcd}
 & P_{\frakX/\frakS,n}(r)\ar{d} \\
 & R_{\frakX,n}(r)\ar{d} \\
 & \frakY(r)\ar{d} \\
\frakX \ar[bend right =-20]{uuur} \ar{uur}\ar{ur}\ar{r} & \frakX^{r+1}_{\frakS}
\end{tikzcd}
$$
The canonical morphism $X\ra P_{\frakX/\frakS,n}(r)_1,$ is a universal homeomorphism on the underlying topological spaces and hence so is $\frakX\ra P_{\frakX/\frakS,n}(r).$ We identify the underlying small étale sites via this homeomorphism. We denote by $\calP_{\frakX/\frakS,n}(r)$ the structural ring of $P_{\frakX/\frakS,n}(r)$ and, for any integers $k\ge 1$ and $l\ge 0,$ $\calP_{\frakX/\frakS,n}^{\{ l \}}(r)_k=\calP_{\frakX/\frakS,n}(r)_k/\ov{\calI}_{\frakX/\frakS,n,k}(r)^{[l+1]}$ and $\ov{\calI}^{\{ l \}}_{\frakX/\frakS,n,k}(r)=\ov{\calI}_{\frakX/\frakS,n,k}(r)/\ov{\calI}_{\frakX/\frakS,n,k}(r)^{[l+1]}.$

When $r=1,$ we drop $(r)$ from the notation we just introduced. Note that $\calP_{\frakX/\frakS,n,k}^{\{ 0 \}}=\Ox_{\frakX_k}$ for every positive integer $k.$

By Kato's construction of the PD envelope, given in (\cite{Kat89} 5.6), and in view of \ref{prop45}, for all integers $k,r\ge 1,$ Kato's logarithmic PD-envelope of $\frakX_k \ra (\frakX_k)^{r+1}$ and the logarithmic scheme $P_{\frakX/\frakS,0}(r)_k$ are the same.
It follows, by (\cite{Kat89} 5.8.1), that, for any integer $k\ge 1,$ there exists a canonical isomorphism
\begin{equation}\label{eqtakrizIIsquare}
\ov{\calI}^{\{1\}}_{\frakX/\frakS,0,k} \xrightarrow{\sim} \omega^1_{\frakX_k/\frakS_k}.
\end{equation}
\end{parag}

\begin{parag}
Let $k$ be a positive integer. Since $f:\frakX \ra \frakS$ is log smooth and $\frakY=\frakX\times_{\frakS,[Q]}^{\op{log}}\frakX \ra \frakX\times_{\frakS}^{\op{log}}\frakX$ is log étale (\ref{prop74}), the projections $\frakY \ra \frakX$ are log smooth. Since they are also strict (\ref{cor76}), they are smooth. It follows that the exact diagonal immersion $\frakX_k \ra \frakY_k$ is regular. Let $\frakX_k \ra \frakU_k \ra \frakY_k$ be a factorization of $\frakX_k \ra \frakY_k$ into a closed immersion followed by an open one and let $\calI_k$ be the ideal of $\frakX_k \ra \frakU_k.$
The canonical morphism of $\Ox_{\frakX_k}$-modules
\begin{equation}\label{grsym}
\op{Sym}_{\Ox_{\frakX_k}}^n\left (\calI_k/\calI_k^2\right ) \xrightarrow{\sim} \calI_k^n/\calI_k^{n+1}
\end{equation}
is then an isomorphism for all $n\ge 1.$
\end{parag}

\begin{proposition}\label{erafprop13}
For any integer $n\ge 1,$ let $R_n=\left (R_{\frakX,n}\right )_1$ \eqref{parag86}. The canonical morphism $R_n \ra \frakY_n$ factors uniquely through the strict diagonal immersion $\frakX_n \ra \frakY_n.$ In addition, the morphism $R_n \ra \frakX_n$ is affine.
\end{proposition}

\begin{proof}
This follows from the definition of $R_{\frakX,n}$ \eqref{erafdil1}.
\end{proof}

\begin{remark}\label{rem87}
Let $r$ and $s$ be nonnegative integers. We have a canonical isomorphism
\begin{equation}\label{eq871}
\frakY(r)\times_{\frakX}\frakY(s)\xrightarrow{\sim} \frakY(r+s),
\end{equation}
where $\frakY(r)$ (resp. $\frakY(s)$) is considered as a logarithmic formal scheme over $\frakX$ via the $(r+1)^{th}$ projection (resp. first projection). For every integer $n\ge 1,$ by the universal properties of $P_{\frakX/\frakS,n}$ and $R_{\frakX,n},$ the isomorphism \eqref{eq871} induces isomorphisms $$R_{\frakX,n}(r)\times_{\frakX}R_{\frakX,n}(s)\xrightarrow{\sim} R_{\frakX,n}(r+s),\ P_{\frakX/\frakS,n}(r)\times_{\frakX}P_{\frakX/\frakS,n}(s)\xrightarrow{\sim} P_{\frakX/\frakS,n}(r+s),$$
where $R_{\frakX,n}(r)$ and $P_{\frakX/\frakS,n}(r)$ (resp. $R_{\frakX,n}(s)$ and $P_{\frakX/\frakS,n}(s)$) are considered as a logarithmic formal scheme over $\frakX$ via the $(r+1)^{th}$ projection (resp. first projection) (\cite{DXU19} 4.10). These isomorphisms fit into the following commutative diagrams
$$
\begin{tikzcd}
P_{\frakX/\frakS,n}(r)\times_{\frakX}P_{\frakX/\frakS,n}(s) \ar{r}{\sim} \ar{d} & P_{\frakX/\frakS,n}(r+s) \ar{d} \\
R_{\frakX,n}(r)\times_{\frakX}R_{\frakX,n}(s) \ar{r}{\sim} \ar{d} & R_{\frakX,n}(r+s) \ar{d} \\
\frakY(r)\times_{\frakX}\frakY(s) \ar{r}{\sim} & \frakY(r+s)
\end{tikzcd}
$$
We now introduce the formal groupoid structures on $R_{\frakX,n}$ and $P_{\frakX/\frakS,n}.$
\end{remark}

\begin{proposition}\label{Kokoprop126}
Let $n\ge 1$ be an integer. Consider the following morphisms:
\begin{enumerate}
\item The morphism
$
\alpha  :P_{\frakX/\frakS,n}\times_{\frakX}P_{\frakX/\frakS,n} \xrightarrow{\sim} P_{\frakX/\frakS,n}(2) \ra P_{\frakX/\frakS,n}$
(resp. $\alpha  :R_{\frakX,n}\times_{\frakX}R_{\frakX,n} \xrightarrow{\sim} R_{\frakX,n}(2) \ra R_{\frakX,n}),
$
induced by the $(1,3)$-projection $\frakY(2)\ra \frakY.$
\item The morphism $\iota:\frakX\ra P_{\frakX/\frakS,n}$ (resp. $\iota:\frakX\ra R_{\frakX,n}$).
\item The morphism
$
\eta :P_{\frakX/\frakS,n} \ra P_{\frakX/\frakS,n}
(\text{resp.}\ \eta  :R_{\frakX,n} \ra R_{\frakX,n})$
induced by the morphism $\frakY\ra \frakY$ that exchanges factors.
\end{enumerate} 
Then the morphisms $\alpha,\ \iota$ and $\eta$ define a $p$-adic $\frakX$-groupoid structure on $P_{\frakX/\frakS,n}$ (resp. $R_{\frakX,n}$) (\cite{DXU19} 4.7).
\end{proposition}

\begin{proof}
We just check that the canonical morphisms $P_{\frakX/\frakS,n} \ra \frakY,$ $R_{\frakX,n} \ra \frakY$ and $Q_{\frakX} \ra \frakY$ factor, as maps of topological spaces, through the exact diagonal immersion $\frakX \ra \frakY.$ If $n \ge 1,$ the assertion for $P_{\frakX/\frakS,n}$ and $R_{\frakX,n}$ follows from \ref{erafprop13}. It remains to check it for $P_{\frakX/\frakS,0}.$ Let $P_0$ be the logarithmic scheme $\left (P_{\frakX/\frakS,0} \right )_1.$ Since $P_0$ and $P_{\frakX/\frakS,0}$ have the same underlying topological spaces.The result then follows from the fact that, if $x$ is a local section of the ideal of the canonical immersion $X \ra P_0,$ then
$x^p=p!x^{[p]}=0.$
\end{proof}

\begin{parag}\label{paragomega}
Let $\frakG$ be a $p$-adic $\frakX$-groupoid (\cite{DXU19} 4.7) and $\omega:\frakG_{\text{zar}} \ra \frakX_{\text{zar}}$ the morphism of topoi induced by the factorization of the morphism of underlying topological spaces of $\frakG \ra \frakX^2_{\frakS}$ through the diagonal. Let $q_1,q_2:\frakG \ra \frakX$ be the projections. Then
$$\omega_*\Ox_{\frakG}=q_{1*}\Ox_{\frakG}=q_{2*}\Ox_{\frakG}.$$
In this way, we see $\Ox_{\frakG}$ as a bialgebra of $\frakX_{\text{zar}}.$ The groupoid structure on $\frakG$ induces a formal Hopf algebra structure on $\Ox_{\frakG}.$ For any positive integers $r$ and $n,$ we denote by $\calR_{\frakX,n}(r)$ the formal Hopf algebra $\Ox_{R_{\frakX,n}(r)}$. 
\end{parag}

\begin{proposition}\label{locdescRQ}
Let $n \ge 1$ be an integer and consider $R_{\frakX,n}$ as a logarithmic formal scheme over $\frakX$ via the first projection $q_1: R_{\frakX,n} \ra \frakX.$ Suppose that there exists a chart $\alpha:P \ra M$ of $f:\frakX \ra \frakS,$ fitting into a commutative diagram
$$
\begin{tikzcd}
\frakX \ar{r} \ar{d} & \frakS \ar{d} \\
B \langle M \rangle \ar{r} \ar{d} & B\langle P \rangle \ar{d} \\
\left [Q \right ] \ar{r} & \left [P \right ],
\end{tikzcd}
$$
and satisfying the conditions :
\begin{enumerate}
\item $\alpha^{gp}$ is injective and the torsion subgroup of $\op{coker}\alpha^{gp}$ is finite and of order coprime with $p.$
\item The morphism $g:\frakX \ra \frakS \times_{B\langle P\rangle}B\langle M\rangle,$ induced by $f$ and the chart $\frakX \ra B\langle M\rangle,$ is étale and strict.
\end{enumerate}
Let $\frakX \ra \frakU \ra \frakY$ be a factorization of the immersion $\Delta:\frakX \ra \frakY$ into a closed immersion followed by an open one and let $\calI_{\frakX/\frakS}$ be the ideal of $\frakX \ra \frakU.$ Suppose that there exist $m_1,\hdots,m_d \in \Gamma\left (\frakX,\calM_{\frakX}\right )$ lifting local coordinates of $X\ra S.$ Let $\eta_1,\hdots,\eta_d \in \Gamma \left (\frakX,\Delta^{-1}\calI_{\frakX/\frakS} \right )$ be as defined in \ref{parag77} and consider them as sections of $\Ox_{R_{\frakX,n}}.$ Then, there exists an isomorphism of $\Ox_{\frakX}$-algebras:
\begin{equation}\label{isoR}
\Ox_{\frakX}\{x_1,\hdots,x_d\} \xrightarrow{\sim} q_{1*}\Ox_{R_{\frakX,n}}
\end{equation}
sending $x_i$ to $\frac{\eta_i}{p^n}$.
\end{proposition}

\begin{proof}
Set $\frakT=\frakS \times_{B\langle P\rangle}B\langle M\rangle.$ By \ref{prop73} and \ref{prop74}, we have
$$\frakY = \left (\frakX \times_{\frakS}^{\op{log}}\frakX \right )\times _{B\langle \left (M\oplus_PM\right )^{sat}\rangle }B\langle N \rangle,$$
$$\frakT \times_{\frakS,[M]}^{\op{log}}\frakT = \left (\frakT \times_{\frakS}^{\op{log}}\frakT \right )\times _{B\langle \left (M\oplus_PM\right )^{sat}\rangle }B\langle N \rangle=\frakS \times_{B\langle P\rangle} B\langle N\rangle,$$
where $N$ is the inverse image of $M$ by
$\left (M\oplus_PM\right )^{gp} \ra M^{gp},\ (x,y)\mapsto x+y.$
The morphism $g:\frakX \ra \frakT$ induces an étale and strict morphism
$\frakY \ra \left (\frakX \times_{\frakS}^{\op{log}}\frakT \right )\times _{B\langle \left (M\oplus_PM\right )^{sat}\rangle }B\langle N \rangle= \frakX \times_{\frakS,[M]}^{\op{log}}\frakT,$
fitting into a commutative diagram
$$
\begin{tikzcd}
\frakX \ar[swap]{d}{\Delta} \ar{dr} & \\
\frakY \ar{r} & \frakX \times_{\frakS,[M]}^{\op{log}}\frakT.
\end{tikzcd}
$$
By \ref{propdiletale}, we get an isomorphism
$R_{\frakX,n}=\frakY_{(\frakX_n/p^n)} \xrightarrow{\sim} \left (\frakX \times_{\frakS,[M]}^{\op{log}}\frakT\right )_{(\frakX_n/p^n)}.$
The strict and étale morphism $g$ induces a strict and étale, hence flat, morphism 
$h:\frakX \times_{\frakS,[M]}^{\op{log}}\frakT \ra \frakT \times_{\frakS,[M]}^{\op{log}}\frakT$
fitting into the commutative diagram
\begin{equation}\label{diagest1}
\begin{tikzcd}
\frakX \ar{r}{g} \ar{d} & \frakT \ar{d} \\
\frakX \times_{\frakS,[M]}^{\op{log}}\frakT \ar{r}{h} \ar{d} & \frakT \times_{\frakS,[M]}^{\op{log}}\frakT \ar{d} \\
\frakX \ar{r}{g} & \frakT,
\end{tikzcd}
\end{equation}
where the upper vertical arrows are the exact diagonal morphisms and the lower vertical arrows are the projections on the first factor. Since the lower square and the outer rectangle of \eqref{diagest1} are cartesian, so is the upper square.
Then, by \ref{propdilflat}, $h$ induces an isomorphism
\begin{alignat*}{2}
\left (\frakX \times_{\frakS,[M]}^{\op{log}}\frakT\right )_{(\frakX_n/p^n)}  & \xrightarrow{\sim} \left (\frakT \times_{\frakS,[M]}^{\op{log}}\frakT\right )_{(\frakT_n/p^n)} \times_{\frakT \times_{\frakS,[M]}^{\op{log}}\frakT}\left (\frakX \times_{\frakS,[M]}^{\op{log}}\frakT\right ) \\
& \xrightarrow{\sim} \frakX \times_{\frakT} \left (\frakT \times_{\frakS,[M]}^{\op{log}}\frakT\right )_{(\frakT_n/p^n)},
\end{alignat*}
where $\left (\frakT \times_{\frakS,[M]}^{\op{log}}\frakT\right )_{(\frakT_n/p^n)}$ is considered as a logarithmic scheme over $\frakT$ via the first projection.
We can reduce to the case $\frakX=\frakT=\frakS \times_{B\langle P\rangle}B\langle M\rangle.$ We can also suppose that $\frakS=\op{Spf}A$ for a $p$-adic algebra $A.$ In this case,
\begin{alignat*}{2}
\frakX &= \frakS \times_{B\langle P\rangle}B\langle M\rangle = \op{Spf} \left ( A \widehat{\otimes}_{\Z_p\langle P \rangle} \Z_p \langle M\rangle \right ) = \op{Spf} C, \\
\frakY &= \frakS \times_{B\langle P \rangle} B\langle N \rangle = \op{Spf} \left ( A\widehat{\otimes}_{\Z_p\langle P\rangle } \Z_p\langle N\rangle \right ) = \op{Spf} D,
\end{alignat*}
where $\widehat{\otimes}$ is the $p$-adically completed tensor product. Let $I$ be the kernel of the exact diagonal $D\ra C$ and $a_1,\hdots,a_d\in I$ corresponding to $\eta_1,\hdots,\eta_d.$ Then, by definition of $R_{\frakX,n}$ \eqref{dilaff}, we have
\begin{alignat*}{2}
R_{\frakX,n} &= \op{Spf}\left ( \left ( \frac{D\{ x_1,\hdots,x_d\}}{\left (p^nx_1-a_1,\hdots,p^nx_d-a_d \right )} \right )_{/p\text{-tor}} \right ),
\end{alignat*}
where $p$-tor denotes the $p$-torsion ideal \eqref{Not2}.
The result then follows from the lemma \ref{lemR} below.
\end{proof}

\begin{lemma}\label{liftbase}
Keep the assumptions of \ref{parag86} and suppose that $\frakS=\op{Spf}A$ for a $p$-adic Noetherian algebra $A$ and $\frakX=\frakS \times_{B\langle P\rangle}B\langle M\rangle=\op{Spf} C.$ The morphism $\frakX \ra \frakS$ corresponds then to a continuous morphism $\xi:A \ra C.$ In this case, $\frakY=\frakS \times_{B\langle P \rangle} B\langle N \rangle=\op{Spf} D,$ where $N$ is the inverse image of $M$ by
$$\left (M\oplus_PM\right )^{gp} \ra M^{gp},\ (x,y)\mapsto x+y,$$
and
$$
D=A\widehat{\otimes}_{\Z_p\langle P\rangle } \Z_p\langle N\rangle.
$$
Let $I$ be the ideal of the exact diagonal $D \ra C$ \eqref{BM}, $I_k$ its reduction modulo $p^k$ for $k\ge 1$ and $a_1,\hdots,a_d\in I$ whose images in $I_1/I_1^2$ are a basis of the $\left ( C/pC \right )$-module $I_1/I_1^2.$ Then, for every positive integer $k,$ the images of $a_1,\hdots,a_d$ in $I_k/I_k^2$ form a basis of the $\left ( C/p^kC \right )$-module $I_k/I_k^2.$
\end{lemma}

\begin{proof}
Let $\Omega=I/I^2$ and consider the morphism
\begin{equation}\label{eq1295}
C^d \ra \Omega,\ (\beta_1,\hdots,\beta_d) \mapsto \sum_{i=1}^d\beta_ia_i.
\end{equation}
Let $x\in I$ such that $p\ov{x}=0$ in $\Omega$ and $k$ a positive integer. There exist $x_i,y_i\in I$ such that $px=\sum x_iy_i$ in $I.$ It follows that, in $I_k/I_k^2,$ we have $px=0.$ Since $\frakX_k \ra \frakS_k$ is log smooth and by \eqref{iso}, $I_k/I_k^2$ is a free $C/p^kC$-module. It follows that the image of $x$ in $I_k/I_k^2$ belongs to $p^{k-1}(I_k/I_k^2).$ The image of $x$ in $I/I^2$ belongs then to $ p^{k-1}I/I^2.$ This being true for every $k$ and $I/I^2$ being separated, we deduce that $x=0.$ It follows that $I/I^2$ is flat over $\Z_p$ and so, by (\cite{Bourbaki} III §5 theorem 1), the canonical morphism
\begin{equation}\label{eq1296}
\left ((p^k)/(p^{k+1})\right ) \otimes_{\mathbb{F}_p} \left (\Omega/p\Omega\right ) \xrightarrow{\sim} p^k\Omega/p^{k+1}\Omega
\end{equation}
is an isomorphism.
Since \eqref{eq1295} is an isomorphism modulo $p,$ we conclude by \eqref{eq1296} and (\cite{Bourbaki} III §2, 8, corollary 3). 
\end{proof}

\begin{lemma}\label{lemR}
Keep the assumptions of \ref{liftbase} and recall the notation \ref{Not2}. Let $r$ be a positive integer. The morphism of topological $C$-algebras
$$\varphi:C\{x_1,\hdots,x_d\} \ra \left ( \frac{D\{ x_1,\hdots,x_d\}}{\left (p^rx_1-a_1,\hdots,p^rx_d-a_d \right )} \right )_{/p\text{-tor}},$$
sending $x_i$ to $x_i$ and $C$ to $D$ via the first projection, is an isomorphism.
\end{lemma}

\begin{proof}
Denote by $\pi:C \ra D$ the morphism corresponding to the first projection $\frakY \ra \frakX$ and $\Delta:D \ra C$ the morphism corresponding to the exact diagonal $\frakX \ra \frakY.$ Note that $\Delta \circ \pi=\op{Id}_{C}.$
For $g\in D \{x_1,\hdots,x_d\}$ and $J=(J_1,\hdots,J_d)\in \N^d,$ set
$x^J=\prod_{j=1}^dx_j^{J_j}$ and $a^J=\prod_{j=1}^da_j^{J_j}$
and denote by $g_J$ the coefficient of $x^J$ in $g.$
Let $f=\sum_{J\in \N^d}\alpha_Jx^J\in \op{Ker}\varphi$ with $\alpha_J\in C$ for all $J\in \N^d.$ Then there exists a positive integer $l$ and $g_1,\hdots,g_d \in D \{x_1,\hdots,x_d\}$ such that
\begin{equation}\label{eq1281}
\sum_{J\in \N^d}p^l\pi(\alpha_J)x^J=\sum_{i=1}^dg_i(p^rx_i-a_i).
\end{equation}
We will prove by induction on $n$ that $p^l\alpha_{J}=0$ for all $J\in \N^d$ such that $|J|=n$ and that
$$
\sum_{\substack{1\le i\le d \\ |J|=n}}g_{i,J}a^{J+\epsilon_i}=0,
$$
where $\epsilon_i\in \N^d$ is the multi-index whose all coefficients are zero except for the $i$th which is equal to 1.
Taking the constant coefficients of both sides in \eqref{eq1281}, we get
$$
p^l\pi (\alpha_0)=-\sum_{i=1}^dg_{i,0}a_i.
$$
Applying $\Delta$ to both sides, we get
$p^l\Delta\circ \pi (\alpha_0)=p^l\alpha_0=0.$
It follows then that
$\sum_{i=1}^dg_{i,0}a_i=0.$
The first step of the induction is thus proven.
Let $n\ge 0$ be an integer and suppose that $p^l\alpha_{J}=0$ for all $J\in \N^d$ such that $|J|=n$ and that
$$
\sum_{\substack{1\le i\le d \\ |J|=n}}g_{i,J}a^{J+\epsilon_i}=0.
$$
Taking the monomials of degree $n+1$ in \eqref{eq1281}, we get
$$
\sum_{|J|=n+1}p^l\pi(\alpha_J)x^J=\sum_{\substack{1\le i\le d\\ |J|=n}}p^rg_{i,J}x^{J+\epsilon_i}-\sum_{\substack{1\le i\le d\\ |J|=n+1}}g_{i,J}a_ix^J.
$$
Taking $x_i=a_i$ for all $1\le i\le d$ and using the induction hypothesis, we get
\begin{equation}\label{eq1284}
\sum_{|J|=n+1}p^l\pi(\alpha_J)a^J=-\sum_{\substack{1\le i\le d\\ |J|=n+1}}g_{i,J}a^{J+\epsilon_i}.
\end{equation}
It follows that, in $I^{n+1}/I^{n+2},$
$\sum_{|J|=n+1}p^l\pi(\alpha_J)a^J=0.$
Let $k$ be a positive integer. In the $\left (C/p^kC\right )$-module $I_k^{n+1}/I_k^{n+2},$ we have
$\sum_{|J|=n+1}p^l\alpha_Ja^J=0.$
By \ref{liftbase}, the family $(a_1,\hdots,a_d)$ is a basis of the $\left (C/p^kC\right )$-module $I_k/I_k^2.$ Then, by \eqref{grsym}, $(a^J)_{|J|=n+1}$ is free in the $\left (C/p^kC\right )$-module $I_k^{n+1}/I_k^{n+2}.$ It follows that
$p^l\alpha_J=0$ for $|J|=n+1.$
Then, by \eqref{eq1284}, we get
$$\sum_{\substack{1\le i\le d\\ |J|=n+1}}g_{i,J}a^{J+\epsilon_i}=0.$$
This concludes the induction. Since $C$ is flat over $\Z_p,$ we get $\alpha_J=0$ for all $J\in \N^d.$ This concludes the injectivity of $\varphi.$ The surjectivity follows from the fact that, for $y\in D,$
$y=\pi \circ \Delta(y)+y-\pi \circ \Delta(y)$
and $y-\pi \circ \Delta(y)\in I.$
\end{proof}

\begin{parag}\label{par1212}
Keep the assumptions and notation of \ref{parag86}. By \ref{formaletaleframelift}, there exists an étale covering $\left ( \frakS_i \ra \frakS \right )_{i\in I},$ and, for every $i\in I,$ an étale covering $\left (\frakX_{ij} \ra \frakX_i:=\frakX\times_{\frakS}\frakS_i \right )_{j\in I_i}$ and charts $P_i \ra M_{ij}$ of $\frakX_{ij} \ra \frakS_i,$ fitting into a commutative diagram
$$
\begin{tikzcd}
\frakX_{ij} \ar{r} \ar{d} & \frakS_i \ar{d} \\
B\langle M_{ij} \rangle \ar{r} \ar{d} & B\langle P_i \rangle \ar{d} \\
\left [Q\right ] \ar{r} & \left [ P\right ], 
\end{tikzcd}
$$
and satisfying:
\begin{enumerate}
\item $P_i^{gp} \ra M_{ij}^{gp}$ is injective and the torsion subgroup of its cokernel has a finite order coprime with $p.$
\item The morphism $\frakX_{ij} \ra \frakS_i\times_{B\langle P_i \rangle} B\langle M_{ij} \rangle,$ induced by $\frakX_{ij}\ra \frakS_i$ and the chart $\frakX_{ij} \ra B\langle M_{ij}\rangle,$ is étale and strict.
\end{enumerate}
For any $i\in I$ and $j\in I_i,$ let $X_{ij}$ and $S_i$ be the special fibers of $\frakX_{ij}$ and $\frakS_i$ respectively. After eventually shrinking $X_{ij},$ we can suppose that $X_{ij} \ra S_i$ has local coordinates $m_1,\hdots,m_d\in \Gamma \left (X_{ij},\calM_{X_{ij}} \right ).$ By \ref{prop69}, the morphism $\calM_{\frakX_{ij}}\ra \calM_{X_{ij}}$ is surjective. Hence, after eventually a furthermore shrinking of $X_{ij},$ there exists $\widetilde{m}_1,\hdots,\widetilde{m}_d\in \Gamma\left (\frakX_{ij},\calM_{\frakX_{ij}}\right )$ that lift $m_1,\hdots,m_d.$ We can then define $\widetilde{\eta}_1,\hdots,\widetilde{\eta}_d\in \Gamma\left (\frakU_i,\Delta^{-1}\calI_{\frakX/\frakS} \right )$ from $\widetilde{m}_d,\hdots,\widetilde{m}_d$ as in \ref{parag77}. This proves that the isomorphism \eqref{isoR} exists étale locally on $\frakX$ and $\frakS.$
\end{parag}

\begin{proposition}\label{Qlogflat}
For any positive integer $n,$ the logarithmic formal scheme $R_{\frakX,n}$ is log flat and flat over $\frakX$ and flat over $\op{Spf}W$ \eqref{dxuflat}.
\end{proposition}

\begin{proof}
By \ref{locdescRQ} and \ref{par1212}, for every positive integer $k,$ $R_{\frakX,n,k} \ra \frakX_k$ is flat. Since it is also strict, it is log flat. We conclude by the flatness of $\frakX \ra \op{Spf}W.$
\end{proof}

\begin{proposition}\label{propfree}
Let $n$ be a positive integer and consider $P_{\frakX/\frakS,n}$ as a logarithmic formal scheme over $\frakX$ by the first projection $q_1:P_{\frakX/\frakS,n} \ra \frakX.$
Under the assumptions of \ref{locdescRQ}, we have an isomorphism of $\Ox_{\frakX}$-algebras
$$\Ox_{\frakX}\langle \langle x_1,\hdots,x_d \rangle \rangle \xrightarrow{\sim} \calP_{\frakX/\frakS,n}$$
sending $x_i$ to $\frac{\eta_i}{p^n},$ where $\Ox_{\frakX}\langle \langle x_1,\hdots,x_d \rangle \rangle$ is the $p$-adic completion of the PD algebra $\Ox_{\frakX}\langle x_1,\hdots,x_d \rangle.$
\end{proposition}

\begin{proof}
Let $k$ be a positive integer and $\eta_{i,k}$ the image of $\frac{\eta_i}{p^n}$ in $\calP_{\frakX/\frakS,n,k}.$ By the local description of $R_{\frakX,n}$ given in \ref{locdescRQ}, there exists an isomorphism of $\Ox_{\frakX_k}$-algebras
$$
\Ox_{\frakX_k}\langle x_1,\hdots,x_d \rangle \xrightarrow{\sim} \calP_{\frakX/\frakS,n,k}
$$
sending $x_i$ to $\eta_{i,k}.$
The result then follows by taking the projective limit for $k\ge 1.$
\end{proof}

\begin{proposition}\label{propflat}
For any nonnegative integers $n$ and $r,$ the formal scheme $P_{\frakX / \frakS,n}(r)$ is flat and log flat over $\frakX$ \eqref{logflatdef} and flat over $\op{Spf}W$ \eqref{dxuflat}.
\end{proposition}

\begin{proof}
By hypothesis, $\frakX$ is flat over $\op{Spf}W.$
By \ref{propfree}, the projections $P_{\frakX/\frakS,n,k}(r) \ra \frakX_k$ are flat for every positive integer $k.$ Since they are also strict, they are log flat. It follows that $P_{\frakX / \frakS,n,k}(r)$ is log flat over $\op{Spec}\left (W/(p^k)\right )$ for every positive integer $k.$ The result follows \eqref{logflatdef}.
\end{proof}

\begin{parag}\label{loccoord}
We will often make the following hypothesis:
suppose that we have local coordinates $m_1,\hdots,m_d\in \Gamma(X,\calM_X)$ for $X\ra S$ (\ref{P2}) that lift to local sections $\widetilde{m}_1,\hdots,\widetilde{m}_d\in \Gamma(\frakX,\calM_{\frakX}).$ For every $1\le i\le d,$ let $m_i'=\pi^{\flat}m_i$ \eqref{diag51} and $\widetilde{\eta}_i=\eta(\widetilde{m}_i) \in \Gamma(\frakX,\Delta^{-1}\calI_{\frakX/\frakS})$ be as defined in \ref{parag77} and $\eta_i$ the reduction of $\widetilde{\eta}_i$ modulo $p.$ Suppose also that we have a chart of $\frakX \ra \frakS$ satisfying the conditions of \ref{locdescRQ}. Let $n$ be a positive integer. By \ref{locdescRQ} and  \ref{propfree}, we have isomorphisms:
\begin{alignat}{2}
\calR_{\frakX,n} & \xrightarrow{\sim} \Ox_{\frakX}\left \{ \frac{\widetilde{\eta}_1}{p^n},\hdots,\frac{\widetilde{\eta}_d}{p^n} \right \}, \\
\calP_{\frakX/\frakS,n} & \xrightarrow{\sim} \Ox_{\frakX}\left \langle \left \langle \frac{\widetilde{\eta}_1}{p^n},\hdots,\frac{\widetilde{\eta}_d}{p^n} \right \rangle \right \rangle. \label{eqPkraz6}
\end{alignat}
Let $\eta_{i(r),n}$ be the image of $\frac{\widetilde{\eta}_i}{p^n}$ in $\calR_n:=\calR_{\frakX,n}/p\calR_{\frakX,n}$. For $n=1,$ we denote $ \eta_{i(r),1}$ simply by $\eta_{i(r)}.$ For any $I\in \N^d,$ we set
$$
\eta^I=\prod_{i=1}^d\eta_i^{I_i},\ \eta_{(r)}^I=\prod_{i=1}^d\eta_{i(r)}^{I_i}.
$$
If $X'$ lifts to a log smooth logarithmic $\frakS$-scheme $\frakX',$ we define $\eta_i',$ $\widetilde{\eta}_i'$ and $\eta_{i(r)}'$ in a similar way from $m_i'.$ We also denote by $(\partial_1',\hdots,\partial_d')$ the dual basis of $(\op{dlog}m_1',\hdots,\op{dlog}m_d').$
\end{parag}

\begin{proposition}\label{HopffrakR}
Let
$
\delta:\calR_{\frakX,n} \ra \calR_{\frakX,n}\otimes_{\Ox_{\frakX}}\calR_{\frakX,n},$ 
$\pi:\calR_{\frakX,n} \ra \Ox_{\frakX}$ and
$\sigma:\calR_{\frakX,n} \ra \calR_{\frakX,n}
$
be the morphisms defining the Hopf algebra structure on $\calR_{\frakX,n}$ \eqref{Kokoprop126}. Under the hypothesis \ref{loccoord},
we have
\begin{alignat*}{2}
\delta(\widetilde{\eta}_i) &= 1\otimes \widetilde{\eta}_i+\widetilde{\eta}_i\otimes 1+\widetilde{\eta}_i\otimes \widetilde{\eta}_i, \\
\pi(\widetilde{\eta}_i) &= 0,\\
\sigma(\widetilde{\eta}_i) &= (1+\widetilde{\eta}_i)^{-1}-1=-\widetilde{\eta}_i+\widetilde{\eta}_i^2-\hdots
\end{alignat*}
\end{proposition}

\begin{proof}
This is a result of \ref{era2prop611} and \ref{era2prop612}. Note that the series $-\widetilde{\eta}_i+\widetilde{\eta}_i^2-\hdots$ is convergent since $\widetilde{\eta}_i=p^n\frac{\widetilde{\eta}_i}{p^n}\in p\calR_{\frakX,n}.$
\end{proof}

\begin{proposition}\label{HopfR}
Let $n$ be a positive integer, $\calR_n=\calR_{\frakX,n}/p\calR_{\frakX,n}$ and
$
\delta:\calR_n \ra \calR_n\otimes_{\Ox_X}\calR_n,$ 
$\pi:\calR_n \ra \Ox_X$ and 
$\sigma:\calR_n \ra \calR_n
$
the morphisms defining the Hopf algebra structure on $\calR_n.$ Under the hypothesis \ref{loccoord},
we have
\begin{alignat*}{2}
\delta\left (\eta_{i(r),n} \right ) &= 1\otimes \eta_{i(r),n}+\eta_{i(r),n}\otimes 1, \\
\pi\left (\eta_{i(r),n} \right ) &= 0,\\
\sigma\left (\eta_{i(r),n} \right ) &= -\eta_{i(r),n}.
\end{alignat*}
\end{proposition}

\begin{proof}
This results from \ref{HopffrakR}, the fact that $\widetilde{\eta}_i=p^n\frac{\widetilde{\eta}_i}{p^n}\in p\calR_{\frakX,n}$ and the flatness of $R_{\frakX,n}$ over $\op{Spf}W$ \eqref{Qlogflat}.
\end{proof}

\begin{proposition}\label{HopffrakP}
Let
$
\delta:\calP_{\frakX / \frakS,n} \ra \calP_{\frakX / \frakS,n}\otimes_{\Ox_{\frakX}}\calP_{\frakX / \frakS,n},$ 
$\pi:\calP_{\frakX / \frakS,n} \ra \Ox_{\frakX}$ and 
$\sigma:\calP_{\frakX / \frakS,n} \ra \calP_{\frakX / \frakS,n}$
be the morphisms defining the Hopf algebra structure on $\calP_{\frakX / \frakS,n}.$ Under the hypothesis \ref{loccoord},
we have
\begin{alignat*}{2}
\delta(\widetilde{\eta}_i) &= 1\otimes \widetilde{\eta}_i+\widetilde{\eta}_i\otimes 1+\widetilde{\eta}_i\otimes \widetilde{\eta}_i, \\
\pi(\widetilde{\eta}_i) &= 0,\\
\sigma(\widetilde{\eta}_i) &= (1+\widetilde{\eta}_i)^{-1}-1=-\widetilde{\eta}_i+\widetilde{\eta}_i^2-\hdots
\end{alignat*}
\end{proposition}

\begin{proof}
This is a result of \ref{era2prop611} and \ref{era2prop612}. Note that the series $-\widetilde{\eta}_i+\widetilde{\eta}_i^2-\hdots$ is convergent since $\widetilde{\eta}_i=p^n\frac{\widetilde{\eta}_i}{p^n}\in p\calP_{\frakX / \frakS,n}.$
\end{proof}

\begin{definition}\label{defhpdstrat}
Let $k\ge 1$ and $n\ge 0$ be integers and $\calE$ an $\Ox_{\frakX_k}$-module. Let $P_{n,k}$ be the logarithmic scheme obtained from $P_{\frakX/\frakS,n}$ by reduction modulo $p^k.$ 
An \emph{$n$-stratification on $\calE$} (resp. \emph{$n$-HPD-stratification on $\calE$}) is a stratification (resp. HPD stratification) (\cite{DXU19} 5.4) on $\calE$ with respect to the Hopf algebra corresponding to the groupoid $P_{n,k}$ (\ref{Kokoprop126} and \ref{paragomega}).
We denote by $n$-$\op{MHS}(\frakX_k/\frakS_k)$ the category of $\Ox_{\frakX_k}$-modules equipped with $n$-HPD-stratifications.
\end{definition}

\begin{parag}\label{paragfinalMuzan1}
We end this section by proving that stratifications with respect to the groupoid $P_{\frakX/\frakS,n}$ are equivalent to integrable quasi-nilpotent $p^n$-connections.
For integers $n\ge 0$ and $k\ge 1,$ let $\ov{\calI}_{\frakX/\frakS,n}$ be the PD-ideal of $P_{\frakX/\frakS,n}$ and $\ov{\calI}_{n,k}$ its reduction modulo $p^k.$ For all positive integers $k,$ the universal property of PD-envelopes implies the existence of a morphism
$$
\left ( P_{\frakX/\frakS,n} \right )_k \ra \left ( P_{\frakX/\frakS,0} \right )_k
$$
fitting into the commutative diagram
$$
\begin{tikzcd}
\frakX_k \ar[equal]{r} \ar{d} & \frakX_k \ar{d} \ar{dr} & \\
\left ( P_{\frakX/\frakS,n} \right )_k \ar[bend right=30]{rr} \ar{r} & \frakY_k & \left ( P_{\frakX/\frakS,0} \right )_k, \ar{l}
\end{tikzcd}
$$
where $\frakX_k \ra \frakY_k$ is the exact diagonal immersion and the other morphisms are the canonical ones.
By taking the inductive limit, these morphisms induce a morphism of formal groupoids
\begin{equation}\label{finalMuzan1}
P_{\frakX/\frakS,n} \ra P_{\frakX/\frakS,0}.
\end{equation}
Under the assumption \ref{loccoord}, the corresponding morphism of rings is
$$
\calP_{\frakX/\frakS,0} \ra \calP_{\frakX/\frakS,n},\ \eta_i \mapsto \eta_i=p^n\frac{\eta_i}{p^n}.
$$
The image of $\ov{\calI}_{\frakX / \frakS,0}$ by this morphism is $p^n\ov{\calI}_{\frakX / \frakS , n}.$ It induces, by flatness of $P_{\frakX / \frakS,n}$ over $\op{Spf}\Z_p,$ an isomorphism
$
\ov{\calI}_{\frakX / \frakS,0} \xrightarrow{\sim} \ov{\calI}_{\frakX / \frakS , n}.
$
Reducing modulo $p^k$ and using \eqref{eqtakrizIIsquare}
\begin{equation}\label{finalMuzan3}
\ov{\calI}_{0,k}^{ \{ 1 \} } := \ov{\calI}_{0,k}/\ov{\calI}_{0,k}^{[2]} \xrightarrow{\sim} \omega^1_{\frakX_k / \frakS_k },
\end{equation}
we get an isomorphism of $\Ox_{\frakX_k}$-modules
$$\beta: \omega^1_{\frakX_k/\frakS_k} \xrightarrow{\sim} \ov{\calI}_{n,k}^{\{1\}} := \ov{\calI}_{n,k}/\ov{\calI}_{n,k}^{[2]}.$$
Under the assumption \ref{loccoord}, if $\widehat{m}_i$ and $\widehat{\eta}_i$ are the reductions of $\widetilde{m}_i$ and $\frac{\widetilde{\eta}_i}{p^n}$ modulo $p^k$ respectively, then
\begin{equation}\label{KhamineiBeta}
\beta(\op{dlog}\widehat{m}_i)=\widehat{\eta}_i.
\end{equation}
Let $p_{1,n},p_{2,n}:P_{\frakX / \frakS , n} \ra \frakX$ be the canonical projections. We abusively denote by $p_{1,n}$ and $p_{2,n}$ their reductions modulo $p^k.$ For a local section $a$ of $\Ox_{\frakX_k},$ the differential $da$ corresponds, by \eqref{finalMuzan3}, to $p_{2,0}^{\#}a-p_{1,0}^{\#}a.$ Consider the derivation
$$
d':\begin{array}[t]{clc}
\Ox_{\frakX_k} & \ra & \ov{\calI}_{n,k}^{\{1\}} \\
a & \mapsto & p_{2,n}^{\#}(a)-p_{1,n}^{\#}(a).
\end{array}
$$
By construction of \eqref{finalMuzan1}, the diagram
$$
\begin{tikzcd}
P_{\frakX / \frakS , n} \ar{dr}{p_{i,n}} \ar{d} & \\
P_{\frakX / \frakS , 0} \ar[swap]{r}{p_{i,0}} & \frakX
\end{tikzcd}
$$
is commutative. It follows that, for any local section $a$ of $\Ox_{\frakX_k},$ we have
\begin{equation}\label{finalMuzan4}
p^n\beta(da) = p^n\beta \left ( p_{2,0}^{\#} a - p_{1,0}^{\#}a \right ) = p_{2,n}^{\#} a - p_{1,n}^{\#} a =d'(a).
\end{equation}
\end{parag}

\begin{proposition}\label{propfinalMuzan1}
Keep the notation of \ref{paragfinalMuzan1}. The data of a $p^n$-connection on an $\Ox_{\frakX_k}$-module $\calE$
$$
\nabla:\calE \ra \calE \otimes_{\Ox_{\frakX_k}}\omega^1_{\frakX_k/\frakS_k}
$$
is equivalent to the data of an additive morphism
$$\nabla':\calE \ra \calE \otimes_{\Ox_{\frakX_k}}\ov{\calI}_{n,k}^{\{1\}}$$
satisfying the Leibniz rule
$$\nabla'(ax)=a\nabla'(x)+x\otimes d'(a),$$
for all local sections $a$ and $x$ of $\Ox_{\frakX_k}$ and $\calE$ respectively. This equivalence is given by
$$
\nabla \mapsto \nabla'=\left ( \op{Id}_{\calE} \otimes \beta \right ) \circ \nabla.
$$
\end{proposition}

\begin{proof}
The Leibniz rule satisfied by $\nabla'$ follows from \eqref{finalMuzan4}.
\end{proof}

\begin{definition}
Let $\nabla$ be a $p^n$-connection on an $\Ox_{\frakX_k}$-module $\calE$ and $\nabla':\calE\ra \calE\otimes_{\Ox_{\frakX_k}}\ov{\calI}_{n,k}^{\{1\}}$ the corresponding additive morphism by \ref{propfinalMuzan1}. We say that $\nabla'$ is \emph{integrable} if $\nabla$ is integrable.
\end{definition}

\begin{parag}\label{compdiffopn}
Let $k\ge 1$ and $n\ge 0$ be integers and $\calE$ an $\Ox_{\frakX_k}$-module. Let $P_{n,k}$ be the logarithmic scheme obtained from $P_{\frakX/\frakS,n}$ by reduction modulo $p^k,$ $\ov{\calI}_{n,k}$ the PD-ideal of $P_{n,k},$ $\calP_{n,k}$ the structural ring of $P_{n,k}$ and $\calP_{n,k}^{\{l\}}=\calP_{n,k}/\ov{\calI}_{n,k}^{[l+1]}$ for any integer $l\ge 0.$ We recall how the composition of differential operators
$$
f:\calP_{n,k}^{\{l'\}} \otimes_{\Ox_{\frakX_k}} \calE \ra \calE,\ g:\calP_{n,k}^{\{l\}} \otimes_{\Ox_{\frakX_k}} \calE \ra \calE
$$
is defined. Let 
$$
\delta^{l,l'}:\calP^{\{l+l'\}}_{n,k} \ra \calP^{\{l\}}_{n,k} \otimes_{\Ox_{\frakX_k}} \calP^{\{l'\}}_{n,k}
$$
be the morphism induced by the comultiplication map of the Hopf algebra $\calP_{n,k}.$
The composition $g \circ f$ is defined as the composition
\begin{equation}\label{compjapon}
g\circ f:\calP^{\{l+l'\}}_{n,k}\otimes \calE \xrightarrow{\delta^{l,l'}} \calP^{\{l\}}_{n,k} \otimes \calP^{\{l'\}}_{n,k} \otimes \calE \xrightarrow{\op{Id} \otimes f} \calP^{\{l\}}_{n,k} \otimes \calE \xrightarrow{g} \calE.
\end{equation}
\end{parag}

\begin{lemma}\label{lemjapon}
Let $k\ge 1$ and $n\ge 0$ be integers and suppose \ref{loccoord} is satisfied. For any $1\le i\le d,$ let $\xi_i$ be the image of $\frac{\widetilde{\eta}_i}{p^n} \in \calP_{\frakX / \frakS , n}$ in $\calP_{n,k}:=\calP_{\frakX/\frakS,n}/(p^k)$ and set, for any $I=(I_1,\hdots,I_d)\in \N^d,$
$$\xi^{[I]}=\prod_{i=1}^d\xi_i^{[I_i]}.$$
Let $\delta:\calP_{n,k} \ra \calP_{n,k}\otimes_{\Ox_{\frakX_k}}\calP_{n,k}$ be the comultiplication map of the Hopf algebra $\calP_{n,k}.$ . Then $\left (\xi^{[I]}\right )_{I\in \N^d}$ is a basis for the $\Ox_{\frakX_k}$-module $\calP_{n,k}.$ Let $(\partial_I)_{I\in \N^d}$ be its dual basis. Then, for any $k\in \N,$ $1\le i\le d$ and $I=(I_1,\hdots,I_d)\in \N^d,$ we have
\begin{enumerate}
\item $\partial_I\circ \partial_{\epsilon_i}=\partial_{I+\epsilon_i}+p^nI_i\partial_I,$
\item $\partial_{k\epsilon_i}=\prod_{j=0}^{k-1}(\partial_{\epsilon_i}-p^nj),$
\item $\partial_I=\prod_{i=1}^d\prod_{j=0}^{I_i-1}(\partial_{\epsilon_i}-p^nj),$
\end{enumerate}
where the composition is the composition of differential operators \eqref{compjapon}.
\end{lemma}

\begin{proof}
Using \ref{lem116f}, the proof of the first formula is similar to \ref{prop17}. The other two follow from the first.
\end{proof}

\begin{parag}\label{paragJapon2}
Let $n\ge 0$ and $k\ge 1$ be integers and consider an integrable $p^n$-connection $\calE \ra \calE \otimes_{\Ox_{\frakX_k}}\omega^1_{\frakX_k/\frakS_k}$ on an $\Ox_{\frakX_k}$-module $\calE.$ Let
$$\nabla:\calE \ra \calE \otimes_{\Ox_{\frakX_k}}\ov{\calI}_{n,k}^{\{1\}}$$
be the corresponding connection given by \ref{propfinalMuzan1}. Suppose that the hypothesis \ref{loccoord} is satisfied and consider the notation introduced in \ref{lem116f}. We have a canonical isomorphism
$$\calP_{n,k}\xrightarrow{\sim} \Ox_{\frakX_k}\langle \xi_1,\hdots,\xi_d \rangle.$$
Denote by $\left ( \partial_I \right )_{I\in \N^d}$ the dual basis of $\left ( \xi^I \right )_{I\in \N^d}.$ Let $\epsilon_1$ be the $\calP_{n,k}$-linear morphism defined by
$$
\epsilon_1:\begin{array}[t]{clc}
\calP_{n,k}^{\{1\}} \otimes_{\Ox_{\frakX_k}}\calE & \ra & \calE\otimes_{\Ox_{\frakX_k}}\calP_{n,k}^{\{1\}} \\
1\otimes x & \mapsto & \nabla(x)+x\otimes 1,
\end{array}
$$
and, for $i\le 1\le d,$ the differential operator
$$
\nabla \left ( \partial_{\epsilon_i} \right ): \calP_{n,k}^{ \{ 1 \} } \otimes_{\Ox_{\frakX_k}} \calE \xrightarrow{\epsilon_1} \calE \otimes_{\Ox_{\frakX_k}} \calP_{n,k}^{ \{ 1 \} } \xrightarrow{\op{Id}_{\calE} \otimes \partial_{\epsilon_i}} \calE.
$$
The integrability of $\nabla$ implies that the differential operators $\nabla (\partial_{\epsilon_i})$ pairwise commute (here, we consider composition as differential operators \eqref{compdiffopn}) and we can define $\nabla (\partial_N),$ for any multi-index $N=(n_1,\hdots,n_d)\in\mathbb{N}^d,$ by
$$
\nabla (\partial_N)=\prod_{i=1}^d\prod_{j=0}^{n_i-1} \left ( \nabla (\partial_{\epsilon_i})-p^nj \right ),
$$
where $p^nj$ denotes $p^nj \op{Id}_{\calE},$ considered as a differential operator.
\end{parag}

\begin{definition}\label{pnqn}
Keep the notation of the proof of \ref{paragJapon2}. We say that an integrable $p^n$-connection $\nabla$ on a $\Ox_{\frakX_k}$-module $\calE$ is \emph{quasi-nilpotent} if, for every local section $x$ of $\calE,$ there exists, locally on $\frakX_k,$ an integer $m$ such that
$$
\nabla\left (\partial_I \right )(1 \otimes x)=0
$$
for all $I\in \N^d$ such that $|I|\ge m.$
\end{definition}

\begin{proposition}\label{equivstrat}
Let $k\ge 1$ and $n\ge 0$ be integers and $\calE$ an $\Ox_{\frakX_k}$-module. Let $P_{n,k}$ be the logarithmic scheme obtained from $P_{\frakX/\frakS,n}$ by reduction modulo $p^k,$ $\ov{\calI}_{n,k}$ the PD-ideal of $P_{n,k},$ $\calP_{n,k}$ the structural ring of $P_{n,k}$ and $\calP_{n,k}^{\{l\}}=\calP_{n,k}/\ov{\calI}_{n,k}^{[l+1]}$ for any integer $l\ge 0.$ The following data are equivalent:
\begin{enumerate}
\item A $\calP_{n,k}$-stratification $(\varepsilon_l)_{l\ge 0}$ on $\calE.$
\item A sequence of $\Ox_{\frakX_k}$-linear morphisms $\theta_l:\calE\ra \calE\otimes_{\Ox_{\frakX_k}}\calP_{n,k}^{\{l\}}$ satisfying the following conditions:
\begin{itemize}
\item $\theta_0=\op{Id}_{\calE}.$
\item For any integer $l\ge 0,$ the morphism $\theta_l$ is equal to the composition
$$\calE\xrightarrow{\theta_{l+1}}\calE\otimes_{\Ox_{\frakX_k}}\calP^{\{l+1\}}_{n,k}\ra \calE\otimes_{\Ox_{\frakX_k}}\calP^{\{l\}}_{n,k}$$
where the second arrow is the canonical projection.
\item For all integers $l,l'\ge 0$
\begin{equation}\label{diagstrconn1}
\begin{tikzcd}
\calE\ar{r}{\theta_{l+l'}}\ar{d}{\theta_{l'}} & \calE\otimes_{\Ox_{\frakX_k}}\calP^{\{l+l'\}}_{n,k}\ar{d}{\op{Id}\otimes \delta^{l,l'}}\\
\calE\otimes_{\Ox_{\frakX_k}}\calP^{\{l'\}}_{n,k}\ar{r}{\theta_l\otimes \op{Id}}& \calE\otimes_{\Ox_{\frakX_k}}\calP^{\{l\}}_{n,k}\otimes_{\Ox_{\frakX_k}}\calP^{\{l'\}}_{n,k},
\end{tikzcd}
\end{equation}
where $\delta^{l,l'}:\calP^{\{l+l'\}}_{n,k} \ra \calP^{\{l\}}_{n,k} \otimes_{\Ox_{\frakX_k}} \calP^{\{l'\}}_{n,k}$ is induced by the comultiplication map of the Hopf algebra $\calP_{n,k}.$
\end{itemize}
\item An integrable $p^n$-connection $\nabla:\calE\ra \calE\otimes_{\Ox_{\frakX_k}}\omega^1_{\frakX_k/\frakS_k}.$
\end{enumerate}
There exists also an equivalence between the data of a stratification
$$
\calP_{n,k} \otimes_{\Ox_{\frakX_k}} \calE \ra \calE \otimes_{\Ox_{\frakX_k}} \calP_{n,k}
$$
and an integrable quasi-nilpotent connection
$$
\nabla:\calE \ra \calE\otimes_{\Ox_{\frakX_k}} \omega^1_{\frakX_k/\frakS_k}.
$$
The connection is obtained from the stratification as follows:
$$
\nabla : \begin{array}[t]{clclclc}
\calE & \ra & \calE \otimes \ov{\calI}_{n,k} & \ra & \calE\otimes \ov{\calI}_{n,k}^{\{1\}} & \xrightarrow{\sim} & \omega^1_{\frakX_k/\frakS_k}\\
x & \mapsto & \epsilon(1\otimes x)-x\otimes 1 & & & &
\end{array},
$$
where the second arrow is the canonical projection.
\end{proposition}

\begin{proof}
See appendix.
\end{proof}

\section{A stratified interpretation of the logarithmic Shiho functor}

\begin{parag}\label{parag91}

In this section, we consider a perfect field of characteristic $p>0.$ We denote by $W$ its ring of Witt vectors and we equip $\op{Spf}W$ with the trivial logarithmic structure. We consider a morphism $\theta:P\ra Q$ of fs monoids and a log smooth morphism of framed logarithmic $p$-adic formal schemes $(f,\theta):(\frakX,Q)\ra (\frakS,P),$ such that $\frakS$ is log flat locally of finite type over $\op{Spf}W$ (\cite{Ahmed2010} 2.3.13). Note that $f$ is log flat since it is log smooth (\cite{Ogus2018} IV 4.1.2) so, by \ref{Wflat}, the formal schemes $\frakX$ and $\frakS$ are flat over $\op{Spf}W$ \eqref{dxuflat}. We also consider the formal groupoids $R_{\frakX,n}$ and $P_{\frakX/\frakS,n}$ defined in \ref{parag86}. Set $X=\frakX_1$ and $S=\frakS_1.$

Denote by $F_1:X\ra X'$ the exact relative Frobenius of $X$ over $S.$ Let $Q',$ $F_P:P\ra P$ and $F_{Q/P}:Q' \ra Q$ be as defined in \ref{thm410}.
By \ref{thm410}, the logarithmic scheme $X'$ is naturally equipped with a frame $X'\ra [Q']$ such that $(F_1,F_{Q/P})$ is a morphism of framed logarithmic schemes. We suppose that $(X',Q')$ (resp. $F_1:X\ra X'$) lifts to a framed logarithmic $p$-adic formal scheme $(\frakX',Q')$ (resp. to an $(\frakS,P)$-morphism $F:(\frakX,Q)\ra (\frakX',Q')$) such that $\frakX'$ is log smooth and locally of finite type over $\frakS.$
Set
$\frakY=\frakX\times_{\frakS,[Q]}^{\op{log}}\frakX,\ \frakY'=\frakX'\times_{\frakS,[Q']}^{\op{log}}\frakX'$ (see \ref{prop74}).
Denote, for every integers $n\ge 0$ and $k\ge 1,$ by $\frakX_k,$ $\frakS_k$ and $P_{\frakX/\frakS,n,k}$ the logarithmic schemes obtained from $\frakX,$ $\frakS$ and $P_{\frakX/\frakS,n}$  by reduction modulo $p^k,$ as in \ref{prop69}. Set $Y=\frakY_1,$ $P_n=P_{\frakX/\frakS,n,1}$ and $P'_n=P_{\frakX'/\frakS,n,1}.$ Denote by $\calP_n$ and $\calP'_n$ the structural rings of $P_n$ and $P'_n$ respectively.
\end{parag}

\begin{lemma}\label{lem92}
Denote by $p_1,p_2:\frakY\ra \frakX,$ $p_1',p_2':\frakY'\ra \frakX'$ the canonical projections and by $\Delta:\frakX \ra \frakY$ and $\Delta':\frakX' \ra \frakY'$ the strict diagonal immersions. Let $G:\frakY\ra \frakY'$ be the morphism induced by $F:(\frakX,Q)\ra (\frakX',Q').$ Then $G$ induces a morphism of formal groupoids:
\begin{equation}\label{phi}
\varphi:P_{\frakX/\frakS,0}\ra P_{\frakX'/\frakS,1},
\end{equation}
fitting into the commutative diagram
$$
\begin{tikzcd}
P_{\frakX/\frakS,0} \ar{r}{\varphi} \ar{d} & P_{\frakX'/\frakS,1} \ar{d} \\
\frakY \ar{r}{G} & \frakY'
\end{tikzcd}
$$
where the vertical arrows are the canonical ones.
In addition, suppose that \ref{loccoord} is satisfied and let $\varphi^{\#}$ be the morphism of structural rings of $\varphi.$ By \ref{lem12}, there exists a local invertible section $u_i$ of $\calM_{\frakX}$ and a local section $b_i$ of $\Ox_{\frakX}$ such that $F^{\flat}(m_i')=pm_i+u_i$ and $\alpha_{\frakX}(u_i)=1+pb_i.$ Then, for $1\le i\le d,$
\begin{equation}\label{eqKoko1323}
\varphi^{\#} \left (\frac{\widetilde{\eta}'_i}{p} \right ) = \left ((p-1)!\widetilde{\eta}_i^{[p]} + \sum_{k=1}^{p-1}\frac{(p-1)!}{k!(p-k)!}\widetilde{\eta}_i^k \right ) \alpha_{\frakY}\left (p_2^{\flat}u_i-p_1^{\flat}u_i \right )+ \frac{p_2^{\#}b_i-p_1^{\#}b_i}{\alpha_{\frakY}(p_1^{\flat}u_i)}.
\end{equation}
\end{lemma}

\begin{proof}
We have the following commutative diagram
$$
\begin{tikzcd}
P_{\frakX/\frakS,0} \ar{dd} \ar[dashed]{dr} & P_{\frakX'/\frakS,1} \ar{d} \\
 & R_{\frakX',1} \ar{d} \\
\frakY\ar{r}{G}\ar{d}\ar[swap, bend right=30]{dd}{p_i} & \frakY'\ar{d}\ar[bend right=-30]{dd}{p'_i} \\
\frakX^2_{\frakS}\ar{r}{F^2}\ar{d} & \frakX'^2_{\frakS}\ar{d} \\
\frakX\ar{r}{F} & \frakX'
\end{tikzcd}
$$
Let $\frakX \ra \frakU \ra \frakY$ and $\frakX' \ra \frakU' \ra \frakY'$ be factorizations of $\Delta$ and $\Delta'$ respectively into a closed immersion followed by an open one such that $G(\frakU) \subset \frakU'.$
Since $P_{\frakX/\frakS,0}$ is flat over $\op{Spf}\Z_p$ by \ref{propflat}, to prove the existence of the dashed arrow, it is sufficient to prove that the image of the ideal of $X' \ra \frakU'$ in $\calP_{\frakX/\frakS,0}$ is generated by $p.$ For that, we can work étale locally and hence suppose that \ref{loccoord} is satisfied.
Since the local sections $\widetilde{\eta}'_1,\hdots,\widetilde{\eta}'_d$ locally generate the ideal of $\frakX' \ra \frakU'$ \eqref{parag77}, it is sufficient to prove that the image of $\widetilde{\eta}'_i$ in $P_{\frakX/\frakS,0}$ belongs to $p\Ox_{P_{\frakX/\frakS,0}}.$ Denote by $G^{\#}$ the morphism of structure sheaves associated to $G.$  By \ref{lemcalc}, we have
\begin{equation}\label{eq921}
G^{\#}\left (\widetilde{\eta}'_i\right ) = \alpha_{\frakY}(p_2^{\flat}F^{\flat}\widetilde{m}'_i-p_1^{\flat}F^{\flat}\widetilde{m}'_i)-1 
= \left (\widetilde{\eta}_i^p+\sum_{k=1}^{p-1}\begin{pmatrix}p\\k \end{pmatrix}\widetilde{\eta}_i^k+1 \right )\alpha_{\frakY}(p_2^{\flat}u_i-p_1^{\flat}u_i)-1.
\end{equation}
In $P_{\frakX/\frakS,0},$ this is equal to
$$\left (p!\widetilde{\eta}_i^{[p]}+\sum_{k=1}^{p-1}\begin{pmatrix}p\\k \end{pmatrix}\widetilde{\eta}_i^k\right )\alpha_{\frakY}(p_2^{\flat}u_i-p_1^{\flat}u_i)+\alpha_{\frakY}(p_2^{\flat}u_i-p_1^{\flat}u_i)-1.$$
All we need is to prove that
$\alpha_{\frakY}(p_2^{\flat}u_i-p_1^{\flat}u_i)-1$
is a section of $p\Ox_{\frakY},$ which is indeed the case:
\begin{alignat*}{2}
\alpha_{\frakY}(p_2^{\flat}u_i-p_1^{\flat}u_i)-1 &= \frac{\alpha_{\frakY}(p_2^{\flat}u_i)-\alpha_{\frakY}(p_1^{\flat}u_i)}{\alpha_{\frakY}(p_1^{\flat}u_i)} 
= \frac{p_2^{\#}(\alpha_{\frakX}(u_i))-p_1^{\#}(\alpha_{\frakX}(u_i))}{\alpha_{\frakY}(p_1^{\flat}u_i)} \\
&= p \frac{p_2^{\#}b_i-p_1^{\#}b_i}{\alpha_{\frakY}(p_1^{\flat}u_i)}.
\end{alignat*}
We deduce that there exists a unique morphism $P_{\frakX/\frakS,0}\ra R_{\frakX',1}$ making the following diagram commutative
$$
\begin{tikzcd}
P_{\frakX/\frakS,0}\ar{r} \ar{d} & R_{\frakX',1}\ar{d} \\
\frakY\ar{r} & \frakY'
\end{tikzcd}
$$
Then by the universal property of the PD-envelope, there exists a unique morphism $\varphi:P_{\frakX / \frakS,0}\ra P_{\frakX' / \frakS,1}$ making the following diagram commutative
$$\begin{tikzcd}
\frakX\ar{r}\ar{d} & \frakX'\ar{d}\ar{dr} & \\
P_{\frakX / \frakS,0}\ar{r}\ar[bend right=30,swap]{rr}{\varphi} & R_{\frakX',1}& P_{\frakX'/\frakS,1}\ar{l}.
\end{tikzcd}$$
\end{proof}

\begin{remark}\label{rem93}
Using the canonical isomorphisms (\ref{rem87})
$$P_{\frakX / \frakS,0}\times_{\frakX}P_{\frakX / \frakS,0}\xrightarrow{\sim}P_{\frakX / \frakS,0}(2),\ P_{\frakX' / \frakS,1}\times_{\frakX'}P_{\frakX' / \frakS,1}\xrightarrow{\sim}P_{\frakX' / \frakS,1}(2)$$
and the morphism $\varphi:P_{\frakX / \frakS,0} \ra P_{\frakX' / \frakS,1}$ defined in \ref{lem92}, we deduce the existence of a morphism
\begin{equation}\label{eq931}
\varphi(2):P_{\frakX/\frakS,0}(2)\ra P_{\frakX'/\frakS,1}(2).
\end{equation}
\end{remark}

\begin{proposition}\label{Khaminei145}
The morphism $\varphi : P_{\frakX/\frakS,0} \ra P_{\frakX'/\frakS,1}$ \eqref{phi} is log flat.
\end{proposition}

\begin{proof}
Consider the canonical projections $p_1:P_{\frakX/\frakS,0} \ra \frakX$ and $p_1':P_{\frakX'/\frakX,1} \ra \frakX'.$ The morphism $\varphi$ factors as
$$
\varphi:P_{\frakX/\frakS,0} \xrightarrow{\widetilde{\varphi}} P_{\frakX'/\frakS,1} \times_{\frakX'}\frakX \ra P_{\frakX'/\frakS,1}.
$$
The second morphism is log flat since so is $F:\frakX\ra \frakX',$ it is hence sufficient to prove that the first morphism is log flat. Since it is strict, it is sufficient to prove that it is flat. We may suppose that the hypothesis \ref{loccoord} is satisfied. The morphism of structural rings of $\widetilde{\varphi}$ identifies then with a morphism
$$
\Ox_{\frakX} \left \langle \left \langle \frac{\widetilde{\eta}_1'}{p},\hdots ,\frac{\widetilde{\eta}_d'}{p} \right \rangle \right \rangle \ra \Ox_{\frakX} \left \langle \left \langle \widetilde{\eta}_1,\hdots,\widetilde{\eta}_d\right \rangle \right \rangle.
$$
By \ref{lemcalc}, keeping the same notation of \ref{lemcalc}, this morphism sends $\frac{\widetilde{\eta}_i'}{p}$ to a section of the form
\begin{equation}\label{Khaminei1451}
\left ( (p-1)!\widetilde{\eta}_i^{[p]}+\sum_{k=1}^{p-1} \frac{(p-1)!}{k!(p-k)!}\widetilde{\eta}_i^k \right ) a+xy,
\end{equation}
where $y=p_2^{\#}b-p_1^{\#}b$ is in the ideal $I$ generated by $\widetilde{\eta}_1,\hdots,\widetilde{\eta}_d,$ $x=(1+pp_1^{\#}b)^{-1}$ and $a-1=pxy.$
Since $(\widetilde{\eta}_1,\hdots,\widetilde{\eta}_d)$ is a basis of the $\Ox_{\frakX}$-module $I/I^2$ and by \eqref{grsym}, we prove by induction that any $b\in I$ is written in the form
$$
b=\sum_{I\in \llbracket 0,p-1 \rrbracket^d,\ I\neq 0}\alpha_I\widetilde{\eta}^I+c,
$$
where $\alpha_I\in \Ox_{\frakX}$ and $c\in I^{(p-1)d+1}.$
Since, for $I\in \N^d\setminus \llbracket 0,p-1 \rrbracket^d,$ $\widetilde{\eta}^I\in p\calP_{\frakX/\frakS,0},$ the section \eqref{Khaminei1451} is equal to
$$
-\widetilde{\eta}_i^{[p]} + \sum_{\substack{k\in \llbracket 0,p-1\rrbracket ^d\\ k\neq 0}} u_{i,k}\widetilde{\eta}^k + pv,
$$
for local sections $u_{i,k}$ of $\Ox_{\frakX}$ and $v$ of $\Ox_{\frakX} \left \langle \left \langle \widetilde{\eta}_1,\hdots,\widetilde{\eta}_d\right \rangle \right \rangle.$
The rest of the proof is the same as in (\cite{Shiho} 3.3).
\end{proof}

\begin{lemma}\label{closedimm}
The diagonal immersion $\frakX \ra \frakX \times_{\frakX'}^{\op{log}}\frakX$ and the exact diagonal immersion $\frakX \ra \frakX \times_{\frakX',[Q]}^{\op{log}}\frakX$ are both closed.
\end{lemma}

\begin{proof}
The exact relative Frobenius $F_1:X\ra X'$ is affine hence so is $F:\frakX \ra \frakX'.$ The diagonal immersion $\frakX \ra \frakX \times_{\frakX'} \frakX$ is thus closed. We conclude by the fact that the canonical morphisms $\frakX \times_{\frakX'}^{\op{log}}\frakX \ra \frakX \times_{\frakX'} \frakX$ and $\frakX \times_{\frakX',[Q]}^{\op{log}}\frakX \ra \frakX \times_{\frakX'}^{\op{log}} \frakX$ are affine and (\cite{SP} \href{https://stacks.math.columbia.edu/tag/07RK}{07RK}).
\end{proof}

\begin{lemma}\label{lem96}
The ideal $\calI_{\frakX/\frakX'}$ of the diagonal immersion $\Delta_{\frakX/\frakX'}:\frakX\ra \frakX\times^{\op{log}}_{\frakX',[Q]}\frakX,$ which is closed by \ref{closedimm}, has a unique PD structure and the canonical morphism $\frakX\times_{\frakX',[Q]}^{\op{log}}\frakX \ra \frakX\times_{\frakS,[Q]}^{\op{log}}\frakX=\frakY$ induces a morphism $\psi:\frakX\times_{\frakX',[Q]}^{\op{log}}\frakX\ra P_{\frakX / \frakS,0}.$
\end{lemma}

\begin{proof}
Denote by $q_1,q_2:\frakX\times_{\frakX',[Q]}^{\op{log}}\frakX \ra \frakX$ the canonical projections.
For any local section $m$ of $\calM_{\frakX},$ consider the local section $\mu(m)$ of $\Delta_{\frakX/\frakX'}^{-1} \left (1+\calI_{\frakX/\frakX'} \right )$ defined as in \ref{parag77} and $\eta(m)=\mu(m)-1.$ Recall that we have an exact sequence \eqref{era2exactseq}
$$
0 \ra \Delta_{\frakX/\frakX'}^{-1} \left (1+\calI_{\frakX/\frakX'} \right ) \xrightarrow{\lambda}  \Delta_{\frakX/\frakX'}^{-1}\calM_{\frakX \times_{\frakX'}^{\op{log}}\frakX} \ra \calM_{\frakX}  \ra 0,
$$
and that
$
\Delta_{\frakX/\frakX'}^{-1}q_1^{\flat}m +\lambda \left ( \mu(m) \right ) = \Delta_{\frakX/\frakX'}^{-1}q_2^{\flat}m.
$
It follows that
$$
\left (\Delta_{\frakX/\frakX'}^{-1}\alpha_{\frakX\times_{\frakX',[Q]}^{\op{log}}\frakX} \right ) \left (\Delta_{\frakX/\frakX'}^{-1}q_1^{\flat}m \right )+ \mu(m) = \left (\Delta_{\frakX/\frakX'}^{-1}\alpha_{\frakX\times_{\frakX',[Q]}^{\op{log}}\frakX} \right ) \left ( \Delta_{\frakX/\frakX'}^{-1}q_2^{\flat}m \right ).
$$
From here on out, we drop $\Delta_{\frakX/\frakX'}^{-1}$ to lighten the notation. The ideal $\calI_{\frakX/\frakX'}$ is locally generated by sections of the form $\eta(m).$ 
By \ref{thmlogflat} and \ref{logflatfiber}, the lifting $F:\frakX\ra \frakX'$ is log flat. Then the base change $\frakX \times_{\frakX'}^{\op{log}}\frakX \ra \frakX$ is also log flat. Composing with the log étale morphism $\frakX\times_{\frakX',[Q]}^{\op{log}}\frakX \ra \frakX\times_{\frakX'}^{\op{log}}\frakX,$ we obtain the log flat projection $\frakX\times_{\frakX',[Q]}^{\op{log}}\frakX \ra \frakX.$ This projection is also strict and thus flat on the underlying formal schemes. Since $\frakX$ is flat over $\op{Spf}\Z_p,$ the formal scheme $\frakX \times_{\frakX',[Q]}^{\op{log}}\frakX$ is also flat over $\op{Spf}\Z_p.$ It is thus sufficient to prove that for any
local section $m$ of $\calM_{\frakX}$ and any positive integer $k,$
$\eta(m)^k\in k!\Ox_{\frakX\times_{\frakX',[Q]}^{\op{log}}\frakX}.$
We start by proving that for any local section $x$ of $\Ox_{\frakX}$ and any positive integer $k,$
$$(q_2^{\#}x-q_1^{\#}x)^k \in k!\Ox_{\frakX\times_{\frakX',[Q]}^{\op{log}}\frakX}.$$
We proceed by induction on $k:$
for $1\le k\le p-1,$ the result is obvious since $k$ is invertible in $\Ox_{\frakX\times_{\frakX'}\frakX}.$ We now check the case $k=p:$
let $x_1$ be the image of $x$ in $\Ox_X.$ Consider the morphism
$\pi:X'\ra X$
defined in \eqref{diag51}. Locally, $\pi^{\#}x_1$ admits a lifting $x'$ to $\Ox_{\frakX'}.$ There exists a local section $a$ of $\Ox_{\frakX}$ such that $F^{\#}(x')=x^p-pa.$
\begin{alignat*}{2}
0 &= q_2^{\#} (F^{\#}(x'))-q_1^{\#}(F^{\#}(x')) \\
&= q_2^{\#}(x)^p-q_1^{\#}(x)^p-p(q_2^{\#}(a)-q_1^{\#}(a)) \\
&= (q_2^{\#}(x)-q_1^{\#}(x))^p+\sum_{k=1}^{p-1}\begin{pmatrix}p\\k \end{pmatrix}q_2^{\#}(x)^k(q_2^{\#}(x)-q_1^{\#}(x))^k-p(q_2^{\#}(a)-q_1^{\#}(a)) \\
&= (q_2^{\#}(x)-q_1^{\#}(x))^p-pb(q_2^{\#}(x)-q_1^{\#}(x))-p(q_2^{\#}(a)-q_1^{\#}(a)),
\end{alignat*}
where $b=\sum_{k=1}^{p-1}\frac{(p-1)!}{k!(p-k)!}q_2^{\#}(x)^k\left ( q_2^{\#}(x)-q_1^{\#}(x) \right )^{k-1}.$
So
$$(q_2^{\#}(x)-q_1^{\#}(x))^p=pb(q_2^{\#}(x)-q_1^{\#}(x))+p(q_2^{\#}(a)-q_1^{\#}(a)).$$
Now let $k>p$ and suppose the result is true for all integers $i\le k.$ Consider the largest integer $r$ such that $p^r \le k.$
\begin{alignat*}{2}
(q_2^{\#}(x)-q_1^{\#}(x))^{p^r} &= (pb(q_2^{\#}(x)-q_1^{\#}(x))+p(q_2^{\#}(a)-q_1^{\#}(a)))^{p^{r-1}} \\
&= \sum_{l=0}^{p^{r-1}}\begin{pmatrix}p^{r-1}\\ l \end{pmatrix}p^{p^{r-1}}b^l(q_2^{\#}(x)-q_1^{\#}(x))^l(q_2^{\#}(a)-q_1^{\#}(a))^{p^{r-1}-l}.
\end{alignat*}
By the induction hypothesis, the $l^{\text{th}}$ term of this sum is a local section of
$$
p^{p^{r-1}}\begin{pmatrix}p^{r-1}\\ l \end{pmatrix}l!(p^{r-1}-l)!\Ox_{\frakX\times_{\frakX'}^{\op{log}}\frakX}=p^{p^{r-1}}(p^{r-1})!\Ox_{\frakX\times_{\frakX'}^{\op{log}}\frakX}=\left (p^r\right )!\Ox_{\frakX\times_{\frakX'}^{\op{log}}\frakX}.
$$
If $k=p^r,$ we get that $\left (q_2^{\#}(x)-q_1^{\#}(x)\right )^k$ is a local section of $k!\Ox_{\frakX\times_{\frakX'}^{\op{log}}\frakX}.$ If $k>p^r,$ then
$$
\left (q_2^{\#}(x)-q_1^{\#}(x)\right )^k=\left (q_2^{\#}(x)-q_1^{\#}(x)\right )^{p^r}\left (q_2^{\#}(x)-q_1^{\#}(x)\right )^{k-p^r}
$$
is a local section of $\left (p^r\right )! \left (k-p^r \right )!\Ox_{\frakX\times_{\frakX'}^{\op{log}}\frakX}=k!\Ox_{\frakX\times_{\frakX'}^{\op{log}}\frakX}.$

Now let $m,\ m',\ u$ and $b$ as in \ref{lem12} and $\mu(F^{\flat}m'),$ $\mu(m)$ and $\eta(m)$ be as in \ref{parag77}. Since $F \circ q_1 =F \circ q_2$, we have $\mu(F^{\flat}m')=1.$
Then
$$
1=\mu(F^{\flat}m')=\mu(pm+u)=\mu(m)^p\mu(u).
$$
Since $u$ is invertible in $\calM_{\frakX},$ we also have
\begin{alignat*}{2}
\mu(u) &= \alpha_{\frakX \times_{\frakX'}^{\op{log}} \frakX} \left (q_2^{\flat}u-q_1^{\flat}u \right ) = \frac{\alpha_{\frakX \times_{\frakX'}^{\op{log}} \frakX} \left (q_2^{\flat}u \right )}{\alpha_{\frakX \times_{\frakX'}^{\op{log}} \frakX} \left (q_1^{\flat}u \right )} = \frac{q_2^{\#} \alpha_{\frakX} \left (u \right )}{q_1^{\#}\alpha_{\frakX} \left (u \right )} = \frac{1+pq_2^{\#}b}{1+pq_1^{\#}b}.
\end{alignat*}
We deduce that
$
1 = (\eta(m)+1)^p \frac{1+pq_2^{\#}b}{1+pq_1^{\#}b}.
$
Then
\begin{alignat*}{2}
(\eta(m)+1)^p&=(1+pq_1^{\#}b)(1-pq_2^{\#}b+p^2(q_2^{\#}b)^2-\hdots ) = 1-p(q_2^{\#}b-q_1^{\#}b)\sum_{n=1}^{\infty}(-pq_2^{\#}b)^{n-1}\\
&=1+pc(q_2^{\#}b-q_1^{\#}b),
\end{alignat*}
where $c=\sum_{n=1}^{\infty}(-pq_2^{\#}b)^{n-1}.$
It follows that
\begin{alignat*}{2}
\eta(m)^p &= -\sum_{k=1}^{p-1}\begin{pmatrix}p\\k \end{pmatrix}\eta(m)^k+pc(q_2^{\#}b-q_1^{\#}b) = pd\eta(m)+pc(q_2^{\#}b-q_1^{\#}b).
\end{alignat*}
We then prove by induction on $k,$ in the same way we did above, that
$
\eta(m)^k \in k!\Ox_{\frakX\times_{\frakX',[Q]}^{\op{log}}\frakX}.
$
The canonical morphism $\frakX\times_{\frakX',[Q]}^{\op{log}}\frakX\ra \frakY=\frakX\times_{\frakS,[Q]}^{\op{log}}\frakX$
and the fact that the ideal of the diagonal immersion $\frakX\ra \frakX\times_{\frakX'}^{\op{log}}\frakX$ has a unique PD-structure, imply the existence of a PD-morphism $\psi:\frakX\times_{\frakX',[Q]}^{\op{log}}\frakX\ra P_{\frakX / \frakS,0}$ fitting into the following commutative diagram
$$
\begin{tikzcd}
 & & P_{\frakX / \frakS,0} \ar{d} \\
 & \frakX\times^{\op{log}}_{\frakX',[Q]}\frakX \ar{ur}{\psi} \ar{r} \ar{d} & \frakY\ar{d} \\
\frakX \ar{ur} \ar{r} & \frakX\times_{\frakX'}^{\op{log}}\frakX \ar{r} & \frakX \times_{\frakS}^{\op{log}}\frakX.
\end{tikzcd}
$$
\end{proof}

\begin{lemma}\label{lem95}
The diagonal immersion $X\ra X \times_{X'}^{\op{log}} X$ is exact and so $X\times_{X',[Q]}^{\op{log}}X=X\times^{\op{log}}_{X'}X.$
\end{lemma}

\begin{proof}
let $Z$ be the fiber product of $F_1:X \ra X'$ by itself in the category of fine logarithmic schemes.
Since the class of exact morphisms is stable by base change in the category of fine logarithmic schemes (\cite{Ogus2018} III 2.2.1.3), the second projection $p:Z\ra X$ is exact. Etale locally on $X$ and $S,$ the morphism $X\ra S$ admits a chart $\theta:P\ra M$ such that $P^{gp} \ra M^{gp}$ is injective, the torsion subgroup of its cokernel has a finite order coprime with $p$ and the morphism $X \ra S\times_{A[P]}A[M],$ induced by $X\ra S$ and $X\ra A[M],$ is strict and smooth. Let $M'$ and $F_{M/P}:M' \ra M$ be as defined in \ref{PFrob}. The exact relative Frobenius $F_1:X\ra X'$ admits $F_{M/P}:(x,y)\mapsto px+\theta^{gp} (y)$ as a chart.
So, étale locally, the second projection $p_2:Z \ra X$ (resp. the canonical morphism $X\times_{X'}^{\op{log}}X\ra Z$) has a chart given by
\begin{alignat*}{2}
q_1 &: M\ra (M\oplus_{M'}M)^{int},\ x\mapsto (x,0) \quad
( \text{resp.}\ i : (M\oplus_{M'}M)^{int}\ra (M\oplus_{M'}M)^{sat} ).
\end{alignat*}
Let
$\Delta:(M\oplus_{M'}M)^{sat}\ra M,\ (x,y)\mapsto x+y$
be a local chart of the diagonal immersion $X\ra X\times_{X'}^{\op{log}}X.$ It is sufficient to prove that $\Delta$ is strict or, equivalently, exact.
We thus have the following commutative diagram of monoids :
$$\begin{tikzcd}
(M\oplus_{M'}M)^{sat}\ar{r}{\Delta} & M\\
(M\oplus_{M'}M)^{int}\ar{u}{i} & \\
M\ar{u}{p_2}\ar[swap]{uur}{\op{Id}_M} &
\end{tikzcd}$$
corresponding to
$$
\begin{tikzcd}
X\ar{r}\ar{ddr} & X\times_{X'}^{\op{log}}X\ar{d} & \\
 & Z\ar{d}{p_2} \\
 & X
\end{tikzcd}
$$
The composition $\Delta\circ i\circ p_2$ is exact and $(M\oplus_{M'}M)^{sat}$ is saturated, so, by (\cite{Ogus2018} I 4.2.1.2), it is sufficient to prove that $\op{coker}(i\circ p_2)^{gp}$ is a torsion group. The morphism $(i\circ p_2)^{gp}$ is
$M^{gp}\ra M^{gp}\oplus_{M'^{gp}}M^{gp},\ x\mapsto (0,x).$
Let $(x,y)\in M^{gp}\oplus_{M'^{gp}}M^{gp}.$ Then
$p(x,y)=(\varphi(x,0),py)=(0,py-\varphi(x,0))$
and that concludes the proof.
\end{proof}

\begin{corollaire}\label{cor1014}
The diagonal immersion $\frakX \ra \frakX \times_{\frakX'}^{\op{log}} \frakX$ is exact and so $\frakX \times_{\frakX',[Q]}^{\op{log}}\frakX = \frakX \times_{\frakX'}^{\op{log}} \frakX.$
\end{corollaire}

\begin{proof}
The exactness of the diagonal immersion $\frakX \ra \frakX \times_{\frakX'}^{\op{log}} \frakX$ follows from \ref{lem95} and the fact that the canonical morphisms $X \ra \frakX$ and $X\times_{X'}^{\op{log}}X \ra \frakX \times_{\frakX'}^{\op{log}}\frakX$ are strict.
\end{proof}

\begin{lemma}\label{lem98}
The composition $\varphi\circ\psi:\frakX\times_{\frakX',[Q]}^{\op{log}}\frakX\ra P_{\frakX' / \frakS,1},$ of the morphisms $\varphi$ and $\psi$ defined in \ref{lem92} and \ref{lem96}, factors through the canonical immersion $\frakX'\ra P_{\frakX' / \frakS,1}.$
\end{lemma}

\begin{proof}
Let $r_1,r_2:P_{\frakX'/\frakS,1}\ra \frakX'$ and $q_1,q_2:\frakX\times_{\frakX',[Q]}^{\op{log}}\frakX\ra \frakX$ be the canonical projections and $\Delta_{\frakX/\frakX'}:\frakX \ra \frakX \times_{\frakX',[Q]}^{\op{log}} \frakX$ the exact diagonal immersion. We denote by $\calI_{\frakX/\frakX'}$ its ideal \eqref{closedimm}. By \ref{lem96}, $\Delta_{\frakX/\frakX'}$ is a universal homeomorphism.
By \ref{parag86}, the formal scheme $P_{\frakX'/\frakS,1}$ is constructed as the inductive limit of $(P_{\frakX'/\frakS,1,k})_{k\ge 1},$ where $P_{\frakX'/\frakS,1,k}$ is the PD-envelope of $\frakX'_k\ra R_{\frakX',1,k}.$ Since $\Ox_{P_{\frakX'/\frakS,1,k}}$ is killed by $p^k,$ the ideal $\ov{\calI}_{\frakX'/\frakS,1,k}$ of the canonical immersion $\frakX'\ra P_{\frakX'/\frakS,1,k}$ is a nilideal and so $\frakX'_k\ra P_{\frakX'/\frakS,1,k}$ is a universal homeomorphism, hence so is $\iota':\frakX'\ra P_{\frakX'/\frakS,1}.$ To prove the lemma, it is sufficient to prove that the image of the ideal $\ov{\calI}_{\frakX'/\frakS,1}$ of $\iota',$ by $(\varphi\circ \psi)^{\#},$ vanishes. 
We have a commutative diagram with exact rows \eqref{era2exactseq}
$$
\begin{tikzcd}
0 \ar{r} & F^{-1}\iota'^{-1}\left (1+\ov{\calI}_{\frakX'/\frakS,1} \right ) \ar{r}{F^{-1}\lambda'} \ar{d}{\Delta_{\frakX/\frakX'}^{-1}\left (\varphi\circ \psi \right )^{\#}} & F^{-1}\iota'^{-1} \calM_{P_{\frakX'/\frakS,1}} \ar{d}{\Delta_{\frakX/\frakX'}^{-1}\left (\varphi \circ \psi \right )^{\flat}} \ar{r} & F^{-1}\calM_{\frakX'} \ar{d}{F^{\flat}} \ar{r} & 0 \\
0 \ar{r} & \Delta_{\frakX/\frakX'}^{-1} \left (1+\calI_{\frakX/\frakX'} \right ) \ar{r}{\lambda} & \Delta_{\frakX/\frakX'}^{-1}\calM_{\frakX \times_{\frakX'}^{\op{log}}\frakX} \ar{r} & \calM_{\frakX}  \ar{r} & 0
\end{tikzcd}
$$
We identify the étale sites of $\frakX$ and $\frakX'$ via $F,$ those of $\frakX$ and $\frakX\times_{\frakX',[Q]}^{\op{log}}\frakX$ via $\Delta_{\frakX/\frakX'}$ and those of $\frakX$ and $P_{\frakX'/\frakS,1}$ via $\iota'.$ We hence drop $F^{-1},$ $\iota'^{-1}$ and $\Delta_{\frakX/\frakX'}^{-1}$ from our future notations. Let $m'$ be a local section of $\calM_{\frakX'}$ and $\mu(m')$ be as defined in \ref{parag77}. We just have to prove that
$$
\left ( \varphi \circ \psi \right )^{\#} (\mu(m')-1)=0.
$$
For that, we just have to prove that
$
\lambda \left (\left ( \varphi \circ \psi \right )^{\#} (\mu(m')) \right )=0.
$
We have
$
r_1^{\flat}m'+\lambda'(\mu(m'))=r_2^{\flat}m'.
$
Applying $\left ( \varphi \circ \psi \right )^{\flat},$ by the commutativity of
$$
\begin{tikzcd}
\frakX \times_{\frakX'}^{\op{log}}\frakX \ar{r}{\varphi \circ \psi} \ar{d}{q_i} & P_{\frakX'/\frakS,1} \ar{d}{r_i} \\
\frakX \ar{r}{F} & \frakX',
\end{tikzcd}
$$
we get
$
q_1^{\flat}F^{\flat}m'+\left ( \varphi \circ \psi \right )^{\flat}\lambda'(\mu(m'))=q_2^{\flat}F^{\flat}m'.
$
Since $q_1^{\flat}F^{\flat}=q_2^{\flat}F^{\flat},$ we deduce that
$$
\left ( \varphi \circ \psi \right )^{\flat}\lambda'(\mu(m'))=0.
$$
It follows that
$
\lambda \left ( (\varphi \circ \psi )^{\#}(\mu(m')) \right )=\left ( \varphi \circ \psi \right )^{\flat}\lambda'(\mu(m'))=0.
$
\end{proof}

\begin{parag}
Lemmas \ref{lem92}, \ref{lem96} and \ref{lem98} and corollary \ref{cor1014}  imply the existence of the following commutative diagram
\begin{equation}\label{totdiag1}
\begin{tikzcd}
 & & \frakX'\ar{d}{\iota'} \\
 & P_{\frakX/\frakS,0} \ar{r}{\varphi} \ar{d} & P_{\frakX'/\frakS,1}\ar{d} \\
\frakX\times_{\frakX'}^{\op{log}}\frakX \ar{ur}{\psi} \ar[bend right=-30]{uurr} \ar{r} & \frakX\times^{\op{log}}_{\frakS,[Q]}\frakX \ar{r}{F^2} & \frakX'\times_{\frakS,[Q']}^{\op{log}}\frakX'
\end{tikzcd}
\end{equation}
The reduction of the previous diagram modulo $p$ yields the following commutative diagram
\begin{equation}\label{totdiag2}
\begin{tikzcd}
 & & X'\ar{d}{\iota_1'} \\
 & P_0 \ar{r}{\varphi_1} \ar{d} & P'_1\ar{d} \\
X\times_{X'}^{\op{log}}X \ar{ur}{\psi_1} \ar[bend right=-30]{uurr} \ar{r} & X\times^{\op{log}}_{S,[Q]}X \ar{r}{F_1^2} & X'\times_{S,[Q']}^{\op{log}}X'
\end{tikzcd}
\end{equation}

\end{parag}

\begin{parag}
Let $p_1,p_2:P_{\frakX / \frakS,0}\ra \frakX$ and $p_1',p_2':P_{\frakX'/\frakS,1}\ra \frakX'$ be the canonical projections. For $1\le i<j\le 3,$ let $p_{ij}:P_{\frakX/\frakS,0}(2)\ra P_{\frakX/\frakS,0}$ (resp. $p_{ij}':P_{\frakX'/\frakS,1}(2)\ra P_{\frakX'/\frakS,1}$) be the morphism induced by the $(i,j)$-projection $\frakY(2)\ra \frakY$ (resp. $\frakY'(2)\ra \frakY'$). Also let $\iota:\frakX\ra P_{\frakX / \frakS,0}$ and $\iota':\frakX'\ra P_{\frakX' / \frakS,1}$ be the canonical immersions.
Consider the morphism $\varphi:P_{\frakX / \frakS,0}\ra P_{\frakX' / \frakS,1}$ defined in \ref{lem92}.
Let $(\calE',\epsilon')$ be an object of $1\text{-}\op{MHS}(\frakX'/\frakS).$ The $1$-HPD-stratification
$\epsilon':p_2'^*\calE'\ra p_1'^*\calE'$
induces an isomorphism
\begin{equation}\label{eq981}
\varphi^*\epsilon':\varphi^*p_2'^*\calE' \xrightarrow{\sim} \varphi^*p_1'^*\calE'.
\end{equation}
By the commutative diagram
$$
\begin{tikzcd}
P_{\frakX / \frakS,0}\ar{r}{\varphi} \ar{d}{p_i} & P_{\frakX'/ \frakS,1} \ar{d}{p_i'} \\
\frakX\ar{r}{F} & \frakX'
\end{tikzcd}
$$
the isomorphism (\ref{eq981}) is equal to
\begin{equation}\label{eq982}
\varphi^*\epsilon':p_2^*F^*\calE' \xrightarrow{\sim} p_1^*F^*\calE'.
\end{equation}
By the commutative diagram
$$
\begin{tikzcd}
 & \frakX\ar{r}{F} \ar{d}{\iota} \ar[swap]{dl}{\op{Id}_{\frakX}} & \frakX'\ar{d}{\iota'}\ar{dr}{\op{Id}_{\frakX'}} & \\
\frakX & P_{\frakX / \frakS,0}\ar{l}{p_i} \ar{r}{\varphi} & P_{\frakX' / \frakS,1} \ar[swap]{r}{p_i'} & \frakX'
\end{tikzcd}
$$
and the fact that $\iota'^*\epsilon'=\op{Id}_{\calE'},$ we have
$\iota^*\varphi^*\epsilon'=\op{Id}_{F^*\calE}.$
Consider the morphism $\varphi(2):P_{\frakX / \frakS,0}(2) \ra P_{\frakX' / \frakS,1}(2)$ defined in \ref{rem93}.
The $1$-HPD-stratification $\epsilon'$ satisfies the cocycle condition
$p_{13}'^*\epsilon'=p_{23}'^*\epsilon'\circ p_{12}'^*\epsilon',$
which implies
\begin{equation}\label{eq983}
\varphi(2)^*p_{13}'^*\epsilon'=\varphi(2)^*p_{23}'^*\epsilon'\circ \varphi(2)^*p_{12}'^*\epsilon'.
\end{equation}
By the commutativity of the diagram
$$
\begin{tikzcd}
P_{\frakX / \frakS,0}(2) \ar{r}{p_{ij}} \ar{d}{\varphi(2)} & P_{\frakX / \frakS,0}\ar{d}{\varphi} \\
P_{\frakX' / \frakS,1}(2) \ar{r}{p'_{ij}} & P_{\frakX' / \frakS,1}
\end{tikzcd}
$$
the equality (\ref{eq983}) becomes
$p_{13}^*\varphi^*\epsilon'=p_{23}^*\varphi^*\epsilon'\circ p_{12}^*\varphi^*\epsilon'.$
We conclude that $\varphi^*\epsilon'$ is an HPD-stratification on $F^*\calE'$ and so we have, for all integers $n\ge 1,$ functors
\begin{equation}\label{eq994}
\Psi:\begin{array}[t]{clc}
1\text{-}\op{MHS}(\frakX'/\frakS) & \ra & \op{MHS}(\frakX/\frakS), \\
(\calE',\epsilon') & \mapsto & (F^*\calE',\varphi^*\epsilon')
\end{array}
\quad
\Psi_n:\begin{array}[t]{clc}
1\text{-}\op{MHS}(\frakX_n'/\frakS_n) & \ra & \op{MHS}(\frakX_n/\frakS_n) \\
(\calE',\epsilon') & \mapsto & (F_n^*\calE',\varphi_n^*\epsilon').
\end{array}
\end{equation}
\end{parag}

\section[Full faithfulness]{Full faithfulness of $\Psi$}

\begin{lemma}\label{dirsum}
Let $u:M\ra N$ be a Kummer morphism of fs monoids. Then $\Z[N \backslash u(M)]$ is a $\Z[M]$-submodule of $\Z[N].$
\end{lemma}

\begin{proof}
We have to prove that for any $a\in M$ and $b\in N,$ if $u(a)+b\in u(M)$ then $b\in u(M).$ So let $a,t\in M$ and $b\in N$ such that $u(a)+b=u(t).$ Since $u$ is Kummer, there exists a positive integer $n$ and $x\in M$ such that $nb=u(x).$ Then $u(x)=u(n(t-a))$ and so $x=n(t-a)\in M.$ It follows that $t-a\in M^{sat}=M$ and so $b=u(t-a)\in u(M).$
\end{proof}

\begin{lemma}\label{dirsum2}
Let $u:M\ra N$ be a Kummer morphism of fs monoids. Then the morphism
$$v:\begin{array}[t]{clc}
N^{gp}\oplus N^{gp} & \ra & N^{gp} \oplus (N^{gp}/u^{gp}(M^{gp})) \\
(x,y) & \mapsto & (x+y,\ov{y})
\end{array}$$
induces an isomorphism
$(N\oplus_MN)^{sat}\xrightarrow{\sim}N\oplus (N^{gp}/u^{gp}(M^{gp})),$
whose inverse is induced by
$$w:\begin{array}[t]{clc}
N^{gp}\oplus N^{gp} & \ra & N^{gp}\oplus_{M^{gp}} N^{gp} \\
(x,y) & \mapsto & (x-y,y).
\end{array}$$
\end{lemma}

\begin{proof}
For any $x\in N^{gp}$ and $t\in M^{gp},$ $v(u^{gp}(t),-u^{gp}(t))=0$ and $w(x,u^{gp}(t))=(x,0)$ so $v$ and $w$ induce morphisms
$$v:\begin{array}[t]{clc}
N^{gp}\oplus_{M^{gp}} N^{gp} & \ra & N^{gp} \oplus (N^{gp}/u^{gp}(M^{gp})) \\
(x,y) & \mapsto & (x+y,\ov{y}).
\end{array}$$
$$w:\begin{array}[t]{clc}
N^{gp}\oplus (N^{gp}/u^{gp}(M^{gp})) & \ra & N^{gp}\oplus_{M^{gp}} N^{gp} \\
(x,y) & \mapsto & (x-y,y).
\end{array}$$
Now let $(x,y)\in N^{gp}\oplus N^{gp}.$ If $\ov{(x,y)}\in (N\oplus_MN)^{sat}\subset N^{gp}\oplus_{M^{gp}}N^{gp}$ then there exists a positive integer $n,$ $t\in M^{gp}$ and $z,s \in N$ such that
$n(x,y)=(z,s)+(u^{gp}(t),-u^{gp}(t)) \in N^{gp}\oplus N^{gp}.$
Then
$n(x+y)=z+s\in N.$
Since $N$ is saturated, $x+y\in N$ and so $v$ induces a morphism
$$v:\begin{array}[t]{clc}
(N\oplus_{M} N)^{sat} & \ra & N \oplus (N^{gp}/u^{gp}(M^{gp})).\end{array}$$
Let $x\in N$ and $y\in N^{gp}.$ There exists a positive integer $n$ such that $ny\in u^{gp}(M^{gp}).$ Then $n(x-y,y)=(nx,0)$ in $N^{gp}\oplus_{M^{gp}}N^{gp}$ and so $w$ induces a morphism
$$w:\begin{array}[t]{clc}
N\oplus (N^{gp}/u^{gp}(M^{gp})) & \ra & (N\oplus_{M} N)^{sat}.
\end{array}$$
The morphisms $v$ and $w$ are clearly inverse to each other.
\end{proof}

\begin{lemma}\label{exactseq}
Let $u:M\ra N$ be a Kummer morphism of fs monoids. Consider the homomorphism
$$d:\begin{array}[t]{clc}\Z [N] & \ra & \Z[N]\otimes_{\Z} \Z\left [ N^{gp}/u^{gp}(M^{gp})\right ] \\
e^n & \mapsto & e^n\otimes e^{\ov{n}}-e^n\otimes 1.\end{array}$$
Then the sequence of $\Z[M]$-modules
\begin{equation}
0 \ra \Z[M] \xrightarrow{\Z [u]} \Z[N] \xrightarrow{d} \Z[N]\otimes_{\Z} \Z\left [ N^{gp}/u^{gp}(M^{gp})\right ]
\end{equation}
is homotopic to zero i.e. there exist homomorphisms
$s:\Z[N] \ra \Z[M]$
and
$$s':\Z[N]\otimes_{\Z} \Z\left [ N^{gp}/u^{gp}(M^{gp})\right ] \ra \Z[N]$$
such that $s \circ \Z[u]=\op{Id}_{\Z[M]}$ and $s' \circ d+\Z[u] \circ s=\op{Id}_{\Z[N]}.$
\end{lemma}

\begin{proof}
By \ref{dirsum}, the $\Z[M]$-module $\Z[N]$ decomposes as
$\Z[N]=\Z[M]\oplus \Z[N\backslash M].$
Let
$s:\Z[N] \ra \Z[M]$
be the projection on $\Z[M]$ and let
$s':\Z[N]\otimes_{\Z} \Z\left [ N^{gp}/u^{gp}(M^{gp})\right ] \ra \Z[N]$
be the base change of
$$\Z\left [N^{gp}/u^{gp}(M^{gp}) \right ] \ra \Z,\ e^{\ov{n}} \mapsto \begin{cases}1\ \op{if}\ \ov{n}\neq 0\\ 0\ \op{else} \end{cases}$$
by $\Z \hookrightarrow \Z[N].$ 
Then
$s\circ \Z[u]= \op{Id}_{\Z[M]}$
and
$s' \circ d+d\circ s= \op{Id}_{\Z[N]}.$
\end{proof}

\begin{proposition}[\cite{Kat19} 3.4.1]\label{Katolftop}
Let $u:M\ra N$ be a Kummer morphism of fs monoids and $T$ an fs logarithmic scheme equipped with a chart $T\ra A[M]$ and $T'=T\times_{A[M]}A[N].$
Let $\calE$ be a quasi-coherent $\Ox_T$-module, $T''=T'\times_{T}^{\op{log}}T'$ and $p:T' \ra T$ and $q_1,q_2:T'' \ra T'$ the canonical projections. Then the sequence of $\Ox_{T}$-modules
$$0 \ra \calE \ra p_*p^*\calE \xrightarrow{p_*(q_2^{\#}-q_1^{\#})} (p\circ q_1)_*(p\circ q_1)^*\calE$$
is homotopic to zero.
\end{proposition}

\begin{proof}
We can suppose that $T$ is affine. Since $p$ is affine, $T'$ and $T''$ are both affine. Let $E$ be the module corresponding to $\calE.$ It is sufficient to prove that the sequence
$$0 \ra E \ra \Gamma(U,\calE)\otimes_{\Z[M]} \Z[N] \ra E\otimes_{\Z[M]}\Z[(N\oplus_MN)^{sat}]$$
is homotopic to $0.$ It is thus sufficient to prove that the sequence
$$0 \ra \Z[M] \xrightarrow{\Z[u]} \Z[N] \xrightarrow{\alpha} \Z[(N\oplus_MN)^{sat}]$$
is homotopic to zero, where
$
\alpha:\begin{array}[t]{clc}
\Z[N] & \ra & \Z[(N\oplus_MN)^{sat}] \\
e^n & \mapsto & e^{(0,n)}-e^{(n,0)}
\end{array}.$
By \ref{dirsum2}, there exists an isomorphism
$$\beta:\Z[(N\oplus_MN)^{sat}]  \xrightarrow{\sim} \Z [N\oplus N^{gp}/u^{gp}(M^{gp})] \xrightarrow{\sim} \Z[N] \otimes_{\Z} \Z[N^{gp}/u^{gp}(M^{gp})]$$
such that
$\beta\circ \alpha(e^n)=e^n\otimes e^{\ov{n}}-e^n\otimes 1 \ \forall n\in N.$
The result then follows from \ref{exactseq}.
\end{proof}

\begin{proposition}[\cite{INT} 1.3]\label{INTlog}
Let $f:X\ra S$ be a morphism of fs logarithmic schemes which is log flat, of Kummer type and locally of finite presentation, $x\in X$ and $y=f(x).$ Then, fppf locally around $x$ and $y,$ there exists a chart $\theta:P\ra Q$ of $f$ such that
\begin{enumerate}
\item $\theta$ is a Kummer morphism of fs monoids.
\item The morphism $f_1:X \ra T=S\times_{A[P]} A[Q],$ induced by $f$ and the chart $X \ra A[Q],$ is flat, surjective and locally of finite presentation.
\end{enumerate}
\end{proposition}

\begin{theorem}\label{logflatdescent}
Let $f:X \ra S$ be a quasi-compact surjective morphism of fs logarithmic schemes satisfying the following condition: 
For any $x\in X$ and $y=f(x),$ there exists, fppf locally around $x$ and $y,$ a chart $\theta:P\ra Q$ of $f$ satisfying:
\begin{enumerate}
\item $\theta$ is a Kummer morphism of fs monoids.
\item The morphism $f_1:X \ra T=S\times_{A[P]} A[Q],$ induced by $f$ and the chart $X \ra A[Q],$ is faithfully flat.
\end{enumerate}
Let $q_1,q_2:X\times_S^{\op{log}}X \ra X$ be the canonical projections and $g=f\circ q_1=f\circ q_2$ the structural morphism. For any quasi-coherent $\Ox_{S}$-module $\calE,$ the sequence
$$0 \ra \calE \ra f_*f^*\calE \rightarrow g_*g^*\calE,$$
where the second arrow is the difference of the morphisms induced by $q_1$ and $q_2,$ is exact.
\end{theorem}

\begin{proof}
We prove that for any $y\in S,$ the sequence
$$0 \ra \calE_y \ra (f_*f^*\calE)_y \ra (g_*g^*\calE)_y$$
is exact.
Let $y\in S.$ Since $f$ is surjective, there exists $x\in X$ such that $f(x)=y.$ By \ref{lemwiwfppf} we can suppose that $f$ has a chart $\theta:P\ra Q$ such that $\theta:P \ra Q$ is a Kummer morphism of fs monoids and the morphism $f_1:X \ra T=S\times_{A[P]} A[Q],$ induced by $f$ and the chart $X \ra A[Q],$ is faithfully flat and locally of finite presentation.
Let $f_2:T \ra S,$ $p_1,p_2:X\times_{T}^{\op{log}}X \ra X$ and $r_1,r_2:T \times_{S}^{\op{log}}T \ra T$ be the canonical projections and $g_1:X\times_T^{\op{log}}X\ra T$ and $g_2:T\times_S^{\op{log}}T \ra S$ the strutural morphisms. Note that $f_1:X\ra T$ is strict so $X\times_T^{\op{log}}X=X\times_TX.$
By \ref{Katolftop}, the sequence
\begin{equation}\label{exseq1}
0 \ra \calE \ra f_{2*}f_2^*\calE \ra g_{2*}g_2^*\calE,
\end{equation}
where the second arrow is the difference between the morphisms induced by $r_1$ and $r_2,$ is exact. The morphism $f_1$ is quasi-compact and faithfully flat, so, by (\cite{Raynaud71} Exposé VIII 1.7), the sequence
\begin{equation}\label{exseq2}
0 \ra f_2^*\calE \ra f_{1*}f^*\calE \ra g_{1*}g_1^*f_2^*\calE
\end{equation}
is exact.
The morphism $f_1\times_{S}f_1:X\times_{S}X \ra T\times_{S}T$ is strict and faithfully flat so $f_1\times_{S}^{\op{log}}f_1:X\times_{S}^{\op{log}}X \ra T\times_{S}^{\op{log}}T$ is also faithfully flat (see \ref{parag42} for the notation $\times_S^{\op{log}}$ and $\times_S$). It follows that the canonical morphism
\begin{equation}\label{exmor1}
g_2^*\calE \ra (f_1\times_S^{\op{log}}f_1)_*(f_1\times_S^{\op{log}}f_1)^*g_2^*\calE=(f_1\times_S^{\op{log}}f_1)_*g^*\calE
\end{equation}
is injective. Now considering the commutative diagram
$$
\begin{tikzcd}
X\times_T^{\op{log}}X \ar{r}{h} \ar{d}{g_1} & X\times_S^{\op{log}}X \ar{d}{g} \\
T\ar{r}{f_2} & S,
\end{tikzcd}
$$
we have a canonical morphism
\begin{equation}\label{exmor2}
g^*\calE \ra h_*h^*g^*\calE=h_*g_1^*f_2^*\calE.
\end{equation}
We have the following commutative diagram:
$$
\begin{tikzcd}
 & X\times_T^{\op{log}}X \ar{dl}{h} \ar{d} \ar[bend right =-60]{ddr}{g_1}  & \\
X\times_S^{\op{log}}X \ar{r} \ar{d} \ar{dr}{f_1\times_S^{\op{log}}f_1} & T\times_S^{\op{log}}X \ar{r} \ar{d} & X \ar[swap]{d}{f_1} \ar[bend right=-30]{dd}{f} \\
X\times_S^{\op{log}}T \ar{r} \ar{d} & T\times_S^{\op{log}}T \ar{r} \ar{d} \ar[swap]{dr}{g_2} & T\ar[swap]{d}{f_2} \\
X \ar[swap]{r}{f_1} & T \ar[swap]{r}{f_2} & S
\end{tikzcd}
$$
It proves that the exact sequences \eqref{exseq1} and \eqref{exseq2} and the morphisms \eqref{exmor1} and \eqref{exmor2} fit into the commutative diagram
$$
\begin{tikzcd}
 & & 0\ar{d} & 0\ar{d} \\
0 \ar{r} & \calE \ar{r} \ar{dr} & f_{2*}f_2^*\calE \ar{r} \ar{d} & g_{2*}g_2^*\calE \ar{d} \\
 & & f_*f^*\calE \ar{r} \ar{d} & g_*g^*\calE \ar{dl} \\
& & f_{2*}g_{1*}g_1^*f_2^*\calE &
\end{tikzcd}
$$
The assertion follows then from a simple diagram chase.
\end{proof}

\begin{lemma}\label{lemwiwfppf}
Keep the hypothesis of \ref{logflatdescent}. Let $(S_i \xrightarrow{u_i}S)_{i\in I}$ be an fppf covering and consider, for any $i\in I,$ the cartesian square
$$
\begin{tikzcd}
X_i \ar{r}{v_i} \ar{d}{f_i} & X \ar{d}{f} \\
S_i \ar{r}{u_i} & S.
\end{tikzcd}
$$
Let $g_i:X_i\times_{S_i}^{\op{log}}X_i \ra S_i$ be the structural morphism for any $i\in I.$
If, for all $i\in I,$ the sequence of $\Ox_{S_i}$-modules
$$0 \ra u_i^*\calE \ra f_{i*}f_i^*u_i^*\calE \ra g_{i*}g_i^*u_i^*\calE$$
is exact, then the sequence of $\Ox_S$-modules
$$0 \ra \calE \ra f_*f^*\calE \ra g_*g^*\calE$$
is exact.
\end{lemma}

\begin{proof}
First, note that the following squares are cartesian in the category of logarithmic schemes:
$$
\begin{tikzcd}
X_i \ar{r}{v_i} \ar{d}{f_i} & X \ar{d}{f} & X_i\times_{S_i}^{\op{log}}X_i \ar{r}{w_i} \ar{d}{g_i} & X\times_S^{\op{log}}X \ar{d}{g} \\
S_i \ar{r}{u_i} & S & S_i\ar{r}{u_i} & S
\end{tikzcd}
$$
The exact sequence
$$0 \ra u_i^*\calE \ra f_{i*}f_i^*u_i^*\calE \ra g_{i*}g_i^*u_i^*\calE$$
becomes equal to
$$0 \ra u_i^*\calE \ra f_{i*}v_i^*f^*\calE \ra g_{i*}w_i^*g^*\calE.$$
The morphism $u_i$ is flat so, by flat base change, we obtain the exact sequence
$$0 \ra u_i^*\calE \ra u_i^*f_*f^*\calE \ra u_i^*g_*g^*\calE.$$
The result then follows by faithfully flat descent.
\end{proof}

\begin{definition}[\cite{Kat19} 2.3]\label{logflattop}
Let $T$ be an fs logarithmic scheme. A family of morphisms $(T_i \xrightarrow{f_i}T)_{i\in I}$ of fs logarithmic schemes is said to be \emph{a covering for the log flat topology}, if the following conditions are satisfied:
\begin{enumerate}
\item $f_i$ is log flat, locally of finite presentation and of Kummer type for any $i\in I.$
\item Set theoretically,
$T=\bigcup_{i\in I}f_i(T_i).$
\end{enumerate}
This defines a pretopology on the category of fs logarithmic schemes.
\end{definition}

\begin{exmp}\label{Flogflatcover}
Let $f:X \ra S$ be a log smooth morphism of fine logarithmic schemes of characteristic $p,$ $F:X\ra X'$ the exact relative Frobenius and $g:X' \ra S$ the canonical projection (\ref{PFrob}). By \ref{FKummer}, $F$ is of Kummer type. By \ref{thmlogflat}, $F$ is log flat. Since $f=g\circ F$ and $f$ and $g$ are locally of finite presentation, $F$ is also locally of finite presentation. Finally, $F$ is a homeomorphism on the underlying topological spaces. We conclude that $(F:X\ra X')$ is a log flat covering.
\end{exmp}

\begin{theorem}\label{eraflogflatdescent}
Let $T$ be an fs logarithmic scheme, $(T_i \xrightarrow{f_i} T)_{i\in I}$ a log flat covering \eqref{logflattop} and $\calE$ a quasi-coherent $\Ox_T$-module. For all $i,j\in I,$ let $f_{ij}:T_i\times_T^{\op{log}}T_j \ra T$ be the canonical morphism. The sequence of $\Ox_T$-modules
$$0 \ra \calE \ra \prod_{i\in I}f_{i*}f_i^*\calE \ra \prod_{i,j\in I}f_{ij*}f_{ij}^*\calE,$$
where the last arrow is the difference between the morphisms induced by the projections $T_i\times_T^{\op{log}}T_j \ra T_i$ and $T_i\times_T^{\op{log}}T_j \ra T_j,$ is exact.
\end{theorem}

\begin{proof}
Let $X=\coprod_{i\in I}T_i$ and $f:X\ra T$ the morphism induced by $(f_i)_{i\in I}.$ Since the assertion is local on $T,$ and $f$ is open by (\cite{Kat19} 2.5), we can suppose that $I$ is finite. We then apply \ref{INTlog} and \ref{logflatdescent} to $f.$
\end{proof}

\begin{parag}
For the remaining of this section, we take the notation and assumption of \ref{paragJapon1} and \ref{parag86}. We consider a perfect field of positive characteristic $p,$ denote its ring of Witt vectors by $W,$ equip $\op{Spf}W$ with the trivial logarithmic structure and consider a log smooth morphism of framed logarithmic $p$-adic formal schemes $f:(\frakX,Q)\ra (\frakS,P),$ such that $\frakS$ is log flat and locally of finite type over $\op{Spf}W.$ Note that $f$ is log flat since it is log smooth (\cite{Ogus2018} IV 4.1.2) so, by \ref{Wflat}, the formal schemes $\frakX$ and $\frakS$ are flat over $\op{Spf}W$ \eqref{dxuflat}.
We denote by $F_1:X \ra X'$ the exact relative Frobenius and suppose that it lifts to a morphism of framed fs logarithmic $p$-adic formal schemes $F:(\frakX,Q) \ra (\frakX',Q')$ over $(\frakS,P),$ such that $\frakX'$ is log smooth over $\frakS$ and where $Q'$ is defined in \ref{PFrob}.
\end{parag}

\begin{proposition}\label{propFnlogtop}
For any positive integer $n,$ the morphism $F_n:\frakX_n \ra \frakX'_n$ forms a log flat covering \eqref{logflattop}.
\end{proposition}

\begin{proof}
The lifting $F_n$ is clearly a homeomorphism on the underlying topological spaces.
We have the commutative diagram
$$
\begin{tikzcd}
F_n^{-1}\ov{\calM}_{\frakX'_n} \ar{r}{\sim} \ar[swap]{d}{F_n^{\flat}} & F_1^{-1}\ov{\calM}_{X'} \ar{d}{F_1^{\flat}} \\
\ov{\calM}_{\frakX_n} \ar{r}{\sim} & \ov{\calM}_X.
\end{tikzcd}
$$
Since $F_1$ is of Kummer type \eqref{FKummer}, so is $F_n.$ The fact that $F_n$ is locally of finite presentation follows from the fact that $F_n$ is an $\frakS_n$-morphism and $\frakX_n$ and $\frakX'_n$ are locally of finite presentation over $\frakS_n.$ Finally, the log flatness of $F_n$ follows from \ref{logflatfiber}.
\end{proof}

\begin{theorem}\label{THMX1}
Let $n$ be a positive integer. Denote by $1\text{-}\op{MHS}^{\text{qcoh}}(\frakX_n'/\frakS_n)$ and $\op{MHS}^{\text{qcoh}}(\frakX_n/\frakS_n)$ the full subcategories of $1\text{-}\op{MHS}(\frakX_n'/\frakS_n)$ and $\op{MHS}(\frakX_n/\frakS_n)$ \eqref{defhpdstrat} consisting of quasi-coherent modules. Then, the functor
$$
\Psi_n:\begin{array}[t]{clc}
1\text{-}\op{MHS}^{\text{qcoh}}(\frakX_n'/\frakS_n) & \ra & \op{MHS}^{\text{qcoh}}(\frakX_n/\frakS_n)
\end{array}
$$
induced by \eqref{eq994}, is fully faithful.
\end{theorem}

\begin{proof}
We identify $\frakX_{n,\text{ét}}$ and $\frakX'_{n,\text{ét}}$ via the universal homeomorphism $F_n.$ Recall the diagram (\ref{totdiag1})
\begin{equation}
\begin{tikzcd}
 & & \frakX'\ar{d}{\iota'} \\
 & P_{\frakX/\frakS,0} \ar{r}{\varphi} \ar{d} & P_{\frakX'/\frakS,1}\ar{d} \\
\frakX\times_{\frakX'}^{\op{log}}\frakX \ar{ur}{\psi} \ar[bend right=-30]{uurr} \ar{r} & \frakX\times^{\op{log}}_{\frakS,[Q]}\frakX \ar{r}{F^2} & \frakX'\times_{\frakS,[Q']}^{\op{log}}\frakX'
\end{tikzcd}
\end{equation}
The functor $\Psi_n$ is then defined by
$$\Psi_n(\calE',\epsilon')=(F_n^*\calE',\varphi_n^*\epsilon')$$
for any object $(\calE',\epsilon')$ of $1\text{-}\op{MHS}^{\text{qcoh}}(\frakX_n'/\frakS_n).$

Let $(\calE'_1,\epsilon'_1)$ and $(\calE'_2,\epsilon'_2)$ be two objects of $1\text{-}\op{MHS}^{\text{qcoh}}(\frakX_n'/\frakS_n),$ $(\calE_1,\epsilon_1)$ and $(\calE_2,\epsilon_2)$ their images by $\Psi_n$ and
$u:(\calE_1,\epsilon_1) \ra (\calE_2,\epsilon_2)$
a morphism of $\op{MHS}^{\text{qcoh}}(\frakX_n/\frakS_n).$
Let $p_1,p_2:\left (P_{\frakX/\frakS,0}\right )_n \ra \frakX_n$ and $q_1,q_2:\frakX_n \times_{\frakX_n'}^{\op{log}} \frakX_n \ra \frakX_n$ be the canonical projections.
Since $u$ is a morphism of stratified modules, the diagram on the left is commutative and that implies the commutativity of the diagram on the right:
$$
\begin{tikzcd}
p_2^*\calE_1 \ar{r}{\epsilon_1} \ar[swap]{d}{p_2^*u} & p_1^*\calE_1 \ar{d}{p_1^*u} & & q_2^*\calE_1 \ar[swap]{d}{q_2^*u} \ar{r}{\psi_n^*\epsilon_1} & q_1^*\calE_1 \ar{d}{q_1^*u} \\
p_2^*\calE_2 \ar{r}{\epsilon_2} & p_1^*\calE_2 & & q_2^*\calE_2 \ar{r}{\psi_n^*\epsilon_2} & q_1^*\calE_2
\end{tikzcd}
$$
We deduce the commutativity of the following diagram (without the dotted arrow)
$$
\begin{tikzcd}
0 \ar{r} & \calE_1' \ar{r} \ar[dashed]{d}{u'} & \calE_1  \ar{r} \ar{d}{u} & \calE'_1\otimes_{\Ox_{\frakX_n'}}\Ox_{\frakX_n \times_{\frakX_n'}^{\op{log}}\frakX_n} \ar{d}{q_2^*u-q_1^*u} \\
0 \ar{r} & \calE_2' \ar{r} & \calE_2 \ar{r} & \calE'_2\otimes_{\Ox_{\frakX_n'}}\Ox_{\frakX_n \times_{\frakX_n'}^{\op{log}}\frakX_n},
\end{tikzcd}
$$
where the rows are given in \ref{eraflogflatdescent}. These rows are exact by \ref{eraflogflatdescent} and \ref{propFnlogtop}. We deduce the existence of the morphism $u':\calE'_1 \ra \calE'_2.$ It remains to prove that $u'$ is compatible with the stratifications $\epsilon_1'$ and $\epsilon_2'.$ Since $\varphi$ is log flat \eqref{Khaminei145}, it is sufficient, by \ref{eraflogflatdescent}, to prove that $F^*u'=u$ is compatible with $\varphi^*\epsilon'_i=\epsilon_i.$ This is true by definition of $u$ and the full faithfulness of $\Psi_n$ follows.
\end{proof}

\begin{theorem}\label{THMX2}
Let $n$ be a positive integer. The functor $\Phi_n$ \eqref{krazphin} preserves quasi-nilpotent connections, the diagram
$$
\begin{tikzcd}
1\text{-}\op{MHS}^{qcoh}(\frakX_n'/\frakS_n) \ar{rr}{\Psi_n} \ar[swap,sloped]{d}{\sim} & & \op{MHS}^{qcoh}(\frakX_n/\frakS_n) \ar[sloped]{d}{\sim} \\
p\text{-}\op{MIC}^{qcoh,qn}(\frakX_n'/\frakS_n) \ar{rr}{\Phi_n} & & \op{MIC}^{qcoh,qn}(\frakX_n/\frakS_n),
\end{tikzcd}
$$
where qn denotes quasi-nilpotent objects \eqref{krazphin}, is commutative and $\Phi_n$ is fully faithful.
\end{theorem}

\begin{proof}
It is sufficient to prove that the diagram
$$
\begin{tikzcd}
1\text{-}\op{MHS}^{qcoh}(\frakX_n'/\frakS_n) \ar{rr}{\Psi_n} \ar[hook]{d} & & \op{MHS}^{qcoh}(\frakX_n/\frakS_n) \ar[hook]{d} \\
p\text{-}\op{MIC}^{qcoh}(\frakX_n'/\frakS_n) \ar{rr}{\Phi_n} & & \op{MIC}^{qcoh}(\frakX_n/\frakS_n)
\end{tikzcd}
$$
is commutative.
Let $(\calE',\epsilon')$ be an object of the upper left category. Its image by $\Psi_n$ is
$$
(\calE,\epsilon)=\Psi_n(\calE',\epsilon')=(F_n^*\calE',\varphi_n^*\epsilon').
$$
Suppose the hypothesis \ref{loccoord} is satisfied. By \ref{lem12}, for every $1\le i\le d,$ there exists a local invertible section $u_i$ of $\calM_{\frakX}$ and a local section $b_i$ of $\Ox_{\frakX}$ such that $F^{\flat}(\widetilde{m}_i')=p\widetilde{m}_i+u_i$ and $\alpha_{\frakX}(u_i)=1+pb_i.$
For $1\le i\le d,$ by \eqref{eqKoko1323}, we have
$$
\varphi^{\#}\left (\frac{\widetilde{\eta}_i'}{p} \right )=\left ((p-1)!\widetilde{\eta}_i^{[p]}+\sum_{k=1}^{p-1}\frac{(p-1)!}{k!(p-k)!}\widetilde{\eta}_i^k \right )\alpha_{\frakY}(p_2^{\flat}u_i-p_1^{\flat}u_i)+\frac{p_2^{\#}b_i-p_1^{\#}b_i}{1+pp_1^{\#}b_i},
$$
where $p_1,p_2:\frakY \ra \frakX$ are the canonical projections.
Denote by $\widehat{\eta}_i$ (resp. $\widehat{\eta}_i'$) the class of $\widetilde{\eta}_i$ (resp. $\frac{\widetilde{\eta}_i'}{p}$) modulo $p^n,$ by $\calP'$ the structural ring of $P_{\frakX'/\frakS,1},$ by $\calP'_n$ its reduction modulo $p^n$ and by $\calI'_n$ its PD ideal. Similarly, let $\calI$ be the PD ideal of $P_{\frakX/\frakS,0}$ and $\calI_n$ its reduction modulo $p^n.$ Denote by $\calP$ the structure sheaf of $P_{\frakX/\frakS,0}$ and by $\calP_n$ its reduction modulo $p^n.$ Then $\left (\widehat{\eta}'^{[I]} \right )_{I\in \N^d}$ is a basis of the $\Ox_{\frakX'_n}$-module $\calP'_n.$ Denote by $(\partial'_I)$ its dual basis. It follows that, for a local section $x'$ of $\calE',$
$$
\epsilon':\calP'_n \otimes_{\Ox_{\frakX'_n}}\calE' \ra \calE' \otimes_{\Ox_{\frakX'_n}}\calP'_n,\ 1\otimes x'\mapsto \sum_{I\in\N^d}(\partial_I'\cdot x')\otimes \widehat{\eta}'^{[I]},
$$
$$
\varphi_n^*\epsilon': \calP_n \otimes_{\Ox_{\frakX_n}}(F_n^*\calE') \ra (F_n^*\calE')\otimes_{\Ox_{\frakX_n}} \calP_n,\ 1\otimes 1\otimes x' \mapsto \sum_{I \in \N^d}1\otimes \left (\partial'_I\cdot x' \right )\otimes \varphi_n^{\#}(\widehat{\eta}'^{[I]}).
$$
Composing $\varphi_n^*\epsilon'$ with the canonical projection
$$
\pi:\left (F_n^*\calE'\right ) \otimes _{\Ox_{\frakX_n}} \calP_n \ra \left (F_n^*\calE'\right ) \otimes_{\Ox_{\frakX_n}} \calP_n/\calI_n^{[2]},
$$
and since the image of $\alpha_{\frakY}(p_2^{\flat}u_i-p_1^{\flat}u_i)-1$ in $\calP$ belongs to $\calI,$ we get
$$
\left (\pi\circ \varphi_n^*\epsilon'\right ) \left (1\otimes 1\otimes x' \right ) =  1\otimes x'\otimes 1+\sum_{i=1}^d1\otimes \left (\partial_{\epsilon_i}'\cdot x' \right )\otimes \left (\widehat{\eta}_i+\frac{p_{2}^{\#}b_i-p_1^{\#}b_i}{1+pp_1^{\#}b_i} \right ).
$$
Identifying $\calI_n/\calI_n^{[2]}$ with $\omega^1_{\frakX_n/\frakS_n}$ via the canonical isomorphism \eqref{eqtakrizIIsquare}, we get 
$$
\left (\pi\circ \varphi_n^*\epsilon'\right ) \left (1\otimes 1\otimes x' \right ) -  1\otimes x'\otimes 1 = \sum_{i=1}^d1\otimes \left (\partial_{\epsilon_i}'\cdot x' \right )\otimes \left (\op{dlog}m_i+\frac{db_i}{1+pb_i} \right ).
$$
Let $\nabla'$ be the $p$-connection corresponding to $\epsilon'.$ Denote by $\calI'$ be the PD-ideal of $\calP_{\frakX'/\frakS,1}$ and $\calI'_n$ its reduction modulo $p^n.$ Then, by \eqref{KhamineiBeta}, we have
$$
\nabla':\begin{array}[t]{clclc}
\calE' & \ra & \calE'\otimes_{\Ox_{\frakX_n}} \left (\calI'_n/\calI_n'^{[2]} \right ) & \xrightarrow{\sim} & \calE' \otimes_{\Ox_{\frakX_n}} \omega^1_{\frakX_n'/\frakS_n} \\
x' & \mapsto & \sum_{i=1}^d(\partial_{\epsilon_i}'\cdot x') \otimes \ov{\widehat{\eta}_i'} & \mapsto & \sum_{i=1}^d(\partial_{\epsilon_i}'\cdot x') \otimes \op{dlog}m_i'.
\end{array}
$$
Set
$$
(\calE,\nabla)=\Phi_n(\calE',\nabla').
$$
If $x'$ is a local section of $\calE',$ then, by definition of $\Phi_n$ and \eqref{surp2},
$$
\nabla(1\otimes x')=\sum_{i=1}^d1\otimes (\partial_{\epsilon_i}'\cdot x') \otimes \left (\op{dlog}m_i +\frac{db_i}{1+pb_i} \right ).
$$
\end{proof}

\appendix

\section{Appendix: \texorpdfstring{$p^n$}{p} %
     -connections and stratifications}

The goal of this appendix is to prove proposition \ref{equivstrat}. We keep the notations of section 11 and we start by proving some lemmas.

\begin{lemma}\label{lem116f}
Let $k\ge 1$ and $n\ge 0$ be integers and suppose \ref{loccoord} is satisfied. For any $1\le i\le d,$ let $\xi_i$ be the image of $\frac{\widetilde{\eta}_i}{p^n} \in \calP_{\frakX / \frakS , n}$ in $\calP_{n,k}:=\calP_{\frakX/\frakS,n}/(p^k)$ and set, for any $I=(I_1,\hdots,I_d)\in \N^d,$
$$\xi^{[I]}=\prod_{i=1}^d\xi_i^{[I_i]}.$$
Let $\delta:\calP_{n,k} \ra \calP_{n,k}\otimes_{\Ox_{\frakX_k}}\calP_{n,k}$ the comultiplication map of the Hopf algebra $\calP_{n,k}.$ Then
$$\delta\left (\xi^{[I]}\right )=\sum_{\substack{a,b,c\in \N^d \\ a+b+c=I}} p^{n|c|} c!\begin{pmatrix}a+c \\ c \end{pmatrix} \begin{pmatrix}b+c \\ c \end{pmatrix}\xi^{[b+c]}\otimes \xi^{[a+c]}.$$
\end{lemma}

\begin{proof}
Let $\widetilde{\delta} : \calP_{\frakX / \frakS,n} \ra \calP_{\frakX / \frakS,n} \otimes_{\Ox_{\frakX}} \calP_{\frakX / \frakS,n}$ be the comultiplication map. By \ref{HopffrakP}, we have
\begin{alignat*}{2}
p^n\widetilde{\delta}\left ( \frac{\widetilde{\eta}_i}{p^n} \right ) = \widetilde{\delta} \left ( \widetilde{\eta}_i \right ) &= 1 \otimes \widetilde{\eta}_i + \widetilde{\eta}_i \otimes 1 + \widetilde{\eta}_i \otimes \widetilde{\eta}_i \\
&= p^n \left ( 1 \otimes \frac{\widetilde{\eta}_i}{p^n} + \frac{\widetilde{\eta}_i}{p^n} \otimes 1 + p^n \frac{\widetilde{\eta}_i}{p^n} \otimes \frac{\widetilde{\eta}_i}{p^n} \right ).
\end{alignat*}
By the flatness of $P_{\frakX / \frakS,n}$ over $\op{Spf} \Z_p$ \eqref{propflat}, we deduce that
$$
\widetilde{\delta}\left ( \frac{\widetilde{\eta}_i}{p^n} \right ) = 1 \otimes \frac{\widetilde{\eta}_i}{p^n} + \frac{\widetilde{\eta}_i}{p^n} \otimes 1 + p^n \frac{\widetilde{\eta}_i}{p^n} \otimes \frac{\widetilde{\eta}_i}{p^n}.
$$
It follows that
$$\delta(\xi_i)=1\otimes \xi_i+\xi_i\otimes 1+p^n\xi_i\otimes \xi_i.$$
Then, for $n\in \N,$
\begin{alignat*}{2}
\delta(\xi_i^{[n]}) &=  \left (1\otimes \xi_i+\xi_i\otimes 1+p^n\xi_i\otimes \xi_i\right )^{[n]} \\
&= \sum_{\substack{a,b,c\in \N \\a+b+c=n}}(1\otimes \xi_i)^{[a]}(\xi_i \otimes 1)^{[b]}(p^n\xi_i\otimes \xi_i)^{[c]} \\
&= \sum_{\substack{a,b,c\in \N \\a+b+c=n}}p^{nc}c!(\xi_i^{[b]}\xi_i^{[c]})\otimes (\xi_i^{[a]}\xi_i^{[c]}) \\
&= \sum_{\substack{a,b,c\in \N \\a+b+c=n}}p^{nc}c!\begin{pmatrix}a+c\\c \end{pmatrix}\begin{pmatrix}b+c \\ c \end{pmatrix} \xi_i^{[b+c]}\otimes \xi_i^{[a+c]}.
\end{alignat*}
The same holds for multi-indices hence the result.
\end{proof}

\begin{proof}[Proof of \ref{equivstrat}]
The proof is similar to \ref{prop39} so we just prove how (3) implies (2) and give a sketch of the other equivalences.
For simplicity, we suppose $k=1$ and we drop the subscript $k$ from our notation.

Suppose we are given the data (1). The data (2) is then obtained by setting
$$
\theta_l:\begin{array}[t]{clclc}
\calE & \ra & \calP_n^{\{l\}}\otimes_{\Ox_X} \calE & \xrightarrow{\epsilon_l} & \calE \otimes_{\Ox_X}\calP_n^{\{l\}} \\
x & \mapsto & 1\otimes x & \mapsto & \epsilon_l(1\otimes x).
\end{array}
$$
Conversly, if (2) is given then $\epsilon_l$ is the $\calP_n^{\{l\}}$-linearization of $\theta_l.$

Suppose we are given the data (2). The data (3) is obtained as follows: let
$$
\nabla:\begin{array}[t]{clc}
\calE & \ra & \calE\otimes_{\Ox_X}\calP_n^{\{1\}} \\
x & \mapsto & \theta_1( x)-x\otimes 1.
\end{array}
$$
Since $\theta_0=\op{Id}_{\calE},$ we get
$$\nabla(\calE) \subset \calE \otimes_{\Ox_X}\ov{\calI}_n^{\{1\}}.$$
In addition, if we denote by $p_1,p_2:P_{n} \ra X$ the canonical projections, then, for any local sections $a$ and $x$ of $\Ox_X$ and $\calE$ respectively, we have
\begin{alignat*}{2}
\nabla(ax) &= \theta_1(ax)-(ax)\otimes 1 = p_2^{\#}(a)\theta_1(x)-p_1^{\#}(a) (x\otimes 1) \\
&= p_2^{\#}(a) \left ( \theta_1(x)-x\otimes 1 \right )+ x\otimes \left ( p_2^{\#}(a) - p_1^{\#}(a) \right ) = a\nabla(x)+x\otimes d'(a).
\end{alignat*}
By \ref{propfinalMuzan1}, $\nabla$ corresponds to a $p^n$-connection on $\calE.$ The integrability follows from the commutativity of \eqref{diagstrconn1}.
Suppose now that an integrable $p^n$-connection $\calE \ra \calE \otimes_{\Ox_X}\omega^1_{X/S}$ is given. Let
$$\nabla:\calE \ra \calE \otimes_{\Ox_X}\ov{\calI}_{n}^{\{1\}}$$
be the corresponding connection given by \ref{propfinalMuzan1}. We construct the morphisms $\theta_l$ of (2) étale locally on $X.$ Suppose that the hypothesis \ref{loccoord} is satisfied and consider the notation introduced in \ref{lem116f}. We have a canonical isomorphism
$$\calP_{n}\xrightarrow{\sim} \Ox_X\langle \xi_1,\hdots,\xi_d \rangle.$$
Denote by $\left ( \partial_I \right )_{I\in \N^d}$ the dual basis of $\left ( \xi^{[I]} \right )_{I\in \N^d}.$ Let $\epsilon_1$ be the $\calP_n$-linear morphism defined by
$$
\epsilon_1:\begin{array}[t]{clc}
\calP_n^{\{1\}} \otimes_{\Ox_X}\calE & \ra & \calE\otimes_{\Ox_X}\calP_n^{\{1\}} \\
1\otimes x & \mapsto & \nabla(x)+x\otimes 1.
\end{array}
$$
We define an $\Ox_X$-linear morphisms
$$
\nabla_l:\mathscr{Hom}_{\Ox_X} \left ( \calP_n^{ \{l\} },\Ox_X \right ) \ra \mathscr{Hom}_{\Ox_X} \left ( \calP_n^{ \{l\} }\otimes_{\Ox_X} \calE,\calE \right ).
$$
For that, it is sufficient to define $\nabla_l$ on the basis $\left (\partial_I \right )_{|I| \le l}.$
For $l,l'\in \N$ and $I\in \N^d$ such that $l\ge l' \ge |I|,$ we consider a linear morphism of
$\mathscr{Hom}_{\Ox_X} \left ( \calP_n^{ \{l'\} }\otimes_{\Ox_X} \calE,\calE \right )$
as a linear morphism of
$\mathscr{Hom}_{\Ox_X} \left ( \calP_n^{ \{l\} }\otimes_{\Ox_X} \calE,\calE \right )$
via the canonical projection $\calP_n^{ \{ l \} } \ra \calP_n^{ \{ l' \} }.$ 
For any $\Ox_X$-linear morphism $f:\calP_n^{\{1\}}\ra \Ox_X,$ let $\nabla_1(f)=(\op{Id}_{\calE}\otimes f)\circ \epsilon_1.$
The integrability of $\nabla$ implies that the differential operators $\nabla_1(\partial_{\epsilon_i})$ pairwise commute and we can define $\nabla_{l}(\partial_N),$ for any multi-index $N=(n_1,\hdots,n_d)\in\mathbb{N}^d$ of length $|N|\le l,$ by
$$
\nabla_l(\partial_N)=\prod_{i=1}^d\prod_{j=0}^{n_i-1}(\nabla_1(\partial_{\epsilon_i})-p^nj),
$$
where $p^nj$ denotes $p^nj \op{Id}_{\calE},$ considered as a differential operator of order 1.
Then
$$
\nabla_l\left (\partial_I \right ) = \nabla_{l'}\left ( \partial_I \right ).
$$
We can hence drop the subscript $l$ in the rest of this proof. We can also show that, for $I,J\in \N^d,$
$
\nabla \left ( \partial_I \circ \partial_J \right ) = \nabla \left ( \partial_I \right ) \circ \nabla \left ( \partial_J \right ).
$
Set
$
\theta_l:\begin{array}[t]{clc}
\calE & \ra & \calE \otimes_{\Ox_X}\calP_n^{\{l\}}. \\
x & \mapsto & \sum_{|I| \le l}\nabla(\partial_I)(1\otimes x) \otimes \xi^{[I]}
\end{array}
$
We check the commutativity of the diagram \eqref{diagstrconn1}. On one hand, by \ref{lem116f}, we have
\begin{equation}\label{eqfinalMuzan5}
\begin{alignedat}{2}
(\op{Id}_{\calE}\otimes \delta^{l,l'})\circ \theta_{l+l'}(x) &= (\op{Id}_{\calE}\otimes \delta^{l,l'})\left (\sum_{|I|\le l+l'} \nabla(\partial_I)(1\otimes x)\otimes \xi^{[I]} \right ) \\
&= \sum_{|I|\le l+l'}\sum_{\substack{a+b+c=I \\ |b+c|\le l \\ |a+c|\le l'}} p^{n|c|} c!\begin{pmatrix}a+c \\ c \end{pmatrix} \begin{pmatrix}b+c \\ c \end{pmatrix}\nabla(\partial_I)(1\otimes x) \otimes \xi^{[b+c]}\otimes \xi^{[a+c]}.
\end{alignedat}
\end{equation}
On the other hand, we have
\begin{equation}\label{eqfinalMuzan6}
\begin{alignedat}{2}
\theta_{l'}(x) &= \sum_{|I|\le l'} \nabla(\partial_I)(1\otimes x)\otimes \xi^{[I]}. \\
(\theta_l \otimes \op{Id})\circ \theta_{l'}(x) &= \sum_{\substack{|I|\le l' \\ |J|\le l}} \nabla(\partial_J)(1\otimes (\nabla(\partial_I)(1\otimes x)))\otimes \xi^{[J]}\otimes \xi^{[I]} \\
&= \sum_{\substack{|I|\le l' \\ |J|\le l}} \nabla(\partial_J) \circ \nabla(\partial_I)(1\otimes x)\otimes \xi^{[J]}\otimes \xi^{[I]} \\
&= \sum_{\substack{|I|\le l' \\ |J|\le l}} \nabla(\partial_J \circ \partial_I)(1\otimes x)\otimes \xi^{[J]}\otimes \xi^{[I]}.
\end{alignedat}
\end{equation}
We have to prove that \eqref{eqfinalMuzan5} and \eqref{eqfinalMuzan6} are equal. For that, we prove that, for $I,J \in \N^d$ such that $|J| \le l$ and $|I| \le l',$
$$
\partial_J \circ \partial_I = \sum_{\substack{a+c=I \\ b+c=J}} p^{n|c|}c!\begin{pmatrix}a+c \\ c \end{pmatrix} \begin{pmatrix}b+c \\ c \end{pmatrix}\partial_{a+b+c}.
$$
We proceed by induction on $|J|.$ For $|J|=0,$ the equality is clear.
Fix an integer $m$ and suppose the result is true for all $J$ and $I$ such that $|J|=m.$ Let $1\le i\le d.$
Then
\begin{alignat*}{2}
\partial_{J+\epsilon_i}\circ \partial_I &= (\partial_{\epsilon_i}\circ \partial_J -p^nJ_i\partial_J)\circ \partial_I \\
&= \sum_{\substack{a+c=I\\ b+c=J}} p^{n|c|}c! \begin{pmatrix}a+c \\ c \end{pmatrix} \begin{pmatrix}b+c \\ c \end{pmatrix} \left (\partial_{\epsilon_i} \circ \partial_{a+b+c}-p^n J_i \partial_{a+b+c} \right ) \\
&=  \sum_{\substack{a+c=I\\ b+c=J}} p^{n|c|}c! \begin{pmatrix}a+c \\ c \end{pmatrix} \begin{pmatrix}b+c \\ c \end{pmatrix} \left (\partial_{a+b+c+\epsilon_i} +p^n(a_i+b_i+c_i) \partial_{a+b+c}-p^n J_i \partial_{a+b+c} \right ) \\
&= \sum_{\substack{a+c=I\\ b+c=J}} p^{n|c|}c! \begin{pmatrix}a+c \\ c \end{pmatrix} \begin{pmatrix}b+c \\ c \end{pmatrix} \left (\partial_{a+b+c+\epsilon_i}+p^n a_i \partial_{a+b+c} \right ).
\end{alignat*}
Applying this to $\xi^{[K]}$ for $K\in \N^d,$ and considering the convention $\begin{pmatrix} \alpha \\ \beta \end{pmatrix}=0$ if $\beta<0,$ we get
\begin{align}\label{yeaa1}
\begin{split}
p^{n|I+J-K|+n}(I+J-K+\epsilon_i)! \begin{pmatrix}I\\K-J-\epsilon_i \end{pmatrix} \begin{pmatrix}J\\ K-I-\epsilon_i \end{pmatrix}\\+p^{n|I+J-K|+n}(I+J-K)!\begin{pmatrix}I\\K-J\end{pmatrix} \begin{pmatrix}J\\ K-I \end{pmatrix}(K_i-J_i).
\end{split}
\end{align}
Applying
$$
\sum_{\substack{a+c=I \\ b+c=J+\epsilon_i}} p^{n|c|}c!\begin{pmatrix}a+c \\ c \end{pmatrix} \begin{pmatrix}b+c \\ c \end{pmatrix}\partial_{a+b+c}
$$
to $\xi^{[K]},$ we get
\begin{equation}\label{yeaa2}
p^{n|I+J-K|+n}(I+J-K+\epsilon_i)!\begin{pmatrix}I \\ K-J-\epsilon_i \end{pmatrix} \begin{pmatrix}J+\epsilon_i \\ K-I \end{pmatrix}.
\end{equation}
To show that \eqref{yeaa1} and \eqref{yeaa2} are equal, recalling the notation \ref{Not9}, we have to prove the equality
\begin{alignat*}{2}
&(I_i+J_i-K_i+1) \begin{pmatrix}I_i\\K_i-J_i-1 \end{pmatrix} \begin{pmatrix}J_i\\ K_i-I_i-1 \end{pmatrix}+\begin{pmatrix}I_i\\K_i-J_i\end{pmatrix} \begin{pmatrix}J_i\\ K_i-I_i \end{pmatrix}(K_i-J_i) \\
=& (I_i+J_i-K_i+1)\begin{pmatrix}I_i \\ K_i-J_i-1 \end{pmatrix} \begin{pmatrix}J_i+1 \\ K_i-I_i \end{pmatrix}.
\end{alignat*}
This is equivalent to
\begin{alignat*}{2}
&\begin{pmatrix}I_i\\K_i-J_i\end{pmatrix} \begin{pmatrix}J_i\\ K_i-I_i \end{pmatrix}(K_i-J_i) \\
=& (I_i+J_i-K_i+1)\begin{pmatrix}I_i \\ K_i-J_i-1 \end{pmatrix} \left ( \begin{pmatrix}J_i+1 \\ K_i-I_i \end{pmatrix} - \begin{pmatrix}J_i\\ K_i-I_i-1 \end{pmatrix} \right ).
\end{alignat*}
But
$$
\begin{pmatrix}J_i+1 \\ K_i-I_i \end{pmatrix} - \begin{pmatrix}J_i\\ K_i-I_i-1 \end{pmatrix}= \begin{pmatrix}J_i \\ K_i-I_i \end{pmatrix}.
$$
So we have to prove that
\begin{equation*}
\begin{pmatrix}I_i\\K_i-J_i\end{pmatrix} (K_i-J_i) = (I_i+J_i-K_i+1)\begin{pmatrix}I_i \\ K_i-J_i-1 \end{pmatrix}.
\end{equation*}
This is immediate.
\end{proof}

\end{document}